\documentclass{memo-l}

\usepackage{amssymb}
\usepackage{graphicx}
\usepackage{amsmath}
\usepackage{amsfonts}

\theoremstyle{definition}

\theoremstyle{remark}

\numberwithin{section}{chapter}
\numberwithin{equation}{chapter}
\numberwithin{figure}{chapter}

\newtheorem{thm}{Theorem}[section]

\newtheorem{cor}[thm]{Corollary}
\newtheorem{prob}[thm]{Problem}
\newtheorem{lem}[thm]{Lemma}
\newtheorem{prop}[thm]{Proposition}

\theoremstyle{definition}
\newtheorem{defn}[thm]{Definition}

\newtheorem{rem}[thm]{Remark}
\newcommand{\R}{\mathbb{R}}

\newcommand{\A}{\mathcal{A}}

\newcommand{\C}{\mathbb{C}}

\newcommand{\RR}{\mathfrak{R}}
\newcommand{\uc}{\mathbb{S}^1}
\newcommand{\set}[1]{\left\{#1\right\}}
\newcommand{\win}{\mathrm{win}}
\newcommand{\0}{\emptyset}
\newcommand{\val}{\mathrm{val}}

\newcommand{\Int}{\mathrm{int}}
\newcommand{\dia}{\mathrm{diam}}
\newcommand{\sh}{\mathrm{Sh}}
\newcommand{\ch}{\mathrm{Ch}}
\newcommand{\iu}{U^\iy}
\newcommand{\e}{\varepsilon}
\newcommand{\al}{\alpha}
\newcommand{\ba}{\beta}

\newcommand{\be}{\beta}
\newcommand{\ga}{\gamma}
\newcommand{\Ga}{\Gamma}
\newcommand{\si}{\sigma}

\newcommand{\da}{\delta}
\newcommand{\vp}{\varphi}

\newcommand{\nin}{\not\in}
\newcommand{\imp}{\mathrm{Imp}}

\newcommand{\re}{\mathrm{Re}}

\newcommand{\disk}{\mathbb{D}}

\newcommand{\degree}{\text{degree}}
\newcommand{\dg}{\mathrm{degree}}

\newcommand{\iy}{\infty}
\newcommand{\g}{\mathbf{g}}
\newcommand{\lam}{\mathcal{L}}
\newcommand{\ECH}{\mathrm{\mathop{conv_{\mathcal{E}}}}}
\newcommand{\HCH}{\mathrm{\mathop{conv_{\mathcal H}}}}

\newcommand{\mc}{\mathcal}
\newcommand{\mrm}{\mathrm}

\newcommand{\ol}{\overline}
\newcommand{\cl}{\overline}

\newcommand{\KP}{\mathcal{KP}}
\newcommand{\kp}{\mathcal{KP}}
\newcommand{\KPP}{\mathcal{KPP}}

\newcommand{\complex}{\mathbb{C}}
\newcommand{\Complex}{\mathbb{C}}
\newcommand{\rsphere}{\mathbb{C}^\infty}
\newcommand{\sphere}{\Complex^\infty}
\newcommand{\ucirc}{\mathbb{S}^1}

\newcommand{\B}{\mathfrak{B}}

\newcommand{\real}{\mathbb{R}}
\newcommand{\reals}{\mathbb{R}}

\newcommand{\zed}{\mathbb{Z}}
\newcommand{\Disk}{\mathbb{D}}
\newcommand{\wh}{\widehat}

\newcommand{\var}{\mathrm{var}}
\newcommand{\ind}{\mathrm{ind}}
\newcommand{\bd}{\partial}
\newcommand{\im}{\mathrm{Im}}
\newcommand{\pr}{\mathrm{Pr}}
\newcommand{\dm}{\mathrm{diam}}

\newcommand{\sm}{\setminus}
\newcommand{\hX}{\hat{X}}
\newcommand{\Sh}{\mathrm{Sh}}
\newcommand{\tf}{\tilde{f}}
\newcommand{\tU}{U^\infty}
\newcommand{\rg}{\mathrm{\mathbf{g}}}
\newcommand{\rh}{\mathrm{\mathbf{h}}}
\newcommand{\fg}{\mathfrak{g}}

\newcommand{\rG}{\mathrm{G}}
\newcommand{\rH}{\mathrm{H}}
\newcommand{\hx}{\hat{x}}
\newcommand{\fS}{\mathfrak{S}}
\newcommand{\tb}{\tilde{b}}
\newcommand{\tB}{\tilde{B}}
\newcommand{\Le}{\mathrm{L}}

\begin{document}

\frontmatter

\title{FIXED POINT THEOREMS FOR PLANE CONTINUA WITH APPLICATIONS}


\author[Blokh]{Alexander~M.~Blokh}
\address[Alexander~M.~Blokh]{Department of Mathematics\\ University of Alabama at Birmingham\\
Birmingham, AL 35294-1170}
\email[Alexander~M.~Blokh]{ablokh@math.uab.edu}
\thanks{The first named author was partially
supported by grant NSF-DMS-0901038}

\author[Fokkink]{Robbert~J.~Fokkink }
\address[Robbert~J.~Fokkink]{Delft University,
Faculty of Information and Systems, P.O. Box 5031, 2600 GA Delft,
Netherlands}
\email[Robbert~J.~Fokkink]{R.J.Fokkink@its.tudelft.nl}

\author[Mayer]{John~C.~Mayer}
\address[John~C.~Mayer]
{Department of Mathematics\\ University of Alabama at Birmingham\\
Birmingham, AL 35294-1170}
\email[John~C.~Mayer]{mayer@math.uab.edu}

\author[Oversteegen]{Lex~G.~Oversteegen}
\address[Lex~G.~Oversteegen]
{Department of Mathematics\\ University of Alabama at Birmingham\\
Birmingham, AL 35294-1170}
\email[Lex~G.~Oversteegen]{overstee@math.uab.edu}
\thanks{The fourth named author was supported in part by grant
NSF-DMS-0906316}

\author[Tymchatyn]{E.~D.~Tymchatyn}
\address[E.~D.~Tymchatyn]{Department of Mathematics and Statistics\\
University of Saskatchewan\\
Saskatoon, Saskatchewan, Canada S7N 0W0}
\email[E.~D.~Tymchatyn]{tymchat@math.usask.ca}
\thanks{The fifth named author was supported in part by NSERC 0GP0005616.}

\author{}
\address{}
\curraddr{}
\email{}

\author{}
\address{}
\curraddr{}
\email{}
\thanks{}

\date{April 6, 2010, revised December 20, 2011}



\begin{abstract}
In this memoir we present  proofs of basic results, including those
developed so far  by Harold Bell, for the plane fixed point problem:
does every map of a non-separating plane continuum have a fixed
point? Some of these results had been announced much earlier by Bell
but without accessible  proofs. We define the concept of the
variation of a map on a simple closed curve and relate it to the
index of the map on that curve: Index = Variation + 1. A  prime end
theory is developed through hyperbolic chords in maximal round balls
contained in the complement of a non-separating plane continuum $X$.
We define the concept of an {\em outchannel} for a fixed point free
map which carries the boundary of $X$ minimally  into itself and
prove  that such a map has a \emph{unique} outchannel, and that
outchannel must have variation $-1$. Also Bell's Linchpin Theorem for
a foliation of a simply connected domain, by closed convex subsets,
is extended to arbitrary domains in the sphere.

We introduce the notion of an oriented map of the plane and show that
the perfect oriented maps of the plane coincide with confluent (that
is composition of monotone and open) perfect maps of the plane. A
fixed point theorem for  positively oriented, perfect maps of the
plane is obtained. This generalizes results announced by Bell in
1982.

A continuous map of an interval $I\subset \mathbb{R}$ to
$\mathbb{R}$ which sends the endpoints of $I$ in opposite directions
has a fixed point. We generalize this to maps on non-invariant
continua in the plane under positively oriented maps of the plane
(with appropriate boundary conditions). Similar methods imply that in
some cases non-invariant continua in the plane are degenerate. This
has important applications in complex dynamics. E.g., a special case
of our results shows that if $X$ is a non-separating invariant
subcontinuum of the Julia set of a polynomial $P$ containing no fixed
Cremer points and exhibiting no local rotation at all fixed points,
then $X$ must be a point. It follows that impressions of some
external rays to polynomial Julia sets are degenerate.

\end{abstract}

\maketitle
\dedicatory{Dedicated to Harold Bell}

\setcounter{page}{4}

\tableofcontents

\listoffigures


\chapter*{Preface}
By a \emph{continuum} we mean a compact and connected metric space
and by a \emph{non-separating} \index{non-separating} continuum $X$
in the plane $\C$ we mean a continuum $X\subset \C$ such that $\C\sm
X$ is connected. Our work is motivated by the following
long-standing problem \cite{ster35} in topology.

{\bf Plane Fixed Point Problem}: {\em ``Does a continuous function taking a
non-separating plane continuum into itself always have a fixed point?"}

To give the reader perspective we would like to make a few brief historical remarks
(see \cite{kw91,bing69,bing81} for much more information).

Borsuk \cite{bors35}  showed in 1932 that the answer to the above
question is yes if $X$ is also locally connected. Cartwright and
Littlewood \cite{cartlitt51} showed in 1951 that a map of a
non-separating plane continuum $X$ to itself has a fixed point if the
map can be extended to an \emph{orientation-preserving} homeomorphism
of the plane. It was 27 years before Harold Bell \cite{bell78}
extended this result to the class of \emph{all} homeomorphisms of the
plane.  Then Bell announced in 1982 (see also Akis \cite{akis99})
that the Cartwright-Littlewood Theorem can be extended to the class
of all holomorphic maps of the plane. For other partial results in
this direction see, e.g., \cite{hami51, hago71, bell79, minc90,
hago96, minc99}.

In this memoir the Plane Fixed Point Problem is addressed. We develop
and further generalize tools, first introduced by Bell, to elucidate
the action of a fixed point free map (should one exist). We are
indebted to Bell for sharing his insights  with us. Some of the
results in this memoir were first obtained by him. Unfortunately,
many of the proofs were not accessible. Since there are now multiple
papers which rely heavily upon these tools (e.g., \cite{overtymc07,
blokoverto1, bclos08}) we believe that they deserve to be developed
in a coherent fashion. We also hope that by making these tools
available to the mathematical community, other applications of these
results will be found. In fact, we include in Part 2 of this text new
applications which illustrate their usefulness.

Part 1 contains the basic theory, the main ideas of which are due to
Bell. We introduce Bell's notion of variation and prove his theorem
that index equals variation increased by 1 (see Theorem~\ref{I=V+1}).
Bell's Linchpin Theorem~\ref{Hypmain} for simply connected domains is
extended to arbitrary domains in the sphere and proved using an
elegant argument due to Kulkarni and Pinkall \cite{kulkpink94}. Our
version of this  theorem (Theorem~\ref{KPthrm}) is essential for the
results later in the paper.

Building upon these ideas, we will introduce in Part 1 the class of
oriented maps of the plane and show that it decomposes into two
classes, one of which preserves and the other of which reverses local
orientation. The extension from holomorphic to positively oriented maps is
important since it allows for simple local perturbations of the map (see
Lemma~\ref{pararepel}) and significantly simplifies further usage of the
developed tools.

In Part 2 new applications of these results are considered. A Zorn's
Lemma argument shows, that if one assumes a negative solution to the
Plane Fixed Point Problem, then there is a subcontinuum $X$ which is
minimal invariant. It follows from Theorem~\ref{densechannel} that
for such a minimal continuum, $f(X)=X$. We recover Bell's result
\cite{bell67} (see also Sieklucki \cite{siek68}, and Iliadis
\cite{ilia70}) that the boundary of $X$ is indecomposable with a
dense channel (i.e., there exists a prime end $\mc{E}_t$ such that
the principal set of the external ray $R_t$ is all of $\partial X$).

As the first application we show in Chapter~\ref{outcha} that $X$ has
a \emph{unique outchannel} (i.e., a channel in which points basically
map farther and farther away from $X$) and this outchannel must have
variation $-1$ (i.e., as the above mentioned points map farther and
farther away from $X$, they are ``flipped with respect to the center
line of the channel'').

The next application of the tools developed in Part 1 directly
relates to the Plane Fixed Point Problem. We introduce the class of
oriented maps of the plane (i.e., all perfect maps of the plane onto
itself which are the compositions of monotone and branched covering
maps of the plane). The class of oriented maps consists of two
subclasses: positively oriented and negatively oriented maps. In
Theorem~\ref{fixpoint} we show that the Cartwright-Littlewood Theorem
can be extended to positively oriented maps of the plane.

These results are  used in \cite{blokoverto1}. There we consider a
branched covering map $f$ of the plane. It follows from the above
that if $f$ has an invariant and fixed point free continuum $Z$,
then $f$ must be negatively oriented. We show in \cite{blokoverto1}
that if, moreover, $f$ is an oriented map of degree $2$, then $Z$
must contain a continuum $X$ such that $X$ is fully invariant (so
that $X$ contains the critical point and $f|_X$ is not one-to-one).
Thus, $X$ bears a strong resemblance to a connected filled in Julia
set of a quadratic polynomial.

The rest of Part 2 is devoted to extending the existence of a fixed
point in planar continua under positively oriented maps established
in Theorem~\ref{fixpoint}. We extend this result to non-invariant
planar continua. First the result is generalized to dendrites;
moreover, it is strengthened by showing that in certain cases the map
must have infinitely many periodic cutpoints.

The above results on dendrites have applications in complex dynamics.
For example, they are used in \cite{bco08} to give a criterion for
the connected Julia set of a complex polynomial to have a
non-degenerate locally connected model. That is, given a connected
Julia set $J$ of a complex polynomial $P$, it is shown in
\cite{bco08} that there exists a locally connected topological Julia
set $J_{top}$ and a monotone map $m:J\to J_{top}$ such that for every
monotone map $g:J\to X$ from $J$ onto a locally connected continuum
$X$, there exists a monotone map $f:J_{top}\to X$ such that $g=f\circ
m$. Moreover, the map $m$ has a dynamical meaning.  It
semi-conjugates the map $P|_J$ to a topological polynomial
$P_{top}:J_{top}\to J_{top}$. In general, $J_{top}$ can be a single
point. In \cite{bco08}  a necessary and sufficient condition for the
non-degeneracy of $J_{top}$ is obtained. These results extend Kiwi's
fundamental result \cite{kiwi97} on the semi-conjugacy of polynomials
without Cremer or Siegel points to all polynomials with connected
Julia set.

Finally  the results on the existence of fixed points  in invariant
planar continua under positively oriented maps are extended to
non-invariant planar continua. We introduce the notion of
``scrambling of the boundary'' of a plane continuum $X$ under a
positively oriented map and extend the fixed point results to
non-invariant continua on which the map scrambles the boundary. These
conclusions are strengthened  by showing that, under additional
assumptions, a non-degenerate continuum must either contain a fixed
point in its interior, or must contain a fixed point near which the
map ``locally rotates''. Hence, if neither of these is the case, then
the continuum in question must be a point. This latter result is used
to show that in certain cases impressions of external rays to
connected Julia sets are degenerate.

These last named results have had other applications in complex
dynamics. In \cite{bclos08} these results were used to generalize the
well-known Fatou-Shishikura inequality in the case of a polynomial
$P$ (in general, the Fatou-Shishikura inequality holds for rational
functions, see \cite{fa, Mitsu}). For polynomials this inequality
limits the number of attracting and irrationally neutral periodic
cycles by the number of critical points of $P$. The improved count
involves classes of (weakly recurrent) critical points and wandering
subcontinua in the Julia set.

The results in Part 1 of this memoir were mostly obtained in the
late 1990's. Most of the applications in Part 2, including the
results on non-invariant plane continua and the applications in
dynamics, have been obtained  during 2006--2009. Finally the authors
are indebted to a careful reading by the referee which resulted in
numerous changes and improvements.

\aufm{Alexander M.~Blokh}

\aufm{Robbert J.~Fokkink}

\aufm{John C.~Mayer}

\aufm{Lex G.~Oversteegen}

\aufm{E.~D.~Tymchatyn}

\mainmatter

\chapter{Introduction}\label{intro}

\subsection{Notation and the main problem} We denote the plane  by $\complex$, \index{complex@$\complex$} the Riemann
sphere by  $\rsphere=\complex\cup\{\infty\}$, \index{complex@$\rsphere$} the
real line by $\real$ \index{real@$\real$} and the unit circle by
$\uc=\real/\zed$. Let $X$ be a plane compactum. Since $\complex$ is locally
connected and $X$ is closed, complementary domains of $X$ are open. By $T(X)$
\index{TX@$T(X)$} we denote the {\em topological hull} \index{topological hull}
\index{hull!topological}
of $X$ consisting of $X$ union all of its bounded complementary domains.  Thus,
$U^\infty=U^\infty(X)=\rsphere\sm T(X)$ \index{U@$U^\infty$} is the unbounded
complementary component of $X$ containing infinity. Observe that if $X$ is a
continuum, then $U^\infty(X)$ is simply connected. The Plane Fixed Point
Problem, attributed to \cite{ster35}, is one of the central long-standing
problems in plane topology. It serves as a motivation for our work and can be
formulated as follows.

\begin{prob}[Plane Fixed Point Problem]\label{fpt} Does a continuous function
taking a non-separating plane continuum into itself always have a
fixed point?
\end{prob}

\subsection{Historical remarks}

To give the reader perspective we would like to make a few historical remarks
concerning the Plane Fixed Point Problem (here we cover only major steps towards solving
the problem).

In 1912 Brouwer \cite{brou12a} proved that any orientation preserving homeomorphism of
the plane, which keeps a bounded set invariant, must have a fixed point (though not necessarily in
that set). This fundamental result has found many important applications.
It was recognized early on that the location of a fixed point  should be determined
if the invariant set is a non-separating continuum (in that case a fixed point should be located in the invariant
continuum) and many papers have been devoted to
obtaining partial solutions to the Plane Fixed Point Problem.

Borsuk \cite{bors35}  showed in 1932 that the answer is yes if $X$
is also locally connected. Cartwright and Littlewood
\cite{cartlitt51} showed in 1951 that a continuous map of a
non-separating continuum $X$ to itself has a fixed point in $X$ if
the map can be extended to an \emph{orientation-preserving}
homeomorphism of the plane. (See Brown \cite{brow77} for a very
short proof of this theorem based on the above mentioned result by
Brouwer). The proof by Cartwright-Littlewood Theorem made use of the
\emph{index of a map on a simple closed curve} and this idea has
remained the basic approach in many partial solutions.

The most general result was obtained by Bell \cite{bell67} in the
early 1960's. He showed that any counterexample must contain an
invariant indecomposable subcontinuum. Hence the Plane Fixed Point
Problem has a positive solution for hereditarily decomposable plane
continua (i.e., for continua $X$ which do not contain indecomposable
subcontinua). Bell's result was also based on the notion of the index
of a map, but he  introduced new ideas to determine the index of a
simple closed curve which runs tightly around a possible
counterexample. Unfortunately, these ideas were not transparent and
were never fully developed. Alternative proofs of Bell's result
appeared soon after Bell's announcement (see \cite{siek68,ilia70}).
Regrettably these results did not develop Bell's ideas.

In 1978 Bell \cite{bell78} used his earlier result   to extend the
result by  Cartwright and Littlewood to the class of \emph{all}
homeomorphisms of the plane.  Then Bell announced in 1982 (see also
Akis \cite{akis99} where a wider class of differentiable functions
was used) that the Cartwright-Littlewood Theorem can be extended to
the class of all holomorphic maps of the plane. The existence of
fixed points for orientation preserving homeomorphisms of the
\emph{entire plane} under various conditions was also considered in
\cite{brow84a,fath87,fran92,guil94}, and the existence of a point of
period two for orientation reversing homeomorphisms in \cite{boni04}.

As indicated above, positive results require an additional hypothesis
either on the continuum $X$ (as in Borsuk's result where the
assumption is that $X$ is locally connected) or on the map (as in
Bell's case where the assumption is that $f$ is a homeomorphism of
the plane). Other positive results of the first type include results
by Hamilton \cite{hami51} ($X$ is chainable), Hagopian \cite{hago71}
($X$ is arcwise connected) Minc \cite{minc90} ($X$ is the continuous
image of the pseudo arc) and \cite{hago96} ($X$ is simply connected).
Positive results of the second type require the map to be either a
homeomorphism \cite{cartlitt51,bell78}, holomorphic (as announced by
Bell) or smooth  with non-negative Jacobian and isolated singularities
\cite{akis99}.

David Bellamy \cite{bell79} produced an important related counterexample. He showed that
there exists a tree-like continuum $X$, whose every proper subcontinuum  is an arc and
which admits a fixed point free homeomorphism. It is not known if  examples of this type can be embedded in the plane.
Minc \cite{minc99}  constructed a  tree-like continuum which is the continuous image of the pseudo arc and
admits a fixed point free map.

\subsection{Major tools}

In this subsection we describe the major tools developed in Part 1.

\subsubsection{Finding fixed points with index and variation}

It is easy to see that a map of a plane continuum to itself can be extended to
a perfect map of the plane. We study the slightly more general question, ``Is
there a plane continuum $Z$ and a  perfect continuous function
$f:\complex\to\complex$ taking $Z$ into $T(Z)$ with no fixed points in $T(Z)$?"
A Zorn's Lemma argument shows that if one assumes that the answer is ``yes,"
then there is a subcontinuum $X\subset Z$, minimal with respect to these
properties. It will follow from Theorem~\ref{densechannel} that for such a
minimal continuum, $f(X)=X=\bd T(X)$ (though it may not be the case that
$f(T(X))\subset T(X)$).  Here $\partial T(X)$ denotes the boundary of
$T(X)$\index{del@$\partial$ boundary operator}.

Many fixed point results make use of the notion of the index $\ind(f,S)$, which
counts the number of revolutions of the vector 
connecting $z$ with $f(z)$ for $z\in S$ running along a simple closed curve $S$
in the plane. As is well-known, if $f:\complex\to \complex$ is a map and
$\ind(f,S)\ne 0$, then $f$ must have a fixed point in $T(S)$ (for completeness
we prove this in Theorem~\ref{fpthm}). In order to establish fixed points in
invariant plane continua $X$, one often approximates $X$ by a simple closed
curve $S$ such that $X\subset T(S)$. If $\ind(f,S)\ne 0$ and $S$ is
sufficiently tight around $X$, one can conclude that $f$ must have a fixed
point in $T(X)$. Hence the main work is in showing that $\ind(f,S)\ne 0$ for a
suitable simple closed curve around $X$.

Bell's fundamental idea was to replace the count of the number of rotations of
the vector $\overrightarrow{zf(z)}$ with respect to a fixed axis (say, the $x$-axis) by a count which involves
the moving frame of external rays. Consider, for example, the unit circle $\uc$ and a fixed point
free map $f:\uc\to\complex$. For each $z=e^{2\pi i \theta}\in \uc$ let $R_z$ be the
external ray $\{re^{2\pi i\theta}\mid r>1\}$. Now count the number of times the point
$f(z)$, $z\in\uc$, crosses the external ray $R_z$, taking into account the direction
of the crossing. Call this count the variation $\var(f,\uc)$. It is easy to
see that in this case $\ind(f,\uc)=\var(f,\uc)+1$.

Another useful idea is to consider a similar count not on the entire unit
circle (or, in general, not on the entire simple
closed curve $S$ containing $X$ in its topological hull $T(S)$), but on subarcs
of $\uc$, which map off themselves and  whose endpoints map inside $T(\uc)$. By
doing so, one obtains Bell's notion of variation $\var(f,A)$ on arcs (see
Definition~\ref{vararc}).
If one can write $\uc$ as a finite union $A_i$ of arcs such that any two meet at most in a common endpoint
and, for all $i$,  $f(A_i)\cap A_i=\0$ and both endpoints map in $T(\uc)$, one can define $\var(f,\uc)=\sum \var(f,A_i)$.
Then, as above, one can use it to compute the index and
prove that index equals variation increased by 1 (see Theorem~\ref{I=V+1}).

The relation between index and variation immediately implies a few
classic results, in particular that of  Cartwright and Littlewood.
To see this one only needs to show that, if $h$ is an orientation
preserving homeomorphism of the plane, $X$ is an invariant plane
continuum and $A_i\subset S$ are subarcs of a tight simple closed
curve around $X$, with endpoints in $X$ as above and such that
$f(A_i)\cap A_i=\0$, then $\var(f,A_i)\ge 0$ (see
Corollary~\ref{posvaro}). Then $\ind(f,S)=\sum\var(f,A_i)+1\ge 1$
and $T(X)$ contains a fixed point as desired. Observe that the
connection between index and variation is essential for all the
applications in Part 2.

\subsubsection{Other tools, such as foliations and oriented maps}

Let $X$ be a non-separating plane continuum and let $f:\C\to\C$ be a map such that
$f(X)\subset T(X)$ and $f$ has no fixed points in $T(X)$.
In general we need more control of the simple closed curve $S$ around
$X$ (and of the action of the map on $S\sm X$).  Bell originally
accomplished this by partitioning the complement of $T(X)$ in the Euclidean convex hull $\ECH(X)$ of $X$ by
Euclidean convex sets. Suppose that $B$ is a maximal round closed ball (or a
half plane) such that $\Int(B)\cap T(X)=\0$ and $|B\cap X|\ge 2$, and consider the
set $\ECH(B\cap X)$. For any two such balls $B_1,B_2$ either
$K_{1,2}=\ECH(B_1\cap X)\cap\ECH(B_2\cap X)$ is empty, or $K_{1,2}$ is a single point in $X$,
or this intersection is a common chord contained
in both of their boundaries. Bell's Linchpin Theorem (see Theorem~\ref{Hypmain} and the remark following it)
states that the collection  $\ECH(B\cap X)$ over all such maximal balls
covers all of $\ECH(X)\sm T(X)$.

Hence the collection $\ECH(B\cap X)$ over  all such balls provides a
partition of $\ECH(X)\sm X$ into Euclidean convex sets contained in
maximal round balls. The collection of  chords in the boundaries of
the sets $\ECH(B\cap X)$ for all such balls have the property that
any two distinct chords meet in at most a common endpoint in $X$. In
other words, this set of chords is a lamination  in the sense of
Thurston \cite{thur85} even though in Thurston's paper laminations
appear in a very different, namely complex dynamical, context. This
Linchpin Theorem can be used to extend the map $f|_{T(X)}$ over
$\ECH(X)\sm T(X)$ (first linearly over all the chords in the
lamination and then over all remaining components of the
complement). We will illustrate the usefulness of Bell's partition
by showing
 that the well-known Schoenflies Theorem follows
immediately.
It can also be used to obtain a particular simple closed curve $S$ around $X$ so that every component
of $S\sm X$ is a chord in the lamination.

In our version of Bell's Linchpin Theorem we consider arbitrary open and connected subsets of the sphere $U$ and
 we use round balls in the spherical metric on the
sphere. Moreover,  the Euclidean geodesics are replaced by
hyperbolic geodesics (in either the hyperbolic metric in each ball
or, if $U$ is simply connected, in the hyperbolic metric on $U$).
This way we get a lamination of all of $U^\infty(X)=\sphere\sm T(X)$
(and not just $\ECH(X)\sm T(X)$) and the resulting lamination is
easier to apply in other settings. We give a proof of this theorem
using an elegant argument due to Kulkarni and Pinkall
\cite{kulkpink94} which also allows for the extension over arbitrary
open and connected subsets of the sphere (see Theorem~\ref{KPthrm};
this later theorem is used in Chapter~\ref{outcha}). Bell's Linchpin
Theorem follows as a corollary.

 This new partition of a complementary domain of a continuum can also be used in other settings  to
extend a homeomorphism on the boundary of a planar domain over the entire domain. In \cite{overtymc07} this is used to show that
an isotopy of a planar continuum, starting at the identity, extends to an isotopy of the  plane.
This extends a well-known result regarding the extension of a holomorphic motion \cite{sullthur86,slod91}.
In \cite{ov09} this partition is used to give necessary and sufficient conditions to extend a homeomorphism,
of an arbitrary planar continuum, over the  plane.

The development of the necessary tools in Part 1  is completed by introducing the
notion of (positively or negatively) oriented maps and studying their
properties. Holomorphic maps are
prototypes of positively oriented maps but in general positively oriented  maps do not have to be differentiable, light,
open or monotone. Locally at non-critical points, positively oriented maps behave  like
orientation-preserving homeomorphisms in the sense that they preserve local
orientation. Compositions of open, perfect and of monotone, perfect surjections
of the plane are \emph{confluent} (i.e., such that components of the preimage of any
continuum map onto the continuum) and naturally decompose into two classes, one
of which preserves and the other of which reverses local orientation. We show
that any confluent map of the plane is itself a composition of a monotone and a
light-open map of the plane. It is  shown that an oriented map of the plane
induces a map from the circle of prime ends of a component of the pre-image
of an acyclic plane continuum to the circle of prime ends of that continuum.

\subsection{Main applications}

Part 2 contains applications of the tools developed in Part 1. Directly or
indirectly, these applications deal with the Plane Fixed Point Problem. We describe them
below.

\subsubsection{Outchannel and hypothetical minimal continua without fixed points}

The first application is in Chapter~\ref{outcha} where we establish the
existence of a unique outchannel. Let us consider this in more detail. If there
exists a counterexample to the Plane Fixed Point Problem, then there exists a continuum
$X$ which is minimal with respect to $f(X)\subset T(X)$ and $f$ has no fixed
point in $T(X)$. Bell has shown \cite{bell67} (see also \cite{siek68,ilia70}) that such a continuum has at
least one dense outchannel of negative variation. Since $X$ is
minimal with a dense outchannel, $X$ is an indecomposable continuum and
$f(X)=X$.

A dense outchannel is a prime end so that its principal set is all of $X$ and if $\{C_i\}$ is a defining sequence of
crosscuts, then $f$ maps these crosscuts essentially ``out of the channel''
(i.e., closer to infinity) for $i$ sufficiently large. The latter statement is accurately reflected by the fact
that $\var(f,C_i)\ne 0$.
In case that the
complement of $X$ is invariant, this can be described by saying that the
crosscut $f(C_i)$ separates $C_i$ from infinity in $U^\infty(X)$ (and hence, in
this case, crosscuts do really map out of the channel). As a new result,
the main steps of the proof of which were outlined by Bell,  we show that there always exists \emph{exactly one
outchannel} and that its variation is $-1$, while all other prime ends must
have variation $0$. Using these results it is shown in \cite{blokoverto1} that
if $f$ is a negatively oriented branched covering map of degree $2$ which has
a non-separating invariant continuum $Z$ without fixed points, then the minimal
subcontinuum $X\subset Z$ has the following additional properties:

\begin{enumerate}

\item  $X$ is indecomposable,

\item the unique critical point $c$ of $f$ belongs to $X$, $f(X)=X$, and
$f(\complex\sm X)=\complex\sm X$ (so that $X$ is \emph{fully invariant}),

\item  $f$ induces a covering map
$F:\uc\to\uc$ from the circle of prime ends of $X$ to itself of degree $-2$.

\item $F$ has three fixed points one of which corresponds to the unique dense outchannel
whereas the remaining two fixed points correspond to dense inchannels (i.e.,
for a defining sequence of crosscuts $\{C_i\}$, $C_i$ separates $f(C_i)$ from
infinity in $U^\infty$)

\end{enumerate}
 Moreover, as part of the argument, the map $f$ is modified in $\complex\sm X$ so that
 the new map $g$ keeps the tail of the external ray, which runs down the outchannel, invariant
 and maps the points on them closer to infinity.

\subsubsection{Fixed points in invariant continua for positively oriented maps}

Other applications of the tools developed in Part 1 are obtained in
Chapter~\ref{ch:fxpt}. These are also related to the Plane Fixed Point Problem. As
we will see below, the corresponding results can be in turn further applied in
complex dynamics, leading to some structural results in the field, such as
constructing finest locally connected models for connected Julia sets or
studying wandering continua inside Julia sets and an extension of the  Fatou-Shishikura
inequality so that it includes counting wandering branch-continua (see Section~\ref{complappl}).

The first application in Chapter~\ref{ch:fxpt} is the most straightforward of
them all: in Theorem~\ref{fixpoint} from Section~\ref{sec:fxpt} we prove that a
positively oriented map $f$ which takes a continuum $X$ into the topological
hull $T(X)$ of $X$ must have a fixed point in $T(X)$. In other words, in
Theorem~\ref{fixpoint} the Plane Fixed Point Problem is solved in the affirmative for
positively oriented maps. As we will see, the extension from holomorphic maps to positively oriented maps
is important since the latter class allows for easy local perturbations. This will allow us
to deal with parabolic points in a Julia set (see Lemma~\ref{pararepel}, Theorem~\ref{pointdyn}
and Corollary~\ref{degimpr}).

The idea of the proof is as follows. First we prove in Corollary~\ref{posvaro}
that if a crosscut $C$ of $X$ is mapped off itself by $f$ then the variation on
$C$ is non-negative. This is done by completing the crosscut $C$ to a very
tight simple closed curve $S$ around $X$ and observing that in fact the
variation in question can be computed by computing the winding number of $f$ on
$S$. Notice that versions of this idea are used later on when we prove the
existence of fixed points in non-invariant continua satisfying certain
additional conditions.

To prove Theorem~\ref{fixpoint}, we first assume by way of contradiction that
$T(X)$ contains no fixed points. In this case there are no fixed points in the
closure $\ol{U}$ of a sufficiently small neighborhood $U$ of $T(X)$. Using this,
we construct a simple closed curve which goes around $X$ inside such a
neighborhood $U$ and ``touches'' $X$ at a sufficiently dense set of points so
that  arcs between consecutive points of $S\cap X$ are very small. Since there are no fixed
points in $\ol{U}$, we can guarantee that the images of these arcs are disjoint
from  themselves. Hence by the above described Corollary~\ref{posvaro} the
variations of all these arcs are non-negative. By Theorem~\ref{I=V+1} this
implies that the index of $f$ on $S$ is not equal to zero and hence, by
Theorem~\ref{fpthm} there must exist a fixed point inside $T(S)$, a
contradiction.

\subsubsection{Fixed points in  non-invariant continua: the case of dendrites}

Our generalizations of Theorem~\ref{fixpoint} are inspired by a simple observation. The
most well-known particular case for which the Plane Fixed Point Problem is solved is that
of a map of a closed interval $I=[a, b]$, $a<b$ into itself in which case there
must exist a fixed point in $I$. However, in this case a more general result
can easily be proven, of which the existence of a fixed point in an invariant interval
is a consequence.

Namely, instead of considering a map $f:I\to I$ consider a map $f:I\to \R$ such
that either (a) $f(a)\ge a$ and $f(b)\le b$, or (b) $f(a)\le a$ and $f(b)\ge
b$. Then still there must exist a fixed point in $I$ which is an easy corollary
of the Intermediate Value Theorem applied to the function $f(x)-x$. Observe
that in this case $I$ need not be invariant under $f$. Observe also that
without the assumptions on the endpoints, the conclusion on the existence of a
fixed point inside $I$ cannot be made because, e.g., a shift map on $I$ does
not have fixed points at all. The conditions (a) and (b) above can be thought
of as boundary conditions imposing restrictions on where $f$ maps the boundary
points of $I$ in $\reals$.

Our main aim in the remaining part of Chapter~\ref{ch:fxpt} is to consider some
other cases for which the Plane Fixed Point Problem can be solved in the affirmative (i.e.,
the existence of a fixed point in a continuum can be established) despite the
fact that the continuum $X$ in question is not  invariant. We proceed with
our studies in two directions. Considering $X$, we replace the invariantness of
the continuum by boundary conditions in the spirit of the above ``interval
version'' of the Plane Fixed Point Problem. We also show that there must exist a fixed
point of ``rotational type'' in the continuum (and hence, if it is known that
such a point does not exist, then the continuum in question is a point).

Since we now deal with continua significantly more complicated than an interval,
inevitably the boundary conditions become rather intricate. Thus we postpone
the precise technical statement of the results until Chapter~\ref{ch:fxpt} and
use here a more descriptive approach. Observe that particular cases for which
the Plane Fixed Point Problem is solved so far can be divided into two categories: either
$X$ has additional properties, or $f$ has additional properties. In the first
category the above considered ``interval case'' is the most well-known. A
direct extension of it is the following well-known theorem (which follows from
Borsuk's theorem \cite{bors35},
see  \cite{nadl92} for a direct proof); recall that a \emph{dendrite}\index{dendrite} is a
locally connected continuum containing no simple closed curves.

\begin{thm}\label{th:dendr}
If $f:D\to D$ is a continuous map of a dendrite into itself then it has a fixed
point.
\end{thm}

Here $f$ is just a continuous map but the continuum $D$ is very nice. In
Section~\ref{sec:dendr} Theorem~\ref{th:dendr} is generalized to the case when
$f:D_1\to D_2$ maps a dendrite $D_1$ into a dendrite $D_2\supset D_1$ and
certain conditions on the behavior of the points of the set $E=\ol{D_2\sm
D_1}\cap D_1$ under the map $f$ are fulfilled (observe that $E$ may be
infinite). This presents a ``non-invariant'' version of Plane Fixed Point Problem for
dendrites and can be done in the spirit of the interval case described earlier.
Moreover, with some additional conditions it has consequences related to the
number of periodic points of $f$.

More precisely, we introduce the notion of \emph{boundary scrambling} for
dendrites in the situation above. It simply means that for each
\emph{non-fixed} point $e\in E$, $f(e)$ is contained in a component of
$D_2\sm\{e\}$ which intersects $D_1$ (see Definition~\ref{bouscr}). Observe
that if $D_1$ \emph{is} invariant then $f$ automatically scrambles the
boundary. We prove the following theorem.

\setcounter{chapter}{7}\setcounter{section}{2}\setcounter{thm}{1}
\begin{thm} Suppose that $f:D_1\to D_2$ is a map between  dendrites,
where $D_1\subset D_2$, which scrambles the boundary. Then $f$ has a fixed point.
\end{thm}

\setcounter{chapter}{1}\setcounter{section}{0}\setcounter{thm}{2}

Next in Section~\ref{sec:dendr} we define \emph{weakly repelling periodic points}.
Basically, a  point $a\in D_1$ is a \emph{weakly repelling periodic point
(for $f^n$)} if there exists $n\ge 1$ and
a component $B$ of $D_1\sm \{a\}$ such that $f^n(a)=a$ and arbitrarily close to $a$ in $B$
there exist cutpoints of $D_1$ fixed under $f^n$ or points $x$ separating $a$
from $f^n(x)$. Note that a fixed point $a$ of $f$ can be a weakly repelling periodic point
for $f^n$ while it is not weakly repelling for $f$.  We
use this notion to prove Theorem~\ref{infprpt} where we show that if $D$ is a dendrite and $f:D\to D$
is continuous and all its periodic points are weakly repelling, then $f$ has
infinitely many periodic cutpoints. Then we rely upon Theorem~\ref{infprpt} in
Theorem~\ref{lamwkrp} where it is shown that if $g:J\to J$ is a
\emph{topological polynomial} on its dendritic Julia set (e.g., if $g$ is a
complex polynomial with a dendritic Julia set) then it has infinitely many
periodic cutpoints.

\subsubsection{Fixed points in non-invariant continua: the planar case}

In Sections~\ref{sec:fxpt-noni} and \ref{sec:fxpt-noniso} we draw a
parallel with the interval case for planar maps and extend
Theorem~\ref{fixpoint} to non-invariant continua under positively
oriented maps such that certain ``boundary'' conditions are
satisfied. Namely, suppose that $f:\C\to \C$ is a positively
oriented map and $X\subset \C$ is a non-separating continuum. Since
we are interested in fixed points of $f|_X$, it makes sense to
assume that at least $f(X)\cap X\ne \0$. Thus, we can think of
$f(X)$ as a new continuum which ``grows'' from $X$ at some places.
We assume that the  ``pieces'' of $f(X)$ which grow outside $X$ are
contained in disjoint non-separating continua $Z_i$ so that $f(X)\sm
X\subset \cup_i Z_i$.

We also assume that places at which the growth takes places - i.e., sets $Z_i\cap
X=K_i$ - are non-separating continua for all $i$. Finally, the main assumption
here is the following restriction upon where the continua $K_i$ map under $f$:
we assume that for all $i$, $f(K_i)\cap [Z_i\sm K_i]=\0$. If this is all that is
satisfied, then the map $f$ is said to \emph{scramble the boundary (of $X$}). A
stronger version of that is when for all $i$, either $f(K_i)\subset K_i$, or $f(K_i)\cap
Z_i=\0$; then we say that $f$  \emph{strongly scrambles the boundary (of $X$)}
(see Definition~\ref{scracon}). The continua $K_i$ are called \emph{exit
continua (of $X$)}. The main result of Section~\ref{sec:fxpt-noni} is:

\setcounter{chapter}{7}\setcounter{section}{3}\setcounter{thm}{2}
\begin{thm}
Suppose that   $f$ is positively oriented
and strongly scrambles the boundary of $X$, then $f$ has a fixed point in $X$.
\end{thm}

\setcounter{chapter}{1}\setcounter{section}{0}\setcounter{thm}{2}
As an illustration, consider the case when $X\cup (\cup_i Z_i)$ is a dendrite
and all sets $K_i$ are singletons. Then it is easy to see that both scrambling
and strong scrambling of the boundary in the sense of dendrites mean the same
as in the sense of the planar definition. Of course, in the planar case we deal
with a much more narrow class of maps, namely positively oriented maps, and
with a much wider variety of continua, namely all  non-separating planar continua.
This fits into the ``philosophy'' of our approach: whenever we obtain a result
for a wider class of continua, we have to consider a more specific class of maps.

For the family of positively oriented maps with isolated fixed points we
specify this result as follows. We introduce the notion of the map $f$
\emph{repelling outside $X$} at a fixed point $p$ (see Definition~\ref{repout};
basically, it means that there exists an invariant external ray of $X$ which
lands at $p$ and along which the points are repelled away from $p$. Then in
Theorem~\ref{locrot} we show that \emph{if $f$ is a positively oriented map with
isolated fixed points and $X\subset\C$ is a non-separating continuum or a point
such that $f$ scrambles the boundary of $X$ and for every fixed point $a$ the
winding number at $a$ equals 1 and $f$ repels at $a$, then $X$ must be a
point}.

\subsubsection{Fixed points in non-invariant continua for polynomials}

These theorems apply to polynomials $P$, allowing us to obtain a few
corollaries dealing with the existence of periodic points in certain
parts of the Julia set of a polynomial and degeneracy of certain
impressions. To discuss this we assume  knowledge of the standard
definitions such as \emph{Julia sets $J_P$, filled-in Julia sets
$K_P=T(J_P)$, Fatou domains, parabolic periodic points} etc which
are formally introduced in Section~\ref{polycase} and further
discussed in Section~\ref{complappl} (see also \cite{miln00}).
Recall that the set $\iu(J_P)$ (called in this context the
\emph{basin of attraction of infinity}) is partitioned by the
\emph{(conformal) external rays} $R_\al$ with arguments $\al\in
\uc$. If $J_P$ is connected, all rays $R_\al$ are smooth and
pairwise disjoint while if $J_P$ is not connected limits of smooth
external rays must be added. Still, given an external ray $R_\al$ of
$K$, its principal set $\ol{R_\al}\sm R_\al$ can be introduced as
usual.

We then define a \emph{general puzzle-piece} of a filled-in Julia set
$K_P$ as a continuum $X$ which is cut from $K_P$ by means of choosing
a few \emph{exit continua} $E_i\subset X$ each of which contains the
principal sets of more than one external ray. We then assume that
there exists a component $C_X$ of the complement in $\C$ to the union
of all such exit continua $E_i$ and their external rays such that
$X\subset (C_X\cap K_P)\cup (\bigcup E_i)$. The external rays
accumulating inside an exit continuum $E_i$ cut the plane into wedges
one of which, denoted by $W_i$, contains points of $X$. The
``degenerate'' case when there are no exit continua is also included
and simply means that $X$ is an invariant subcontinuum of $K_P$.

The main assumption on the dynamics of a general puzzle piece $X$
which we make is that  $P(X)\cap C_X\subset X$ and for each exit
continuum $E_i$ we have $P(E_i)\subset W_i$. It is easy to see that
this essentially means that $P$ \emph{scrambles the boundary of $X$}
(where the role of the ``boundary'' is played by the union of exit
continua).

The conclusion, obtained in Theorem~\ref{pointdyn}, is based upon the above
described results, in particular on Theorem~\ref{locrot}. It states that for
a general puzzle-piece either $X$ contains an invariant parabolic Fatou domain,
or $X$ contains a fixed point which is neither repelling nor parabolic, or $X$
contains a repelling or parabolic fixed point $a$ at which the local rotation
number is not $0$. Let us now list the main dynamical applications of this result.

\subsubsection{Further dynamical applications}

There are a few ways Theorem~\ref{pointdyn} applies in complex (polynomial)
dynamics. First, it is instrumental in studying \emph{wandering cut-continua}
for polynomials with connected Julia sets. A continuum/point $L\subset J_P$ is a
\emph{cut-continuum (of valence $\val(L)$)} if the cardinality $\val(L)$ of the
set of components of $J_P\sm L$ is greater than $1$. A collection of disjoint cut-continua
(it might, in particular, consist of one continuum) is said to be
\emph{wandering} if their forward images form a family of pairwise disjoint
sets. The main result of \cite{bclos08} in the case of polynomials with
connected Julia sets is the following generalization of the Fatou-Shishikura
inequality.

\begin{thm}\label{fdhs}
Let $P$ be a polynomial with connected Julia set, let $N$ be the sum of the
number of distinct cycles of its bounded Fatou domains and the number of cycles of its
Cremer points, and let $\Ga\ne\0$ be a wandering collection of cut-continua
$Q_i$ with valences greater than $2$ which contain no preimages of critical
points of $P$. Then $\sum_\Ga (\val(Q_i)-2)+N\le d-2.$
\end{thm}

In \cite{bclos08} a partition of the plane into pieces by
rays with rational arguments landing at periodic cutpoints of $J_P$ and their
preimages is used. Theorem~\ref{pointdyn} plays a significant role in the proof of
the fact that wandering cut-continua do not enter the pieces containing
Cremer or Siegel periodic points which is an important ingredient of the
arguments in \cite{bclos08} proving Theorem~\ref{fdhs}.

Another application of Theorem~\ref{pointdyn} can be found in \cite{bco08}
where  Kiwi's fundamental result \cite{kiwi97} on the semiconjugacy
of polynomials on their Julia sets without Cremer or Siegel points is extended to all
polynomials with connected Julia sets; in both cases topological polynomials on
their topological Julia sets serve as locally connected models. Denote the monotone semiconjugacy
in question by $\varphi$. In showing in \cite{bco08} that if $x$ is a
(pre)periodic point and $\varphi(x)$ is not equal to a $\varphi$-image of a
Cremer or Siegel point or its preimage then $J_P$ is locally connected at $x$,
Theorem~\ref{pointdyn} plays a crucial role.

Finally, our results concerning dendrites (such as Theorem~\ref{infprpt} and
Theorem~\ref{lamwkrp}) are used 
in \cite{bco08} where  a criterion for the connected Julia set to have a
non-degenerate locally connected model is obtained.  We also rely on Theorem~\ref{infprpt}
and Theorem~\ref{lamwkrp} to show in \cite{bco08} that if such model exists,
and is a dendrite, then the polynomial must have infinitely many bi-accessible
periodic points in its Julia set.

\subsection{Concluding remarks and acknowledgments}

All of the positive  results on the existence of fixed points in this memoir are either for simple continua
(i.e., those which do not contain indecomposable subcontinua) or for positively
oriented maps of the plane. Hence the following special case of the
Plane Fixed Point Problem is a major remaining open problem:

\begin{prob}\label{negprob}
Suppose that $f:\complex\to\complex$ is a negatively oriented branched covering
map,  $|f^{-1}(y)|\le 2$ for all $y\in\C$ and $Z$ is a non-separating plane continuum such that $f(Z)\subset Z$. Must
$f$ have a fixed point in $Z$?
\end{prob}

Suppose that $c$ is the unique critical point of $f$ and that $X\subset Z$ is a minimal continuum such that
$f(X)\subset T(X)$. Then, as was mentioned above, the answer
is yes if  there exists $y\in X\sm \{f(c)\}$ such that $|f^{-1}(y)\cap X|<2$. In particular the answer
to Problem~\ref{negprob} is yes if $f|_Z$ is one-to-one.

Finally let us express, once again, our gratitude to Harold Bell for sharing his insights with us.
His notion of variation of an arc, his index equals variation plus one theorem and his linchpin
theorem of  partitioning a complementary domain of a planar continuum into convex subsets
are essential for the results we obtain here.
Theorem~\ref{outchannel}
(Unique Outchannel) is a new result the main steps of which were outlined by Bell.
  Complete proofs of the following results by Bell:
Theorems~\ref{I=V+1}, \ref{lollipop}, \ref{Hypmain} and \ref{outchannel},  appear
in print for the first time.
For the convenience of the reader we have included an index at the end of the paper.



\part{Basic Theory}

\chapter{Preliminaries and outline  of Part 1}\label{descr1}
In this chapter we give the formal definitions and describe the results of part 1 in more detail.
By a \emph{map} \index{map} $f:X\to Y$ we will always mean a continuous function.

Let $p:\real\to \uc$ denote the covering map $p(x)=e^{2\pi ix}$. Let $g:\uc\to
\uc$ be a map. By the \emph{degree} \index{degree} of the map $g$,
\index{degree@$\dg(g)$} denoted by $\dg(g)$, we mean the number
$\hat{g}(1)-\hat{g}(0)$, where $\hat{g}:\real\to\real$ is a lift of the map $g$
to the universal covering space $\real$ of $\uc$ (i.e., $p\circ\hat{g}=g\circ
p$). It is well-known that $\dg(g)$ is independent of the choice of the lift.

\section{Index}\label{defindex}

Let $g:\uc\to\complex$ be a map and $f:g(\uc)\to\complex$ a fixed point free
map.   Define the map $v:\uc\to \uc$ by
$$v(t)=\frac{f(g(t))-g(t)}{|f(g(t))-g(t)|}.$$

Then the map $v:\uc\to \uc$  lifts to a map $\wh v:\real\to\real$. Define the
{\em index of $f$ \index{index} with respect to $g$}, denoted
\index{index1@$\ind(f,g)$} $\ind(f,g)$ by
$$\ind(f,g)=\wh v(1)-\wh v(0)=\dg(v).$$

Note that $\ind(f,g)$ measures the net number of revolutions of the
vector\linebreak $f(g(t))-g(t)$ as $t$ travels through the unit circle one
revolution in the positive direction.

\begin{rem}\label{remark}The following basic facts hold.
\begin{enumerate}
\item[(a)] If $g:\uc \to\complex$
is a constant map with $g(\uc)=c$ and $f(c)\ne c$, then $\ind(f,g)=0$. \\
\item[(b)] If $f$ is a constant map and $f(\complex)=w$ with $w\not\in g(\uc)$, then
$\ind(f,g)=\win(g, \uc,w)$, the \emph{ winding number of $g$ about
$w$}\index{win@$\win(g,\uc,w)$}. In particular, if $f:\uc\to T(\uc)\setminus
\uc$ is a constant map, then $\ind(f,id|_{\uc})= 1$, where $id|_{\uc}$ is the
identity map on $\uc$\index{id@$id$ identity map}.
\end{enumerate}
\end{rem}

Note also, that for a simple closed curve $S'$ and a  point $w\nin T(f(S'))$ we have $\win(f,S',w)=0$.
Suppose $S\subset\Complex$ is a simple closed curve and $A\subset S$ is a
subarc \index{order!on subarc of simple closed curve} of $S$ with endpoints $a$
and $b$. Then we write $A=[a,b]$ if $A$ is the arc obtained by traveling in the
counter-clockwise direction from the point $a$ to the point $b$ along $S$. In
this case we denote by $<$ the linear order on the arc $A$ such that $a<b$.
\index{counterclockwise order!on an arc in a simple closed curve} We  will call
the order $<$ the \emph{counterclockwise order on}   $A$. Note that
$[a,b]\ne[b,a]$.


More generally, for any arc $A=[a,b]\subset \uc$, with $a<b$ in the
counterclockwise order, define the {\em fractional index}
\index{index!fractional} \cite{brow90} \index{index2@$\ind(f,g|_{[a,b]})$} of
$f$ on the sub-path $g|_{[a,b]}$  by $$\ind(f,g|_{[a,b]})=\wh v(b)-\wh v(a).$$
While, necessarily, the index of $f$ with respect to $g$ is an integer, the
fractional index of $f$ on $g|_{[a,b]}$ need not be. We shall have occasion to
use fractional index in the proof of Theorem~\ref{I=V+1}.

\begin{prop} \label{fracindex}
Let $g:\uc\to\complex$ be a map with $g(\uc)=S$, and suppose $f:S\to\complex$
has no fixed points on $S$. Let $a\not=b\in \uc$ with $[a,b]$ denoting the
counterclockwise subarc on $\uc$ from $a$ to $b$  (so $[a,b]$ and $(b,a)$ are complementary arcs and
$\uc=[a,b]\cup[b,a]$).
Then $\ind(f,g)=\ind(f,g|_{[a,b]})+\ind(f,g|_{[b,a]})$.
\end{prop}

\section{Variation} \label{compofvar}

In this section we introduce the notion
of variation of a map on an arc and relate it to winding number.

\begin{defn} [Junctions]\label{junction}\index{junction}
The {\em standard junction} $J_O$ is the union of the three rays
$J^i_O=\{z\in\complex\mid z=re^{i\pi/2},\ r\in[0,\infty)\}$,
$J^+_O=\{z\in\complex\mid z=r,\ r\in[0,\infty)\}$, $J^-_O=\{z\in\complex\mid
z=re^{i\pi},\ r\in[0,\infty)\}$, having the origin $O$ in common.   A {\em
junction (at $v$)} $J_v$ is the image of $J_O$ under any orientation-preserving
homeomorphism $h:\complex\to\complex$ where $v=h(O)$. We will often suppress
$h$ and refer to $h(J^i_O)$ as $J^i_v$, and similarly for the remaining rays in
$J_v$.  Moreover, we require that for each bounded neighborhood $W$ of $v$, $d(J^+_v\sm
W, J^i_v\sm W)>0$.
\end{defn}

\begin{defn}[Variation on an arc] \label{vararc}
Let $S\subset\Complex$ be a simple closed curve, $f:S\to\complex$ a map and
$A=[a,b]$ a subarc of $S$ such that $f(a),f(b)\in T(S)$ and $f(A)\cap A=\0$.
We define the {\em variation of $f$ on $A$ with respect to
$S$}\index{variation!on an arc}, denoted $\var(f,A,S)$, by the following
algorithm:

\begin{enumerate}
\item Let $v\in A$ and let $J_v$ be a junction with $J_v\cap S=\{v\}$.

\item\label{crossings}
{\em Counting crossings:} Consider the set $M=f^{-1}(J_v)\cap [a,b]$. Each time
a point of $f^{-1}(J^+_v)\cap [a,b]$ is immediately followed in $M$, in the
counterclockwise order $<$ on $[a,b]\subset S$, by a point of $f^{-1}(J^i_v)$,
count $+1$ and each time a point of $f^{-1}(J^i_v)\cap [a,b]$ is immediately
followed in $M$  by a point of $f^{-1}(J^+_v)$, count $-1$. Count no other
crossings.

\item
The sum of the crossings found above is the variation $\var(f,A,S)$.
\index{variation@$\var(f,A,S)$}
\end{enumerate}

\end{defn}

Note that $f^{-1}(J^+_v)\cap [a,b]$ and $f^{-1}(J^i_v)\cap [a,b]$ are disjoint
closed sets in $[a,b]$. Hence, in (\ref{crossings}) in the above  definition,
we count only a finite number of crossings and $\var(f,A,S)$ is an integer. Of
course, if $f(A)$ does not meet both $J^+_v$ and $J^i_v$, then $\var(f,A,S)=0$.

If $\al:S\to\complex$ is any map such that $\al|_A=f|_A$ and $\al(S\setminus
(a,b))\cap J_v=\0$, then $\var(f,A,S)=\win(\al,S,v)$. In particular, this
condition is satisfied if $\al(S\sm (a,b))\subset T(S)\setminus\{v\}$. The
invariance of winding number under suitable homotopies implies that the
variation $\var(f,A,S)$ also remains invariant under such homotopies. That is,
even though the specific crossings in (\ref{crossings}) in the algorithm may
change, the sum remains invariant. We will state the  required results about
variation below without proof. Proofs can  be obtained directly by using
the fact that $\var(f,A,S)$ is integer-valued and continuous under suitable
homotopies.

\begin{prop} [Junction Straightening] \label{straightjunction} \label{int}
Let $S\subset\Complex$ be a simple closed curve, $f:S\to\complex$ a map and
$A=[a,b]$ a subarc of $S$ such that $f(a),f(b)\in T(S)$ and $f(A)\cap A=\0$.
Any two junctions $J_v$ and $J_u$  with $u,v\in A$ and $J_w\cap S=\{w\}$ for
$w\in\{u,v\}$ give the same value for  $\var(f,A,S)$.  Hence $\var(f,A,S)$ is
independent of the particular junction used  in Definition~\ref{vararc}.
\end{prop}

The computation of $\var(f,A,S)$ depends only upon the crossings of
the junction $J_v$  coming from a proper compact subarc of the open
arc $(a,b)$. Consequently, $\var(f,A,S)$ remains invariant under
homotopies $h_t$ of $f|_{[a,b]}$ in the complement of $\{v\}$ such
that  $h_t(a),h_t(b) \not\in J_v$ for all $t$. Moreover, the
computation is stable under an isotopy $h_t:J_v\to A\cup
[\complex\sm T(S)]$ that moves the entire junction $J_v$ (even off
$A$), provided that during the isotopy $h_t(v)\not\in f(A)$ and
$f(a),f(b)\not\in h_t(J_v)$ for all $t$.

In case $A$ is an open arc $(a,b)\subset S$ such that $\var(f,\cl{A},S)$ is
defined, it will be convenient to denote $\var(f,\cl{A},S)$ by $\var(f,A,S)$

The following lemma follows immediately from the definition.

\begin{lem} \label{summ} Let $S\subset\Complex$ be a simple closed curve.
Suppose that $a<c<b$ are three points in $S$ such that
$\{f(a),f(b),f(c)\}\subset T(S)$ and $f([a,b])\cap [a,b]=\0$. Then
$\var(f,[a,b],S)=\var(f,[a,c],S)+\var(f,[c,b],S)$.
\end{lem}

\begin{defn}[Variation on a finite union of arcs] \label{partition}
\index{variation!on finite union of arcs} Let $S\subset\Complex$ be a simple
closed curve and $A=[a,b]$ a subcontinuum of $S$  partitioned by  a finite set
$F=\{a=a_0<a_1<\dots<a_n=b\}$ into subarcs. For each $i$ let $A_i=[a_i,a_{i+1}]$. Suppose
that $f$ satisfies $f(a_i)\in T(S)$ and $f(A_i)\cap A_i=\0$ for each $i$. We
define the \emph{variation of $f$ on $A$ with respect to  $S$}, denoted
$\var(f,A,S)$, by
$$\var(f,A,S)=\sum_{i=0}^{n-1} \var(f,[a_i,a_{i+1}],S).$$  In
particular, we include the possibility that  $a_{n}=a_0$ in which
case $A=S$.
\end{defn}

By considering a common refinement of two partitions $F_1$ and
$F_2$ of an arc $A\subset S$ such that $f(F_1)\cup f(F_2)\subset
T(S)$ and satisfying the conditions in Definition~\ref{partition},
it follows from Lemma~\ref{summ} that we get the same value for
$\var(f,A,S)$ whether we use the partition $F_1$ or the partition
$F_2$. Hence, $\var(f,A,S)$ is well-defined. If $A=S$ we denote
\index{variation@$\var(f,S)$}
$\var(f,S,S)$ simply by $\var(f,S)$\index{variation!of a simple closed curve}.

The first main result in Part 1, Theorem~\ref{I=V+1} is that given  a map $f:\C\to\C$, a simple closed curve
$S\subset\C$ and a partition of $S$ into subarcs $A_i$ such that any two meet at most in a common endpoint,
for each $i$ $f(A_i)\cap A_i=\0$ and both endpoints map into $T(S)$,
\[\ind(f,S)=\sum\var(f,A_i)+1.\]

In the first version of this theorem we partition $S$ into finitely many subarcs $A_i$.  We extend this in Section~\ref{Cloops}
 by allowing partitions of $S$ which consist of, possibly countably infinitely many subarcs.
Since in our applications we often assume that we have an invariant
continuum $X$ such that $f$ has no fixed point in $T(X)$ it follows
from Theorem~\ref{fpthm} that, for a sufficiently tight simple
closed curve $S$ around $X$ with $X\subset T(S)$, we must have
$\ind(f,S)=0$. It follows from the above theorem relating index and
variation that for some subarc $A$
 (which is the closure of a component of $S\sm X$ and, hence a crosscut
of $X$), $\var(f,A)<0$. In order to locate this crosscut of negative variation we establish
 Bell's Lollipop
Theorem in Section~\ref{locate}.

\section{Classes of maps}

Cartwright and Littlewood solved the Plane Fixed Point Problem for orientation preserving \emph{homeomorphism} of the plane.
In Section~\ref{soriented} we introduce and study (positively) oriented \emph{maps} of the plane.
We will show in Part 2 that the Plane Fixed Point Problem has a positive solution for the class of positively oriented
maps.
We  show in Section~\ref{soriented} that the class of oriented maps  consists of all compositions of monotone and open perfect maps
of the plane and that all such maps are confluent. In particular, analytic maps are confluent.

Let us begin by listing a few well-known definitions.

\begin{defn}\label{basic-top}
A {\em perfect map} \index{map!perfect}is a closed continuous
function each of whose point inverses is compact. \emph{We will
assume in the remaining sections that all maps of the plane
considered in this memoir are perfect.} Let $X$ and $Y$ be
topological spaces. A map $f:X\to Y$ is {\em monotone}
\index{map!monotone}provided for each continuum $K\subset Y$,
$f^{-1}(K)$ is connected and $f$ is {\em light} \index{map!light}
provided for each point $y\in Y$, $f^{-1}(y)$ is totally
disconnected. A map $f:X\to Y$ is {\em confluent}
\index{map!confluent} provided for each continuum $K\subset Y$ and
each component $C$ of $f^{-1}(K)$, $f(C)=K$ . Every map $f:X\to Y$
between compacta is  the composition $f=l\circ m$ of a a monotone
map $m:X\to Z$ and a light map $l:Z\to Y$ for some compactum $Z$
\cite[Theorem 13.3]{nadl92}. This representation is called the
\emph{monotone-light decomposition of $f$.} \index{monotone-light
decomposition of a map}
\end{defn}

Observe that any confluent map $f$  is onto. It is well-known that
each homeomorphism of the plane is either orientation-preserving or
orientation-reversing. We will establish an appropriate extension of
this result for confluent perfect mappings of the plane
(Theorem~\ref{orient}) by showing that such maps  either preserve or
reverse local orientation. As a consequence it follows that all
perfect and confluent maps of the plane satisfy the Maximum Modulus
Theorem. We will call such maps \emph{positively-} or
\emph{negatively oriented} maps, respectively.

Complex polynomials  $P:\C\to\C$ are prototypes of positively
oriented maps, but positively oriented maps, unlike polynomials, do
not have to be light or open. Observe that even though in some
applications our maps are holomorphic (see Section~\ref{complappl}),
the notion of a positively oriented map is essential  in
Section~\ref{complappl} since it allows for easy local perturbations
(see Lemma~\ref{pararepel}).

\begin{defn}[Degree of $f_p$]
Let $f:U \to \Complex$ be a map from a simply connected domain
$U\subset\complex$ into the plane.  Let $S\subset\complex$ be a positively
oriented simple closed curve in $U$, and $p\in U\setminus f^{-1}(f(S))$ a
point.  Define $f_p:S\to\ucirc$ by \[ f_p(x)=\frac{f(x)-f(p)}{|f(x)-f(p)|}.\]
Then $f_p$ has a well-defined {\em degree}, denoted $\degree(f_p)$.
\index{degreea@$\dg(f_p)$}Note that $\degree(f_p)$ is  the  winding number
$\win(f,S,f(p))$ of $f|_S$ about $f(p)$.
\end{defn}

\begin{defn}
A map $f:U \to \Complex$ from a simply connected domain $U$ is
\emph{positively oriented} (respectively, {\em negatively oriented})
provided for each simple closed curve $S$ in $U$ and each point
$p\in T(S)\setminus f^{-1}(f(S))$, we have that $\degree(f_p)> 0$
($\degree(f_p)< 0$, respectively). \end{defn} \index{map!positively
oriented}\index{map!negatively oriented}
\begin{defn} A perfect surjection $f:\Complex\to\Complex$ is
\emph{oriented} provided for each simple closed curve $S$ and each
$x\in T(S)$, $f(x)\in T(f(S))$.  \index{map!oriented}
\end{defn}

Clearly every  positively oriented and  each negatively oriented map
is oriented. It will follow that all oriented maps satisfy the
Maximum Modulus Theorem~\ref{orient}  (i.e., for every
non-separating continuum $X$, $\partial f(X)\subset f(\partial X)$).
In particular, every positively or negatively oriented map is
oriented.

It is well-known that both open maps and monotone maps (and hence
compositions of such maps) of continua are confluent. It will
follow (Lemma~\ref{cacyclic}) from a result of Lelek and Read
\cite{leleread74} that each perfect, oriented surjection of the plane is the
composition of a monotone map and a light open map.

\section{Partitioning domains}

In Chapter~\ref{partKP} we consider partitions of an open and connected subset $U$ of the sphere
into convex subsets which are contained in round balls. Bell originally did this, using the Euclidean metric
on the plane, for the complement of $X$ in its convex hull in the plane.
 Following Kulkarni and Pinkall \cite{kulkpink94}, we will consider
 $U$ as a subset of the sphere
and we will work with maximal round balls $B\subset U$ in the spherical metric (such balls correspond to either
round balls in the plane
or to half planes). We first specify Kulkarni and Pinkall's result for our situation (see Theorem~\ref{KPthrm}).
 It leads to a partition
of $U$ into pairwise disjoint closed subsets $F_\al$ such for each
$\al$ there exists a unique maximal closed round ball $B_\al$ with
$\Int(B_\al)\cap \partial U=\0$, $|B_\al\cap \partial U|\ge 2$ and
$F_\al\subset B_\al$. In fact, $F_\al$ is the intersection of $U$
with the hyperbolic convex hull of $B_\al\cap \partial U$ in
\emph{the hyperbolic metric on the ball $B_\al$.} Note that  every
chord in the boundary of any partition element $F_\al$ is part of a
round circle. This is the partition of $U$ which is used in Part 2,
Chapter~\ref{outcha}. We show in Section~\ref{kpchords} that the
collection of all chords in the boundaries of all the sets $F_\al$,
called  \emph{$\kp$-chords},
 is sufficiently rich for a satisfactory prime end theory. (Basically most prime ends can be defined through equivalence
 classes of crosscuts which are all $\kp$-chords.)

However, even though in the above version of the Linchpin Theorem
elements of the partition are closed in $U$ and pairwise disjoint and
$U$ is an arbitrary connected open subset of the sphere, it has the
disadvantage that chords in the boundary of the sets $F_\al$ are not
naturally depending on $U$ (they depend only on $B_\al$,  $B_\al\cap
\partial U$ and the hyperbolic metric in $B_\al$). Moreover there may
well be uncountably many distinct elements $F_\al$ which join the
same two accessible points in $\partial U$. In order to avoid this
problem we replace, when $U$ is simply connected, any chord in the
boundary of any set $F_\al$ by the  hyperbolic geodesic (\emph{in
the hyperbolic metric on $U$}) joining the same pair of points (see
Theorem~\ref{Hypmain}).  We will show that the resulting set of
hyperbolic geodesics is a closed lamination of $U$ in the sense of
Thurston \cite{thur85}. This  version of the Linchpin Theorem, which
states that every point in $U$ is either contained in a unique
hyperbolic geodesic $\fg$ in $U$, or in the interior of an unique
hyperbolically convex  gap $\mathfrak{G}$, both of which are
contained in a maximal round ball, is used in \cite{overtymc07,ov09}
to extend a homeomorphism on the boundary of a simply connected
domain over the entire domain. To illustrate the usefulness of these
partitions  we include a simple proof of the Schoenflies Theorem in
Section~\ref{Schoenflies}. However, we will assume the Schoenflies
Theorem throughout this paper.

\chapter{Tools}\label{c:tools}

\section{Stability of Index} \label{compofind}
Let $f:\C\to\C$ be a map.
All basic definitions of index of $f$ on a simple closed curve and variation of $f$ on an arc are contained in
Chapter~\ref{descr1}.
The following standard theorems and observations about the stability of index
under a fixed point free homotopy are consequences of the fact that index is
continuous and integer-valued.

\begin{thm}\label{fpfhomotopy}  Let $h_t:\uc\to\complex$ be a homotopy.
If $f:\cup_{t\in[0,1]} h_t(\uc)\to \complex$ is fixed point free, then
 $\ind(f,h_0)=\ind(f,h_1)$.
\end{thm}

An embedding $g:\uc\to S\subset\complex$ is \emph{orientation
preserving} \index{embedding!orientation preserving}
\index{orientation preserving!embedding} if $g$ is isotopic to the
identity map $id|_{\uc}$. It follows from Theorem~\ref{fpfhomotopy}
that if $g_1,g_2:\uc\to S$ are orientation preserving homeomorphisms
and $f:S\to\complex$ is a fixed point free map, then
$\ind(f,g_1)=\ind(f,g_2)$. Hence we can denote $\ind(f,g_1)$ by
$\ind(f,S)$ \index{index3@$\ind(f,S)$} and if $[a,b]$ is a
positively oriented subarc of $\uc$ we denote the fractional index
$\ind(f,g_1|_{[a,b]})$ by $\ind(f,g_1([a,b]))$,
\index{index0@$\ind(f,A)$} by some abuse of notation when the
extension of $g_1$  over $\uc$ is understood.

\begin{thm} \label{fpfhomotopy2}
Suppose $g:\uc\to\complex$ is a map with $g(\uc)=S$, and $f_1,f_2:S\to\complex$
are homotopic maps such that each level of the homotopy is fixed point free on
$S$. Then $\ind(f_1,g)=\ind(f_2,g)$.
\end{thm}

In particular, if $S$ is a simple closed curve and $f_1,f_2:S\to\complex$ are
maps such that there is a homotopy $h_t:S\to\complex$ from $f_1$ to $f_2$ with
$h_t$ fixed point free on $S$ for each $t\in[0,1]$, then
$\ind(f_1,S)=\ind(f_2,S)$.

\begin{cor} \label{mapinhull}
Suppose $g:\uc\to\complex$ is an orientation preserving embedding with
$g(\uc)=S$, and $f:S\to T(S)$ is a fixed point free map.  Then
$\ind(f,g)=\ind(f,S)=1$.
\end{cor}

\begin{proof}
Since $f(S)\subset T(S)$ which is a disk with boundary $S$ and $f$ has no fixed
point on $S$, there is a  fixed point free homotopy of $f|_S$ to a constant map
$c:S\to \complex$ taking $S$ to a point in $T(S)\setminus S$. By
Theorem~\ref{fpfhomotopy2}, $\ind(f,g)=\ind(c,g)$.  Since $g$ is orientation
preserving it follows from Remark~\ref{remark} (b) that $\ind(c,g)= 1$.
\end{proof}

\begin{thm}  \label{fpthm}
Suppose $g:\uc\to\complex$ is a map with $g(\uc)=S$, and $f:T(S)\to\complex$ is
a map such that $\ind(f,g)\not= 0$, then $f$ has a fixed point in $T(S)$.
\end{thm}

\begin{proof}
Notice that $T(S)$ is a locally connected, non-separating, plane continuum and,
hence, contractible.  Suppose $f$ has no fixed point in $T(S)$. Choose point
$q\in T(S)$.  Let $c:\uc\to \complex$ be the constant map $c(\uc)=\{q\}$.  Let
$H$ be a homotopy from $g$ to $c$ with image in $T(S)$. Since $H$ misses the
fixed point set of $f$, Theorem~\ref{fpfhomotopy} and Remark~\ref{remark} (a)
imply  $\ind(f,g)=\ind(f,c)=0$.
\end{proof}

\section{Index and variation for finite
partitions}\label{indvar}

What links Theorem~\ref{fpthm} with variation is Theorem~\ref{I=V+1} below,
first announced by Bell in the mid 1980's (see also Akis \cite{akis99}).  Our
proof is a modification of Bell's unpublished proof.  We first need a variant
of Proposition~\ref{straightjunction}. Let $r:\complex\to T(\uc)$ be radial
retraction: $r(z)=\frac{z}{|z|}$ when $|z|\geq 1$ and
$r|_{T(\uc)}=id|_{T(\uc)}$.

\begin{lem} [Curve Straightening] \label{cs}
Suppose $f:\uc\to\complex$ is a map with no fixed points on $\uc$.
If $[a,b]\subset \uc$ is a proper subarc with $f([a,b])\cap
[a,b]=\0$, $f((a,b))\subset \complex\sm T(\uc)$ and
$f(\{a,b\})\subset \uc$, then there exists a map
$\tf:\uc\to\complex$ such that $\tf|_{\uc\setminus
(a,b)}=f|_{\uc\setminus (a,b)}$, $\tf|_{[a,b]}:[a,b]\to
(\complex\setminus T(\uc))\cup\{f(a),f(b)\}$ and $\tf|_{[a,b]}$ is
homotopic to $f|_{[a,b]}$ in $\{a,b\}\cup \complex\sm T(S)$ relative
to $\{a,b\}$, so that either  $r|_{\tf([a,b])}$ is   locally
one-to-one or a constant map. Moreover,
$\var(f,[a,b],\uc)=\var(\tf,[a,b],\uc)$.
\end{lem}

Note that if $\var(f,[a,b],\uc)=0$, then $r$ carries $\tf([a,b])$ one-to-one
onto the arc (or point) in $\uc\setminus (a,b)$ from $f(a)$ to $f(b)$. If the
$\var(f,[a,b],\uc)=m>0$, then $r\circ \tf$ wraps the arc $[a,b]$
counterclockwise about $\uc$ so that $\tf([a,b])$ meets each ray in $J_v$ $m$
times. A similar statement holds for negative variation. Note also that it is
possible for index to be defined yet variation not to be defined on a simple
closed curve  $S$. For example, consider the map $z\to 2z$ with $S$ the unit
circle since there is no partition of $S$ satisfying the conditions in
Definition~\ref{vararc}.

\begin{thm}[Index = Variation + 1, Bell] \label{I=V+1} \index{index!Index=Variation+1 Theorem}
Suppose $g:\uc\to\complex$ is an orientation preserving embedding onto a simple
closed curve $S$ and $f:S\to\complex$ is a  fixed point free map.  If
$F=\{a_0<a_1<\dots<a_n\}$ is a partition of $S$ and $A_i=[a_i,a_{i+1}]$ for
$i=0,1,\dots,n$ with $a_{n+1}=a_0$ such that $f(F)\subset T(S)$ and $f(A_i)\cap
A_i=\0$ for each $i$, then
$$\ind(f,S)=\ind(f,g)=\sum_{i=0}^n\var(f,A_i,S)+1=\var(f,S)+1.$$
\end{thm}

\begin{proof}
By an appropriate conjugation of $f$ and $g$, we may assume
without loss of generality that $S=\uc$ and $g=id$.  Let $F$ and
$A_i=[a_i,a_{i+1}]$ be  as in the hypothesis. Consider the
collection of arcs
$$\mc{K}=\{K\subset S\mid \text{$K$ is the
closure of a component of $f^{-1}(f(S)\sm T(S))$}\}.$$ For each
$K\in\mc{K}$, there is an $i$ such that $K\subset A_i$.  Since
$f(A_i)\cap A_i=\0$, it follows from the remark after
Definition~\ref{vararc} that $\var(f,A_i,S)=\sum_{K\subset A_{i},
K\in\mc{K}} \var(f,K,S)$. By the remark following
Proposition~\ref{straightjunction}, we can compute $\var(f,K,S)$
using one fixed junction for $A_i$. It is now clear that there are
at most finitely many  $K\in\mc{K}$ with $\var(f,K,S)\not=0$.
Moreover, the images of the endpoints of each $K$ lie on $S$.

Let $m$ be the cardinality of the set $\mc{K}_f=\{K \in \mc{K} \mid
\var(f,K,S)\not= 0\}$. By the above remarks, $m<\infty$ and  $\mc{K}_f$ is
independent of the partition $F$. We prove the theorem by induction on $m$.

Suppose for a given $f$  we have $m=0$.  Observe that from the
definition of variation and the fact that the computation of
variation is independent of the choice of an appropriate
partition, it follows that,
$$\var(f,S)=\sum_{K\in\mc{K}} \var(f,K,S)=0.$$

We claim that there is a map $f_1:S\to\complex$ with $f_1(S)\subset T(S)$ and a
homotopy $H$ from $f|_S$ to $f_1$ such that each level $H_t$ of the homotopy is
fixed point free and $\ind(f_1,id|_S)=1$.

To see the claim, first apply the Curve Straightening Lemma~\ref{cs}
to each $K\in\mc{K}$ (if there are infinitely many, they form a null
sequence) to obtain a fixed point free homotopy of $f|_S$ to a map
$\tf:S\to\complex$ such that $r|_{\tf(K)}$ is locally one-to-one (or
the constant map) on each $K\in\mc{K}$, where $r$ is radial
retraction of $\complex$ to $T(S)$, and $\var(\tf,K,S)=0$ for each
$K\in\mc{K}$. Let $K$ be in $\mc{K}$ with endpoints $x,y$. Since
$\tf(K)\cap K=\0$ and $\var(\tf,K,S)=0$, $r|_{\tf(K)}$ is
one-to-one, and $r\circ \tf(K)\cap K=\0$.

Define $f_1|_K=r\circ \tf|_K$. Then $f_1|_K$ is fixed point free homotopic to
$f|_K$ (with endpoints of $K$ fixed). Hence, if $K\in\mc{K}$ has endpoints $x$
and $y$, then $f_1$ maps $K$ to the subarc of $S$ with endpoints $f(x)$ and
$f(y)$ such that $K\cap f_1(K)=\0$. Since $\mc{K}$ is a null family, we can do
this for each $K\in\mc{K}$ and set $f_1|_{\uc\setminus
\cup\mc{K}}=f|_{\uc\setminus \cup\mc{K}}$ so that  we obtain the desired
$f_1:S\to\complex$ as the end map of a fixed point free homotopy from $f$ to
$f_1$.  Since $f_1$ carries $S$ into $T(S)$, Corollary~\ref{mapinhull} implies
$\ind(f_1,id|_S)=1$.

Since the homotopy $f\simeq f_1$ is fixed point free, it follows
from Theorem~\ref{fpfhomotopy2} that $\ind(f,id|_S)=1$. Hence, the
theorem holds if $m=0$ for any $f$ and any appropriate partition
$F$.

By way of contradiction suppose the collection $\mathcal{F}$ of all maps $f$ on
$\uc$ which satisfy the hypotheses of the theorem, but not the conclusion is
non-empty. By the above $0<|\mc{K}_f|<\infty$ for each. Let $f\in\mathcal{F}$
be a counterexample for which $m=|\mc{K}_f|$ is minimal.  By modifying $f$, we
will show there exists  $f_1\in\mathcal{F}$ with $|\mc{K}_{f_1}|<m$, a
contradiction.

Choose $K\in\mc{K}$ such that $\var(f,K,S)\not=0$.  Then $K=[x,y]\subset
A_i=[a_i,a_{i+1}]$ for some $i$.  By the Curve Straightening Lemma~\ref{cs} and
Theorem~\ref{fpfhomotopy2}, we may suppose $r|_{f(K)}$ is locally one-to-one on
$K$.  Define a new map $f_1:S\to\complex$ by setting $f_1|_{\ol{S\sm
K}}=f|_{\ol{S\sm K}}$ and setting $f_1|_K$ equal to the linear map taking
$[x,y]$ to the subarc $f(x)$ to $f(y)$ on $S$ missing $[x,y]$.
Figure~\ref{varfig} (left) shows an example of a (straightened) $f$ restricted
to $K$ and the corresponding $f_1$ restricted to $K$ for a case where
$\var(f,K,S)=1$, while Figure~\ref{varfig} (right) shows a case where
$\var(f,K,S)=-2$.

\begin{figure}

\includegraphics{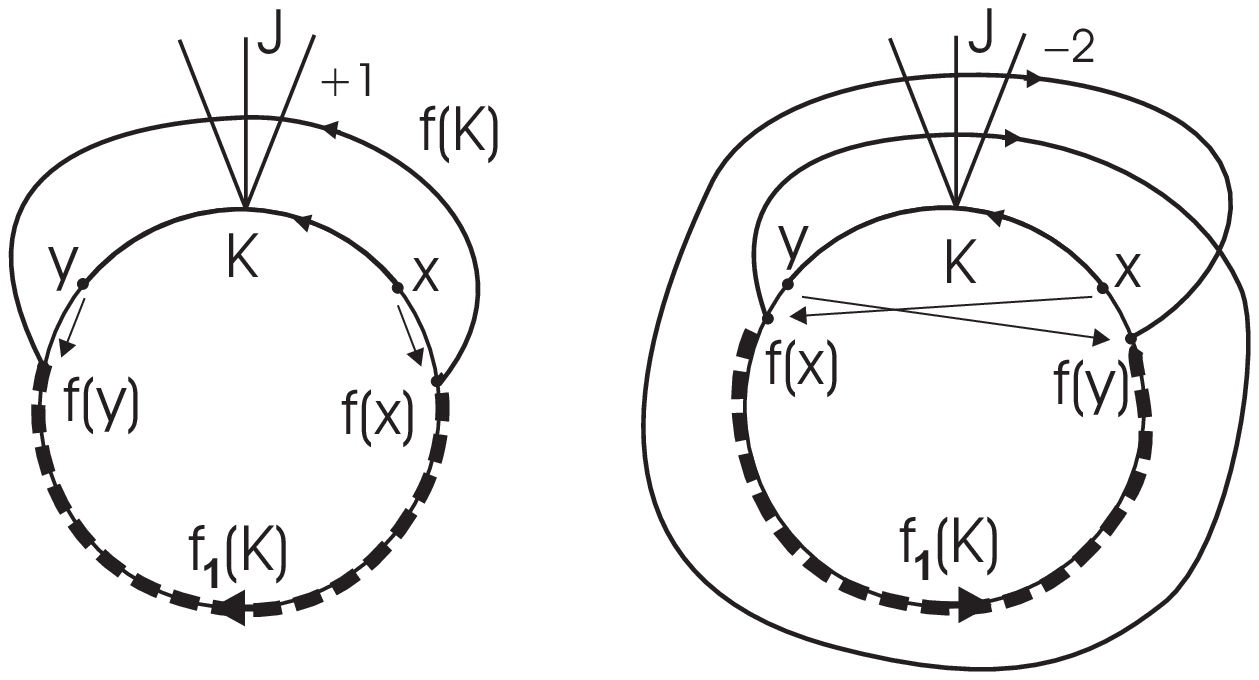}

\caption{Replacing $f:S\to\complex$ by $f_1:S\to\complex$ with one
less subarc of nonzero variation.} \label{varfig}

\end{figure}

Since on $\ol{S\sm K}$, $f$ and $f_1$ are the same map, we have
$$\var(f,S\sm K,S)=\var(f_1,S\sm K,S).$$   Likewise for the
fractional index, $$\ind(f,S\sm K)=\ind(f_1,S\sm K).$$ By definition (refer to
the observation we made in the case $m=0$),
$$\var(f,S)=\var(f,S\sm K,S)+\var(f,K,S)$$
$$\var(f_1,S)=\var(f_1,S\sm K,S)+\var(f_1,K,S)$$
and by Proposition~\ref{fracindex},
$$\ind(f,S)=\ind(f,S\sm K)+\ind(f,K)$$
$$\ind(f_1,S)=\ind(f_1,S\sm K)+\ind(f_1,K).$$ Consequently,
$$\var(f,S)-\var(f_1,S)=\var(f,K,S)-\var(f_1,K,S)$$ and
$$\ind(f,S)-\ind(f_1,S)=\ind(f,K)-\ind(f_1,K).$$

We will now show that the changes in index and variation, going from
$f$ to $f_1$ are the same (i.e., we will show that
$\var(f,K,S)-\var(f_1,K,S)=\ind(f,K)-\ind(f_1,K)$). We suppose first
that $\ind(f,K)=n+\alpha$ for some nonnegative $n\in\mathbb{N}$ and
$0\le\alpha<1$. That is, the vector $f(z)-z$ turns through $n$ full
revolutions counterclockwise and $\alpha$ part of a revolution
counterclockwise as $z$ goes from $x$ to $y$ counterclockwise along
$S$. (See Figure~\ref{varfig} (left) for the case $n=0$ and $\alpha$
about $0.8$.)

 Assume first that $f(x)<x<y<f(y)$ in the circular
order as illustrated in Figure~\ref{varfig} on the left. Then as $z$
goes from $x$ to $y$ counterclockwise along $S$, $f_1(z)$ goes along
$S$ from $f(x)$ to $f(y)$ in the clockwise direction, so $f_1(z)-z$
turns through  $-(1-\alpha)=\alpha-1$ part of a revolution. Hence,
$\ind(f_1,K)=\alpha-1$.  It is easy to see that $\var(f,K,S)=n+1$
and $\var(f_1,K,S)=0$.  Consequently,
$$\var(f,K,S)-\var(f_1,K,S)=n+1-0=n+1$$ and
$$\ind(f,K)-\ind(f_1,K)=n+\alpha-(\alpha-1)=n+1.$$

We assumed that $f(x)<x<y<f(y)$. The cases where $f(y)<x<y<f(x)$ and
$f(x)=f(y)$ (and, hence, $\alpha=0$) are treated similarly. In this
case $f_1$ still wraps around in the positive direction, but the
computations are slightly different: $\var(f,K)=n$,
$\ind(f,K)=n+\al$, $\var(f_1,K)=0$ and $\ind(f_1,K)=\al$.

Thus when $n\geq 0$, in going from $f$ to $f_1$, the change in variation and
the change in index are the same.  However, in obtaining $f_1$ we have removed
one $K\in\mc{K}_f$, reducing the minimal $m=|\mc{K}_f|$ for $f$ by one,
producing a counterexample $f_1$ with $|\mc{K}_{f_1}|=m-1$, a contradiction.

The cases where $\ind(f,K)=n+\alpha$ for negative $n$ and
$0<\alpha<1$ are handled similarly, and illustrated for $n=-2$,
$\alpha$ about $0.4$ and $f(y)<x<y<f(x)$ in Figure~\ref{varfig}
(right).
\end{proof}

\section{Locating arcs of negative variation}\label{locate}
The principal tool in proving  Theorem~\ref{outchannel} (unique outchannel)  is
the following theorem first obtained by Bell (unpublished).  It provides a
method for locating arcs of negative variation on a curve of index zero.

\begin{thm} [Lollipop Lemma, Bell]  \label{lollipop} \index{Lollipop Lemma}
Let $S\subset\complex$ be a simple closed curve  and $f:T(S)\to\complex$ a
fixed point free map. Suppose $F=\{a_0<\dots<a_n<a_{n+1}<\dots<a_m\}$ is a
partition of $S$, $a_{m+1}=a_0$ and $A_i=[a_i,a_{i+1}]$   such that
$f(F)\subset T(S)$ and $f(A_i)\cap A_i=\0$ for $i=0,\dots, m$. Suppose $I$ is
an arc in $T(S)$ meeting $S$ only at its endpoints $a_0$ and $a_{n+1}$. Let
$J_{a_0}$ be a junction in $(\complex\sm T(S))\cup\{a_0\}$ and suppose that
$f(I)\cap (I\cup J_{a_0})=\0$. Let $R=T([a_0,a_{n+1}]\cup I)$ and
$L=T([a_{n+1},a_{m+1}]\cup I)$. Then  one of the following holds:

\begin{enumerate}
\item \label{first}
If $f(a_{n+1})\in R$, then
$$\sum_{i\leq n}\var(f,A_i,S)+1=\ind(f,I\cup [a_0,a_{n+1}]) .$$

\item \label{second}
If $f(a_{n+1})\in L$, then
$$\sum_{i>n}\var(f,A_i,S)+1=\ind(f,I\cup [a_{n+1},a_{m+1}]).$$

\end{enumerate}
\end{thm}

(Note that in (\ref{first}) in effect we compute $\var(f,\bd R)$ but
technically, we have not defined $\var(f,A_i,\bd R)$  since the endpoints of
$A_i$ do not have to map inside $R$ but they do map into  $T(S)$. Similarly in
Case (\ref{second}).)

\begin{proof}

Suppose $f(a_{n+1})\in L$ (the case when $f(a_{n+1})\in R$ can be
treated similarly). Consider the set $C=[a_{n+1},a_{m+1}]\cup I$ (so
$T(C)=L$). We want to construct a map $f':C\to\complex$, fixed point
free homotopic to $f|_C$, that does not change variation on any arc
$A_i$ in $C$ and has the properties listed below.
\begin{enumerate}
\item \label{inL}
$f'(a_i)\in L$,  $f'(A_i)\cap A_i=\0$ for all $n+1\le i\le m$ and
$f'(a_0)\in L$. Hence $\var(f',A_i,C)$ is defined for each $i>n$.

\item \label{varf=g}
$\var(f',A_i,C)=\var(f,A_i,S)$ for all $n+1\le i\le m$.

\item \label{varI}
$f'(I)\cap I=\0$ and $\var(f',I,C)=0$.
\end{enumerate}

Having such a map, it then follows from Theorem~\ref{I=V+1}, that
$$\ind(f',C)=\sum_{i=n+1}^m \var(f',A_i,C)+\var(f',I,C)+1.$$

By Theorem~\ref{fpfhomotopy2} $\ind(f',C)=\ind(f,C)$.  By (\ref{varf=g}) and
(\ref{varI}), $\sum_{i> n} \var(f',A_i,C)+\var(f',I,C)=\sum_{i> n}
\var(f,A_i,S)$ and  the Theorem would follow.

\begin{figure}

\includegraphics{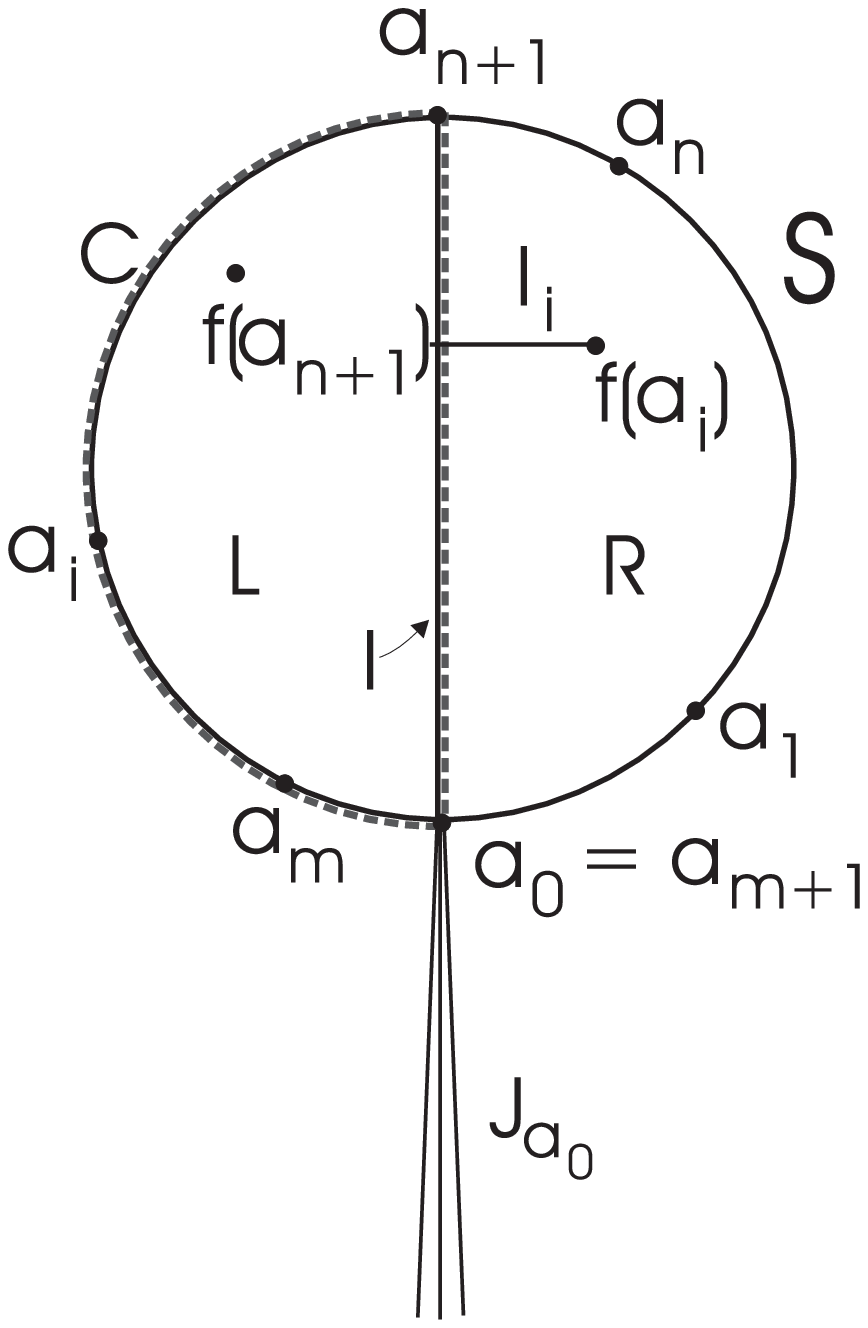}

\caption{Bell's Lollipop.} \label{lollypic}

\end{figure}

It remains to define the map $f':C\to\complex$ with the above properties.  For
each $i$ such that $n+1\le i \le m+1$, chose an arc $I_i$ joining $f(a_i)$ to
$L$ as follows:

\begin{enumerate}
\item[(a)]
If $f(a_i)\in L$, let $I_i$ be the degenerate arc $\{f(a_i)\}$.

\item[(b)]
If $f(a_i)\in R$ and $n+1<i<m+1$, let $I_i$ be an arc in
$R\setminus\{a_0,a_{n+1}\}$ joining $f(a_i)$ to $I$.

\item[(c)]
If $f(a_0)\in R$, let $I_0$ be an arc joining $f(a_0)$ to $L$ such that
$I_0\cap(L\cup J_{a_0})\subset A_{n+1}\setminus \{a_{n+1}\}$.
\end{enumerate}

Let $x_{n+1}=y_{n+1}=a_{n+1}$, $y_0=y_{m+1}\in
I\setminus\{a_0,a_{n+1}\}$ and $x_0=x_{m+1}\in
A_m\setminus\{a_m,a_{m+1}\}$. For $n+1<i<m+1$, let $x_i\in A_{i-1}$
and $y_i\in A_i$ such that $y_{i-1}<x_i<a_i<y_i<x_{i+1}$. For
$n+1<i<m+1$ let $f'(a_i)$ be the endpoint of $I_i$ in $L$,
$f'(x_i)=f'(y_i)=f(a_i)$ and extend $f'$ continuously from
$[x_i,a_i]\cup [a_i,y_i]$ onto $I_i$ and define $f'$ from
$[y_i,x_{i+1}]\subset A_{i}$ onto $f(A_i)$ by
$f'|_{[y_i,x_{i+1}]}=f\circ h_i$, where $h_i:[y_i,x_{i+1}]\to A_i$
is a homeomorphism such that $h_i(y_i)=a_i$ and
$h_i(x_{i+1})=a_{i+1}$. Similarly, define $f'$ on
$[y_0,a_{n+1}]\subset I$ to $f(I)$ by $f'|_{[y_0,a_{n+1}]}=f\circ
h_0$, where $h_0:[y_0,a_{n+1}]\to I$ is an onto homeomorphism such
that $h(a_{n+1})=a_{n+1}$ and extend  $f'$ from
$[x_{m+1},a_0]\subset A_m$ and $[a_0,y_0]\subset I$ onto $I_0$ such
that $f'(x_{m+1})=f'(y_0)=f(a_0)$ and $f'(a_0)$ is the endpoint of
$I_0$ in $L$. To define $f'|_{[a_{n+1},x_{n+2}]}$ let
$h_{n+1}:[y_{n+1},x_{n+2}]\to [a_{n+1},a_{n+2}]$ be a homeomorphism
such that $h_{n+1}(y_{n+1})=a_{n+1}$. Then define $f'(x)$  as
$f\circ h_{n+1}(x)$ for $x\in [y_{n+1},x_{n+2}]$ and
$f'(x)=f(a_{n+1})$ if $x\in [a_{n+1},y_{n+1}]$.

Note that $f'(A_i)\cap A_i=\0$ for $i=n+1,\dots,m$ and $f'(I)\cap
[I\cup J_{a_{0}}]=\0$. To compute the variation of $f'$ on each of
$A_m$ and $I$ we can use the junction $J_{a_{0}}$. Hence
$\var(f',I,C)=0$ and, by the definition of $f'$ on $A_m$,
$\var(f',A_m,C)=\var(f,A_m,S)$. For $i=n+1,\dots,m-1$ we can use the
same junction $J_{v_{i}}$ to compute $\var(f',A_i,C)$ as we did to
compute $\var(f,A_i,S)$. Since $I_i\cup I_{i+1}\subset T(S)\sm A_i$
we have that $f'([a_i,y_i])\cup f'([x_{i+1},a_{i+1}])\subset I_i\cup
I_{i+1}$ misses that junction and, hence, make no contribution to
variation $\var(f',A_i,C)$. Since
$f'^{-1}(J_{v_{i}})\cap[y_i,x_{i+1}]$ is isomorphic to
$f^{-1}(J_{v_{i}})\cap A_i$, $\var(f',A_i,C)=\var(f,A_i,S)$ for
$i=n+1,\dots,m$.

To see that $f'$ is fixed point free homotopic to $f|_C$, note that
we can pull the image of $A_i$ back along the arcs $I_i$ and
$I_{i+1}$ in $R$ without fixing a point of $A_i$ at any level of the
homotopy.
\end{proof}

Note that if $f$ is fixed point free on $T(S)$, then $\ind(f,C)=0$
and the next Corollary follows.

\begin{cor} \label{corlol}
Assume the hypotheses of Theorem~\ref{lollipop}.  Then if
$f(a_{n+1})\in R$ there exists $i\leq n$ such that
$\var(f,A_i,S)<0$. If $f(a_{n+1})\in L$ there exists $i> n$ such
that $\var(f,A_i,S)<0$.
\end{cor}

\section{Crosscuts and bumping arcs}\label{infinitepartitions}\index{variation!for crosscuts}

For the remainder of Chapter 3, we assume that
$f:\complex\to\complex$ takes the continuum $X$ into $T(X)$ with no
fixed points in $T(X)$, and $X$ is minimal with respect to these
properties.

\begin{defn} [Bumping Simple Closed Curve] \label{bumping}\index{bumping!simple closed curve}
A simple closed curve $S$ in $\complex$ which has the property that $S\cap X$
is nondegenerate and $T(X)\subset T(S)$ is said to be a {\em bumping simple
closed curve for $X$}. A subarc $A$ of a bumping simple closed curve, whose
endpoints lie in $X$, is said to be a {\em bumping (sub)arc for $X$}
\index{bumping!arc} or a \emph{link of $S$}\index{link}. Moreover, if $S'$ is any bumping
simple closed curve for $X$ which contains $A$, then $S'$ is said to
\emph{complete} $A$. \index{completing a bumping arc} In fact, an arc $A$ with
endpoints in $X$ which can be completed will be called a \emph{bumping arc of
$X$}.
\end{defn}

Given a positively oriented simple closed curve $S$, we can consider
its positively oriented subarcs denoted by $[a, b]_S$, where $a, b$
are the endpoints of the arc; if the curve is fixed, we simply write
$[a, b]$. Similar notation is used for half-open or open subarcs of
bumping simple closed curves. Often we will fix the choice of links
into which we divide $S$. In general a bumping arc of $X$ may have
points other than its endpoints which belong to $X$ (e.g., if $X$ is
the closed unit disk and the unit circle is divided into several
subarcs then each of them can be considered as a bumping arc of
$X$). By definition, any bumping arc $A$ of $X$ can be extended to a
bumping simple closed curve $S$ of $X$. Hence, every bumping arc has
a well defined natural order $<$ inherited from the positive
circular order of a bumping simple closed curve $S$ containing $A$.
If $a<b$ are the endpoints of $A$, then we will often write
$A=[a,b]$.

A {\em crosscut} \index{crosscut} of $\tU=\rsphere\sm T(X)$ is an
open arc $Q$ lying in $\tU\sm\{\infty\}$ such that $\ol Q$ is an arc
with endpoints $a\not=b\in T(X)$. In this case we will often write
$Q=(a,b)$. (As seems to be traditional, we use ``crosscut of $T(X)$"
interchangeably with ``crosscut of $\tU$.") Evidently, a crosscut of
$\tU$ separates $\tU$ into two disjoint domains, exactly one of
which is unbounded. If $S$ is a bumping simple closed curve so that
$X\cap S$ is nondegenerate, then each component of $S\sm X$ is a
crosscut of $T(X)$. A similar statement holds for a bumping arc $A$.
Given a non-separating continuum $T(X)$, let $A\subset\complex$ be a
crosscut of $U^\infty(X)=\rsphere\sm T(X)$. Given a crosscut $A$ of
$U^\infty(X)$  denote by $\Sh(A)$, the \emph{shadow of}
\index{crosscut!shadow} \index{shadow} \index{shadow@$\Sh(A)$} $A$,
the bounded component of $\complex\setminus [T(X)\cup A]$. Moreover,
suppose that $A$ is a bumping arc of $X$. Then by the \emph{shadow
$\sh(A)$ of $A$}, we mean the union of all bounded components of
$\C\sm (X\cup A)$ (since there may be more than endpoints of $A$ in
$X\cap A$, we should talk about the union of all bounded components
of $\C\sm (X\cup A)$ here).

A variety of tools (such as index, variation, junction) have been
described in previous sections. So far they have been applied to the
properties of maps of the plane restricted to simple closed curves.
Another application can be found in Theorem~\ref{fpthm} where the
existence of a fixed point in the topological hull of a simple
closed curve is established. However we are mostly interested in
studying continua $X$ as described in our Standing Hypotheses. The
following construction shows how the above described tools apply in
this situation.

Since $f$ has no fixed points in $T(X)$ and $X$ is compact, we can choose a
bumping simple closed curve $S$ in a small neighborhood of $T(X)$ such that all
crosscuts in $S\sm X$ are small, have positive distance to their image and so
that $f$ has no fixed points in $T(S)$. Thus, we obtain the following corollary
to Theorem~\ref{fpthm}.

\begin{cor} \label{bumpingscc} Let $f:\C\to\C$ be a map and $X\subset\C$ a subcontinuum with
$f(X)\subset T(X)$ and so that $f|_{T(X)}$ is fixed point free. Then
there is a bumping simple closed curve $S$ for $X$ such that
$f|_{T(S)}$ is fixed point free; hence, by \ref{fpthm},
$\ind(f,S)=0$. Moreover, any bumping simple closed curve $S'$ for
$X$ such that $S'\subset T(S)$ has $\ind(f,S')=0$. Furthermore, any
bumping arc $A$ of $T(X)$ for which $f$ has no fixed points in
$T(X\cup A)$ can be completed to a bumping simple closed curve $S$
for $X$ for which $\ind(f,S)=0$. \end{cor}

The idea of the proofs of a few forthcoming results is that in some cases we
can use the developed tools (e.g., variation) in order to compute out index and
show, relying upon the properties of our maps, that index is \emph{not} equal
to zero thus contradicting Corollary~\ref{bumpingscc}. To implement such a plan
we need to further study properties of variation in the setting described
before Corollary~\ref{bumpingscc}.

\begin{prop} \label{varcross}
 Let $f:\C\to\C$ be a map and $X\subset\C$ a subcontinuum with
$f(X)\subset T(X)$ and so that $f|_{T(X)}$ is fixed point free. In
the situation of Corollary~\ref{bumpingscc}, suppose $A$ is a
bumping subarc for $X$. If $\var(f,A,S)$ is defined for some bumping
simple closed curve $S$ completing $A$, then for any bumping simple
closed curve $S'$ completing $A$, $\var(f,A,S)=\var(f,A,S')$.
\end{prop}

\begin{proof}
Since $\var(f,A,S)$ is defined, $A=\cup_{i=1}^n A_i$, where each
$A_i$ is a bumping arc with $A_i\cap f(A_i)=\0$ and $|A_i\cap
A_j|\le 1$ if $i\ne j$. By the remark following
Definition~\ref{partition}, it suffices to establish the desired
result for each $A_i=A$.   Let $S$ and $S'$ be two bumping simple
closed curves completing $A$ for which variation is defined. Let
$J_a$ and $J_{a'}$ be junctions whereby $\var(f,A,S)$ and
$\var(f,A,S')$ are respectively computed. Suppose first that both
junctions lie (except for $\{a,a'\}$) in $\complex\sm (T(S)\cup
T(S'))$.  By the Junction Straightening
Proposition~\ref{straightjunction}, either junction can be used to
compute either variation on $A$, so the result follows. Otherwise,
at least one junction is not in $\complex\sm (T(S)\cup T(S'))$.  But
both junctions are in $\complex\sm T(X\cup A)$. Hence, we can find
another bumping simple closed curve $S''$ such that $S''$ completes
$A$, and both junctions lie in $(\complex\sm T(S''))\cup\{a,a'\}$.
Then by the Propositions~\ref{straightjunction} and the definition
of variation, $\var(f,A,S)=\var(f,A,S'')=\var(f,A,S')$.
\end{proof}

It follows from Proposition~\ref{varcross} that  variation on a
crosscut $Q$, with $\ol{Q}\cap f(\ol{Q})=\0$, of $T(X)$ is
independent of the bumping simple closed curve $S$ for $T(X)$ of
which $Q$ is a subarc and is such that $\var(f,Q,S)$ is defined.
Hence, given a bumping arc $A$ of $X$,  we can denote $\var(f,A,S)$
 by $\var(f,A,X)$ or simply by $\var(f,A)$ \index{varfA@$\var(f,A)$} when $X$ is understood.
The figure illustrates how variation is computed.

\begin{figure}\label{varill}
\begin{picture}(307,231)
\put(0,0){\scalebox{0.7}{\includegraphics{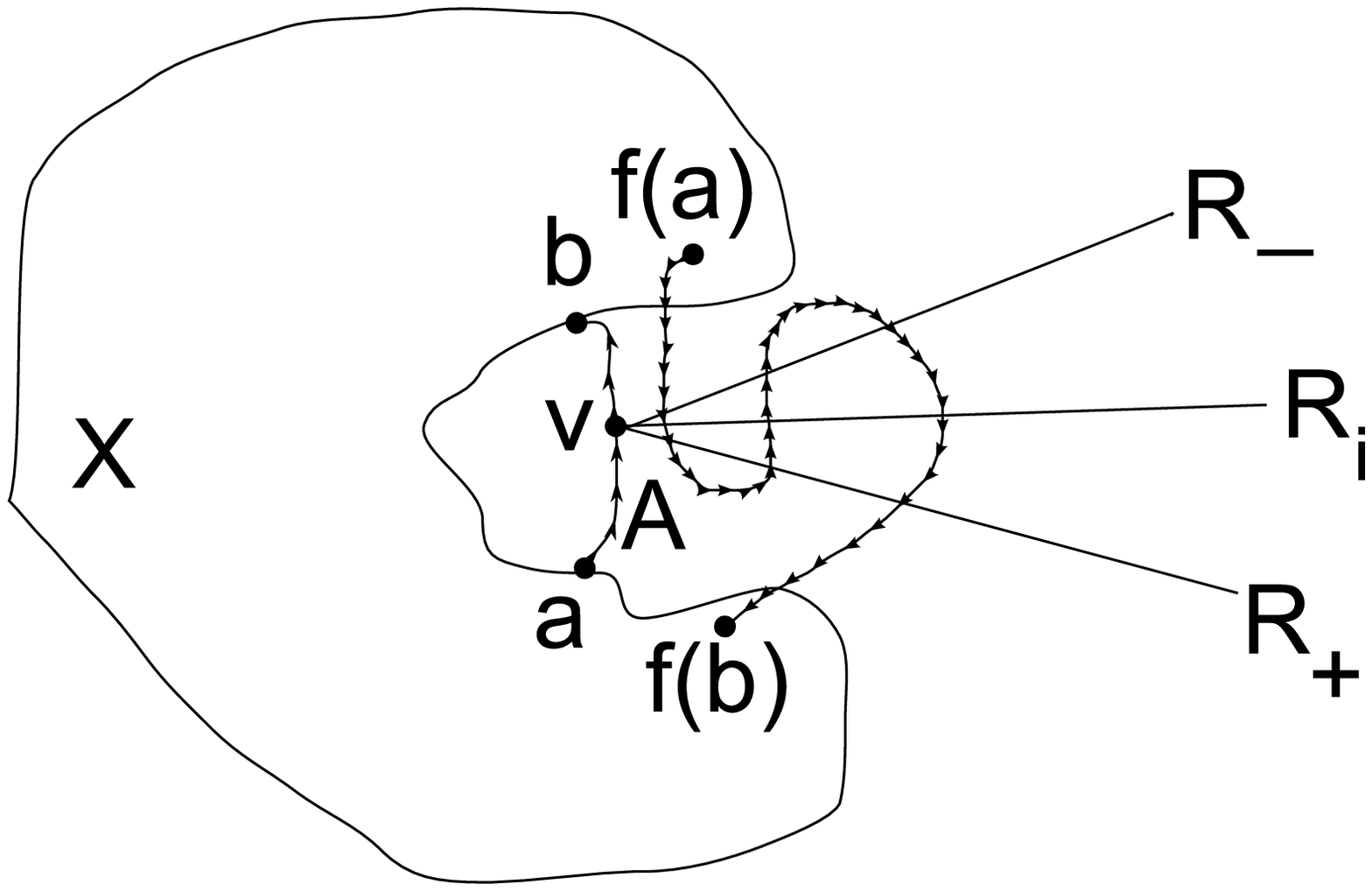}}}
\end{picture}
\caption{$\mathrm{var}(f, A)=-1+1-1=-1$.}
\end{figure}

The following proposition follows from Corollary~\ref{bumpingscc},
Proposition~\ref{varcross} and Theorem~\ref{I=V+1}.

\begin{prop} \label{crossunder}Let $f:\C\to\C$ be a map and $X$ a subcontinuum of $\C$
so that $f(X)\subset T(X)$ and $f$ has no fixed points in $T(X)$.
Suppose $Q$ is a crosscut of $T(X)$ such that $f$ is fixed point free on
$T(X\cup Q)$ and $f(\ol{Q})\cap \ol{Q}=\0$. Suppose $Q$ is replaced by a
bumping subarc $A$ with the same endpoints such that $Q\cup T(X)$ separates
$A\sm X$ from $\infty$ and each component $Q_i$ of $A\sm X$ is a crosscut such
that $f(\cl{Q_i})\cap \cl{Q_i}=\0$.  Then
$$\var(f,Q,X)=\sum_{i}
\var(f,Q_i,X)=\var(f,A,X).$$
\end{prop}

\section{Index and Variation for Carath\'eodory Loops}\label{Cloops}

We extend the definitions of index and variation
 to {\em Carath\'eodory loops}. In particular, if $g:\uc\to g(\uc)=S$ is a continuous extension of a Riemann
map $\psi:\disk^\infty\to\rsphere\sm T(g(\uc))$, then $g$ is a
Carath\'eodory loop, where $\disk^\infty=\{z\in\rsphere\mid |z|>1\}$
\index{deltai@$\disk^\infty$} is the ``unit disk'' about $\infty$.

\begin{defn}[Carath\'eodory Loop]\index{Carath\'eodory Loop}
Let $g:\uc\to\complex$ such that $g$ is continuous and has a
continuous extension $\psi:\ol{\rsphere\sm T(\uc)}\to\ol{\rsphere\sm
T(g(\uc))}$ such that $\psi|_{\complex\setminus T(\uc)}$ is an
orientation preserving homeomorphism from $\complex\setminus T(\uc)$
onto $\complex\setminus T(g(\uc))$. We call $g$ (and loosely,
$S=g(\uc)$), a {\em Carath\'eodory loop}.
\end{defn}

Let $g:\uc\to\complex$ be a Carath\'eodory loop and let $f:g(\uc)\to\complex$
be a fixed point free map.  In order to define variation of $f$ on $g(\uc)$, we
do the partitioning in $\uc$ and transport it to the Carath\'eodory loop
$S=g(\uc)$. An {\em allowable} partition \index{allowable partition} of $\uc$
is a set $\{a_0<a_1<\dots<a_n\}$ in $\uc$ ordered counterclockwise, where
$a_0=a_n$ and $A_i$ denotes the counterclockwise interval $[a_{i},a_{i+1}]$,
such that for each $i$, $f(g(a_i))\in T(g(\uc))$ and $f(g(A_i))\cap g(A_i)=\0$.
Variation $\var(f,A_i,g(\uc))=\var(f,A_i)$ on each path $g(A_i)$ is then
defined exactly as in Definition~\ref{vararc}, except that the junction (see
Definition~\ref{junction}) is chosen so that the vertex $v\in g(A_i)$ and
$J_v\cap T(g(\uc))\subset \{v\}$, and the crossings of the junction $J_v$  by
$f(g(A_i))$ are counted (see Definition~\ref{vararc}). Variation on the whole
loop, or an allowable subarc thereof, is defined just as in
Definition~\ref{partition}, by adding the variations on the partition elements.
At this point in the development, variation is defined only relative to the
given allowable partition $F$ of $\uc$ and the parameterization $g$ of $S$:
$\var(f,F,g(\uc))$.

Index on a Carath\'eodory loop $S$ is defined exactly as in
Section~\ref{defindex} with $S=g(\uc)$ providing the
parameterization of $S$. Likewise, the definition of fractional
index and Proposition~\ref{fracindex} apply to Carath\'eodory
loops.

Theorems~\ref{fpfhomotopy}, \ref{fpfhomotopy2}, Corollary~\ref{mapinhull}, and
Theorem~\ref{fpthm} (if $f$ is also defined on $T(S)$) apply to Carath\'eodory
loops.  It follows that index on a Carath\'eodory loop $S$ is independent of
the choice of parameterization $g$. The Carath\'eodory loop $S$ is
approximated, under small homotopies, by simple closed curves $S_i$. Allowable
partitions of $S$ can be made to correspond to allowable partitions of $S_i$
under small homotopies. Since variation and index are invariant under suitable
homotopies (see the comments after Proposition~\ref{straightjunction}) we have
the following theorem.

\begin{thm} \label{CaraVI} \index{index!I=V+1 for Carath\'eodory Loops}
Suppose $S=g(\uc)$ is a parameterized Carath\'eodory loop in $\complex$ and
$f:S\to\complex$ is a fixed point free map. Suppose variation of $f$ on
$\uc=A_0\cup\dots\cup A_n$ with respect to $g$ is defined for some partition
$A_0\cup\dots\cup A_n$ of $\uc$. Then
$$\ind(f,g)=\sum_{i=0}^n\var(f,A_i,g(\uc))+1.$$
\end{thm}

\section{Prime Ends}

Prime ends provide a way of studying the approaches to the boundary of a
simply-connected plane domain with non-degenerate boundary.  See
\cite{colllohw66} or \cite{miln00} for an analytic summary of the topic and
\cite{urseyoun51} for a more topological approach. We will be interested in the
prime ends of $\tU=\rsphere\sm T(X)$.  Recall that
$\disk^\infty=\{z\in\rsphere\mid |z|>1\}$ is the ``unit disk about $\infty$."
The Riemann Mapping Theorem guarantees the existence of a conformal map
$\phi:\disk^\infty\to \tU$ taking $\infty\to\infty$, unique up to the argument
of the derivative at $\infty$. Fix such a map $\phi$. We identify $\uc=
\bd\disk^\infty$ with $\real/\zed$ and identify points $e^{2\pi i t}$ in
$\bd\disk^\infty$ by their argument $t \pmod{1}$. Crosscut and shadow were
defined in Section~\ref{infinitepartitions}.

\begin{defn}[Prime End]
A {\em chain of crosscuts}\index{chain of crosscuts} \index{prime
end}is a sequence $\{Q_i\}_{i=1}^\infty$ of crosscuts of $\tU$ such
that for $i\not= j$, $\ol{Q}_i\cap \ol{Q}_j=\0$, $\mrm{diam}(Q_i)\to
0$, and for all $j>i$, $Q_i$ separates $Q_j$ from $\infty$ in $\tU$.
Hence, for all $j>i$, $Q_j\subset \Sh(Q_i)$. Two chains of crosscuts
are said to be {\em equivalent} \index{chain of
crosscuts!equivalent} if and only if it is possible to form a
sequence of crosscuts by selecting alternately a crosscut from each
chain so that the resulting sequence of crosscuts is again a chain.
A {\em prime end} $\mc{E}$ is an equivalence class of chains of
crosscuts.\index{Et@$\mc{E}_t$}\end{defn}

If $\{Q_i\}$ and $\{Q'_i\}$  are equivalent  chains of crosscuts of $\tU$, it
can be shown that $\{\phi^{-1}(Q_i)\}$ and $\{\phi^{-1}(Q'_i)\}$ are equivalent
chains of crosscuts of $\disk^\infty$  each of which converges to the same
unique point $e^{2\pi it}\in \uc= \bd\disk^\infty$, $t\in[0,1)$, independent of
the representative chain. Hence, we denote by $\mc{E}_t$ the prime end of $\tU$
defined by $\{Q_i\}$.

\begin{defn}[Impression and Principal Continuum] \index{prime end!impression}\index{prime end!principal continuum}
\index{impression} \index{principal continuum}
Let $\mc{E}_t$ be a prime end of
$\tU$ with defining chain of crosscuts $\{Q_i\}$. The set
\index{imet@$\im(\mc{E}_t)$}
$$\im(\mc{E}_t)=\bigcap_{i=1}^\infty \cl{\Sh(Q_i)}$$ is a subcontinuum
of $\bd \tU$ called the {\em impression} of $\mc{E}_t$.  The
set \index{pret@$\pr(\mc{E}_t)$} $$\pr(\mc{E}_t)=\{z\in\bd \tU\mid \text{for some chain
$\{Q'_i\}$ defining $\mc{E}_t$, $Q'_i\to z$}\}$$ is a continuum
called the {\em principal continuum} of $\mc{E}_t$.
\end{defn}

For a prime end $\mc{E}_t$, $\pr(\mc{E}_t)\subset \im(\mc{E}_t)$,
possibly properly.    We will be interested in the existence of
prime ends $\mc{E}_t$ for which $\pr(\mc{E}_t)=\im(\mc{E}_t)=\bd
\tU$.

\begin{defn} [External Rays] \index{external ray}  \index{Rt@$R_t$}
Let $t\in [0,1)$ and define
$$R_t=\{z\in\complex\mid z=\phi(re^{2\pi it}),1<r<\infty\}.$$ We
call $R_t$ the {\em external ray (with argument $t$)}. If $x\in R_t$ then the
$(X,x)$-\emph{end of} $R_t$ is the  bounded component $K_x$ of
\index{external ray!end of}
$R_t\sm\{x\}$.\end{defn}

In this case $X$ is a continuum, $U^\infty(X)$ is simply connected,
the external rays $R_t$ are all smooth and pairwise disjoint.
Moreover, for each $x\in U^\infty(X)$ there exists a unique $t$ such that $x\in R_t$.

\begin{defn}[Essential crossing]\label{essential}\index{external ray!essential crossing}
\index{essential crossing}
An external ray $R_t$ is said to \emph{cross} a crosscut $Q$ \emph{essentially}
if and only if there exists $x\in R_t$ such that the $(T(X),x)$-end of $R_t$ is
contained in the bounded complementary domain of $T(X)\cup Q$. In this case we
will also say that $Q$ crosses $R_t$ essentially.
\end{defn}

The results listed below are known.

\begin{prop}[\mbox{\cite{colllohw66}}] \label{trans}
Let $\mc{E}_t$ be a prime end of $\tU$. Then $\pr(\mc{E}_t)=\cl{R_t}\sm R_t$.
Moreover, for each $1<r<\infty$ there is a  crosscut $Q_r$ of $\tU$ with
$\{\phi(re^{2\pi it})\}=R_t\cap Q_r$ and $\dm(Q_r)\to 0$ as $r\to 1$ and such
that $R_t$ crosses $Q_r$ essentially.
\end{prop}

\begin{defn}[Landing Points and Accessible Points]  \index{accessible point}\index{external ray!landing point}
If $\pr(\mc{E}_t)=\{x\}$, then we say $R_t$ {\em lands} on $x\in T(X)$ and $x$
is the {\em landing point} \index{landing point} of $R_t$.  A point $x\in \bd
T(X)$ is said to be {\em accessible} (from $\tU$) if and only if there is an arc in
$\tU\cup\{x\}$ with $x$ as one of its endpoints. \end{defn}

\begin{prop}
A point $x\in \bd T(X)$ is accessible if and only if $x$ is the landing point of some
external ray $R_t$. \end{prop}

\begin{defn}[Channels]\index{channel}\index{prime end!channel}
A prime end $\mc{E}_t$ of $\tU$ for which $\pr(\mc{E}_t)$ is nondegenerate is
said to be a {\em channel in} $\bd \tU$ (or in $T(X)$). If moreover
$\pr(\mc{E}_t)=\bd \tU=\bd T(X)$, we say $\mc{E}_t$ is a {\em dense} channel.
\index{channel!dense} A crosscut $Q$ of $\tU$ is said to {\em cross} the
channel $\mc{E}_t$ if and only if  $R_t$ crosses $Q$ essentially.
\end{defn}

When $X$ is locally connected, there are no channels, as the
following classical theorem proves.  In this case, every prime end
has degenerate principal set and degenerate impression.

\begin{thm} [Carath\'eodory] \label{carath} $X$ is locally connected if and only if the
Riemann map $\phi:\disk^\infty\to \tU=\rsphere\sm T(X)$
taking $\infty\to\infty$ extends continuously to
$\uc=\bd\disk^\infty$.
\end{thm}

\section{Oriented maps}\label{soriented}
Basic notions of (positively) oriented and confluent maps are defined in Chapter~\ref{descr1}.
 In this section we study (positively) oriented maps and we will establish that it is a natural class of plane maps
 which are the proper generalization of an orientation preserving homeomorphism
  of the plane.
 The following
lemmas are in preparation for the proof of Theorem~\ref{orient}.

\begin{lem}\label{lorient}
Suppose $f:\Complex\to\Complex$ is a perfect surjection. Then $f$ is confluent
if and only if $f$ is oriented.
\end{lem}

\begin{proof}
Suppose that $f$ is oriented. Let $A$ be an arc in $\Complex$ and let $C$ be a
component of $f^{-1}(A)$. Suppose that $f(C)\not=A$. Let $a\in A\setminus
f(C)$. Since $f(C)$ does not separate $a$ from infinity, we can choose a simple
closed curve $S$ with $C\subset T(S)$, $S\cap f^{-1}(A)=\0$ and $f(S)$  so
close to $f(C)$ that $f(S)$ does not separate $a$ from $\infty$.  Then
$a\not\in T(f(S))$. Since $f$ is oriented, $f(C)\subset T(f(S))$. Hence there
exists a $y\in A\cap f(S)$. This contradicts the fact that $A\cap f(S)=\0$.
Thus $f(C)=A$.

Now suppose that $K$ is an arbitrary continuum in $\Complex$ and let $L$ be a
component of $f^{-1}(K)$. Let $x\in L$ and let $A_i$ be a sequence of arcs in
$\Complex$ such that $\lim A_i=K$ and $f(x)\in A_i$ for each $i$. Let $M_i$ be
the component of $f^{-1}(A_i)$ containing the point $x$. By the previous
paragraph $f(M_i)=A_i$. Since $f$ is perfect, $M=\limsup M_i\subset L$ is a
continuum and $f(M)=K$. Hence $f$ is confluent.

Suppose next that $f:\Complex\to\Complex$ is not oriented. Then
there exists a simple closed curve $S$ in $\Complex$ and $p\in T(S)$
such that  $f(p)\not\in T(f(S))$. Let $L$ be a half-line with
endpoint  $f(p)$ running to infinity in $\Complex\setminus f(S)$.
Let $L^*$ be an arc in $L$ with endpoint $f(p)$ and diameter greater
than the diameter of the continuum $f(T(S))$. Let $K$ be the
component of $f^{-1}(L^*)$ which contains $p$. Then $K\subset T(S)$,
since $p\in T(S)$ and $L\cap f(S)=\0$. Hence, $f(K)\not=L^*$, and so
$f$ is not confluent.   \end{proof}

\begin{lem}\label{finite}
Let $f:\Complex\to\Complex$ be a light, open, perfect surjection. Then there
exists an integer $k$ and a finite subset $B\subset\Complex$ such that $f$ is a
local homeomorphism at each point of $\Complex\setminus B$, and for each point
$y\in\Complex\setminus f(B)$, $|f^{-1}(y)|=k$.
\end{lem}

\begin{proof}
Let $\sphere$ be the one point compactification of $\Complex$.  Since $f$ is
perfect, we can extend $f$ to a map of $\sphere$ onto $\sphere$  so that
$f^{-1}(\infty)=\infty$. By abuse of notation we also denote the extended map
by $f$. Then $f$ is a light open mapping of the compact $2$-manifold $\sphere$.
The result now follows from a theorem of Whyburn \cite[X.6.3]{whyb42}.
\end{proof}

The following is a special case, for oriented perfect maps, of the
monotone-light factorization theorem. A non-separating plane
continuum is said to be {\em acyclic}. \index{acyclic}

\begin{lem}\label{cacyclic}
Suppose that $f:\Complex\to\Complex$ is an oriented, perfect map. It follows
that $f=g\circ h$, where $h:\Complex\to\Complex$ is a monotone perfect
surjection with acyclic fibers and $g:\Complex\to\Complex$ is a light, open
perfect surjection.
\end{lem}

\begin{proof} As above, $f$ extends to a map of the sphere such that
$f(\infty)=f^{-1}(\infty)=\infty$.  By the monotone-light
factorization theorem \cite[Theorem 13.3]{nadl92}, $f=g\circ h$,
where $h:\Complex\to X$ is monotone, $g:X\to\Complex$ is light, and
$X$ is the quotient space obtained from $\Complex$ by identifying
each component of $f^{-1}(y)$ to a point for each $y\in\Complex$.
Let $y\in \Complex$ and let $C$ be a component of $f^{-1}(y)$. If
$C$ were to separate $\Complex$, then $f(C)=y$ would be a point
while $f(T(C))$ would be a non-degenerate continuum. Choose an arc
$A\subset \Complex\setminus \{y\}$ which meets both $f(T(C))$ and
its complement and let $x\in T(C)\setminus C$ such that $f(x)\in A$.
If $K$ is the component of $f^{-1}(A)$ which contains $x$, then
$K\subset T(C)$. Hence $f(K)$ cannot map onto $A$ contradicting the
fact that $f$ is confluent. Thus, for each $y\in\Complex$, each
component of $f^{-1}(y)$ is acyclic.

By Moore's Plane Decomposition Theorem \cite{dave86}, $X$ is homeomorphic to
$\Complex$. Since $f$ is confluent, it is easy to see that $g$ is confluent. By
a theorem of Lelek and Read~\cite{leleread74} $g$ is open since it is confluent
and light (also see \cite[Theorem 13.26]{nadl92}).  Since $h$ and $g$ factor
the perfect map $f$ through a Hausdorff space $\Complex$, both $h$ and $g$ are
perfect \cite[3.7.5]{enge89}.
\end{proof}

Below $\bd Z$ means the boundary of the set $Z$.

\begin{thm}[Maximum Modulus Theorem] \label{orient}\index{Maximum Modulus Theorem}
Suppose that $f:\Complex\to\Complex$ is a perfect surjection. Then the following are
equivalent:

\begin{enumerate}

\item\label{pnorient}
$f$ is either  positively or negatively oriented.

\item \label{iorient}$f$ is oriented.

\item\label{conf}  $f$ is confluent.
\end{enumerate}
Moreover, if $f$ is oriented, then for any non-separating continuum $X$,
$\bd f(X)\subset f(\bd X)$.
\end{thm}

\begin{proof}
It is clear that (\ref{pnorient}) implies (\ref{iorient}). By
Lemma~\ref{lorient} every oriented map is confluent. Hence suppose that
$f:\Complex\to\Complex$ is a perfect confluent map. By Lemma~\ref{cacyclic},
$f=g\circ h$, where $h:\Complex\to\Complex$ is a monotone perfect  surjection
with acyclic fibers and $g:\Complex\to\Complex$ is a light, open perfect
surjection. By Stoilow's Theorem \cite{whyb64} there exists a homeomorphism
$j:\Complex\to\Complex$ such that $g\circ j$ is an analytic surjection. Then
$f=g\circ h=(g\circ j)\circ (j^{-1}\circ h)$. Since $k=j^{-1}\circ h$ is a
monotone surjection of $\Complex$ with acyclic fibers, it is a near
homeomorphism \cite[Theorem 25.1]{dave86}. That is, there exists a sequence
$k_i$ of homeomorphisms of $\Complex$ such that $\lim k_i=k$. We may assume
that all of the $k_i$ have the same orientation.

Let $f_i=(g\circ j)\circ k_i$, $S$ a simple closed curve in the domain of $f$
and $p\in T(S)\setminus f^{-1}(f(S))$. Note that $\lim f^{-1}_i(f_i(S))\subset
f^{-1}(f(S))$. Hence $p\in T(S)\setminus f^{-1}_i(f_i(S))$ for $i$ sufficiently
large. Moreover, since $f_i$ converges to $f$, $f_i|_S$ is homotopic to $f|_S$
in the complement of $f(p)$ for $i$ large. Thus for large $i$,
$\degree((f_i)_p)=\degree(f_p)$, where
\[(f_i)_p(x)=\frac{f_i(x)-f_i(p)}{|f_i(x)-f_i(p)|}
\text{ and } f_p(x)=\frac{f(x)-f(p)}{|f(x)-f(p)|}.\] Since $g\circ
j$ is an analytic map, it is positively oriented and we conclude
that $\degree((f_i)_p)=\degree(f_p)>0$ if $k_i$ is orientation
preserving and $\degree((f_i)_p)=\degree(f_p)<0$ if $k_i$ is
orientation reversing. Thus, $f$ is positively oriented if each
$k_i$ is orientation-preserving and $f$ is negatively oriented if
each $k_i$ is orientation-reversing.

Suppose that $X$ is a non-separating continuum and $f$ is oriented.
Let $y\in\bd f(X)$. Choose $y_i\in\bd f(X)$ and rays $R_i$, joining
$y_i$ to $\infty$ such that $R_i\cap f(X)=\{y_i\}$ and $\lim y_i=y$.
Choose $x_i\in X$ such that $f(x_i)=y_i$. Since $f$ is confluent,
there exist closed and connected sets $C_i$, joining $x_i$ to
$\infty$ such that $C_i\cap X\subset f^{-1}(y_i)$. Hence there exist
$x'_i\in f^{-1}(y_i)\cap \bd X$. We may assume that $\lim
x'_i=x_\infty\in\bd X$ and $f(x_\infty)=y$ as desired.
\end{proof}

We shall need the following three results in the
next section.

\begin{lem}\label{acyclic}
Let $X$ be a plane continuum and $f:\Complex\to\Complex$ a perfect,
surjective map such that $f^{-1}(T(X))= T(X)$ (i.e., $T(X)$ is fully
invariant) and $f|_{\Complex\setminus T(X)}$ is confluent. Then for
each $y\in\Complex\setminus T(X)$,  each component of $f^{-1}(y)$ is
acyclic.
\end{lem}

\begin{proof}
Suppose there exists $y\in\Complex\setminus T(X)$ such that some
component $C$ of $f^{-1}(y)$ is not acyclic. Then there exists $z\in
T(C)\setminus [f^{-1}(y)\cup T(X)]$. Then $T(X)\cup \{y\}$ does not
separate $f(z)$ from infinity in $\Complex$. Let $L$ be a ray in
$\Complex\setminus[T(X)\cup \{y\}]$ from $f(z)$ to infinity. Then
$L=\cup L_i$, where each $L_i\subset L$ is an arc with endpoint
$f(z)$. For each $i$ the component $M_i$ of $f^{-1}(L_i)$ containing
$z$ maps onto $L_i$. Then $M=\cup M_i$ is a connected closed subset
in $\Complex\setminus f^{-1}(y)$ from $z$ to infinity. This is a
contradiction since $z$ is contained in a bounded complementary
component of $f^{-1}(y)$.
\end{proof}

\begin{thm}\label{confeq}
Let $X$ be a plane continuum and $f:\Complex\to\Complex$ a perfect,
surjective map such that $f^{-1}(T(X))=T(X)$ and
$f|_{\Complex\setminus T(X)}$ is confluent. If $A$ and $B$ are
crosscuts of $T(X)$ such that $B\cup X$ separates $A$ from $\infty$
in $\Complex$, then $f(B)\cup T(X)$ separates $f(A)\setminus f(B)$
from $\infty$.
\end{thm}

\begin{proof}
Suppose not. Then there exists a half-line  $L$ joining $f(A)$ to infinity in
$\Complex\setminus (f(B)\cup T(X))$. As in the proof of Lemma~\ref{acyclic},
there exists a closed and connected set $M\subset \Complex\setminus (B\cup X)$
joining $A$ to infinity, a contradiction.
\end{proof}

\begin{prop}\label{ray} Under the conditions of
Theorem~\ref{confeq}, if $L$ is a ray irreducible from $T(X)$ to
infinity, then each component of $f^{-1}(L)$  is closed in $\complex\sm X$
and is a  connected set from $X$ to
infinity.
\end{prop}

\section{Induced maps of prime ends}

Suppose that $f:\Complex\to\Complex$ is an oriented perfect
surjection and $f^{-1}(Y)=X$, where $X$ and $Y$ are acyclic continua
and $Y$ has no cutpoints.    We will show that in this case the map
$f$ induces a confluent map $F$ of the circle of prime ends of $X$
to the circle of prime ends of $Y$. This result was announced by
Mayer in the early 1980's but never appeared in print. It was also
used (for homeomorphisms) by Cartwright and Littlewood in
\cite{cartlitt51}. There are easy counterexamples that show if $f$
is not confluent then it may not induce a continuous function
between the circles of prime ends. For example, if $Y=\disk$, $X$ is
the union of the unit disk and a copy of a half ray $R$ which
spirals to the unit circle and $f$ is radial projection of $R$ onto
the unit circle, then $f$ can be extended to a perfect map $F$ of
the plane so that $F^{-1}(Y)=X$ but $F$ does not induce a continuous
function from  the circle of prime ends of $X$ to the circle of
prime ends of $Y$.

\begin{thm}\label{tinduced}\index{map!on circle of prime ends}
Let $X$ and $Y$ be non-degenerate acyclic plane continua and
$f:\Complex\to\Complex$ a perfect map such that:

\begin{enumerate}
\item \label{cut} $Y$ has no cutpoint,

\item \label{complement}
$f^{-1}(Y)=X$ and
\item \label{sep}
$f|_{\Complex\setminus X}$ is confluent.
\end{enumerate}

Let $\varphi: \disk^\infty\to \sphere\setminus X$ and $\psi:\disk^\infty\to
\sphere\setminus Y $ be conformal mappings. Define
$\hat{f}:\disk^\infty\to\disk^\infty$ by $\hat{f}=\psi^{-1} \circ f \circ
\varphi$.

Then $\hat{f}$ extends to a map
$\bar{f}:\ol{\disk^\infty}\to\ol{\disk^\infty}$. Moreover,
$\bar{f}^{-1}(\uc)=\uc$ and $F=\bar{f}|_{\uc}$ is a confluent map.
\end{thm}

\begin{proof}
Note that $f$ takes accessible points of $X$ to accessible points of $Y$. For
if $P$ is a path in $[\Complex\setminus X]\cup \{p\}$ with endpoint $p\in X$,
then by (\ref{complement}), $f(P)$ is a path in $[\Complex\setminus Y]\cup
\{f(p)\}$ with endpoint $f(p)\in Y$.

Let $A$ be a crosscut of $X$ such that the diameter of $f(A)$ is
less than half of the diameter of $Y$ and let $U$ be the bounded
component of $\Complex\setminus (X\cup A)$. Let the  endpoints of
$A$ be $x,y\in X$ and suppose that $f(x)=f(y)$. If  $x$ and $y$ lie
in the same component of $f^{-1}(f(x))$   then each crosscut
$B\subset U$ of $X$ is mapped to a generalized return cut of $Y$
based at $f(x)$ (i.e., by (1) and (3) $f(\ol{U})\cap Y=f(x)$ and the
endpoints of $B$ map to $f(x)$). Note that in this case by (1),
$\bd{f(U)}\subset f(A)\cup \set{f(x)}$.

Now suppose that $f(x)=f(y)$ and $x$ and $y$ lie in distinct components of
$f^{-1}(f(x))$. Then by unicoherence of $\Complex$, $\bd{U}\subset A\cup X$ is
a connected set and $\bd{U}\not\subset \bar{A}\cup f^{-1}(f(x))$. Now
$\bd{U}\setminus (\bar{A}\cup f^{-1}(f(x)))=\bd{U}\setminus f^{-1}(f(\bar{A}))$
is an open non-empty set in $\bd{U}$ by(2).  Thus there is a crosscut $B\subset
U\setminus f^{-1}(f(\bar{A}))$ of $X$ with $\bar{B}\setminus
B\subset\bd{U}\setminus f^{-1}(f(\bar{A}))$. Now $f(B)$ is contained in a
bounded component of $\Complex\setminus (Y\cup f(A))=\Complex\setminus (Y\cup
f(\bar{A}))$ by Theorem~\ref{confeq}. Since $Y\cap f(\bar{A})=\{f(x)\}$ is
connected and $Y$ does not separate $\Complex$, it follows by unicoherence that
$f(B)$ lies in a bounded component of $\Complex\setminus f(\bar{A})$. Since
$Y\setminus \{f(x)\}$ meets $f(\bar{B})$ and misses $f(\bar{A})$ and
$Y\setminus \{f(x)\}$ is connected, $Y\setminus\{f(x)\}$ lies in a bounded
complementary component of $f(\bar{A})$. This is impossible as the diameter of
$f(A)$ is smaller than the diameter of $Y$.  It follows that there exists a
$\delta>0$ such that if the diameter of $A$ is less than $\delta$ and
$f(x)=f(y)$, then $x$ and $y$ must lie in the same component of $f^{-1}(f(x))$.

In order to define the extension $\bar{f}$ of $f$ over the boundary
$\uc$ of $\disk^\infty$, let $C_i$ be a chain  of crosscuts of
$\disk^\infty$ which converge to a point $p\in \uc$ such that
$A_i=\varphi(C_i)$ is a null sequence of crosscuts or return cuts of
$X$ with endpoints $a_i$ and $b_i$ which converge to a point $x\in
X$. There are three cases to consider:

Case 1. $f$ identifies the endpoints of $A_i$ for some $A_i$ with
diameter less than $\delta$. In this case the chain of crosscuts is
mapped by $f$ to a sequence of generalized return cuts based at
$f(a_i)=f(b_i)=f(x)$. Hence $f(a_i)$ is an accessible point of $Y$
which corresponds (under $\psi^{-1}$) to a unique point $q\in \uc$
(since $Y$ has no cutpoints). Define $\bar{f}(p)=q$.

Case 2. Case 1 does not apply and there exists an infinite
subsequence $A_{i_{j}}$ of crosscuts such that
$f(\bar{A}_{i_{j}})\cap f(\bar{A}_{i_{k}})=\0$ for $j\not=k$. In
this case $f(A_{i{_j}})$ is  a chain of generalized crosscuts which
converges to the point $f(x)\in Y$. The chain $\psi^{-1}\circ
f(A_i)$ corresponds to a unique point $q\in \uc$. Define
$\bar{f}(p)=q$.

Case 3. Cases 1 and 2 do not apply. Without loss of generality
suppose there exists an $i$ such that for $j>i$ $f(\bar{A}_i)\cap
f(\bar{A}_j)$ contains $f(a_i)=f(x)$. In this case $f(A_j)$ is a
sequence of generalized crosscuts based at the accessible point
$f(x)$ which corresponds to a unique point $q$ on $\uc$ as above.
Define $\bar{f}(p)=q$.

It remains to be shown that $\bar{f}$ is a continuous extension of
$\hat{f}$ and  $F$ is confluent. For continuity it suffices to show
continuity at $\uc$. Let $p\in \uc$ and let $C$ be a small crosscut
of $\disk^{\infty}$ whose endpoints are on opposite sides of $p$ in
$\uc$ such that $A=\varphi(C)$ has diameter less than $\delta$
\cite{miln00} and such that the endpoints of $A$ are two accessible
points of $X$. Since $f$ is uniformly continuous near $X$, the
diameter of $f(A)$ is small and since $\psi^{-1}$ is uniformly
continuous with respect to connected sets in the complement of $Y$
(\cite{urseyoun51}), the diameter of $B=\psi^{-1} \circ f\circ
\varphi(C)$ is small. Also $B$ is either a generalized crosscut or
generalized return cut. Since $\hat{f}$ preserves separation of
crosscuts, it follows that the image of the domain $U$ bounded by
$C$ which does not contain $\infty$ is small. This implies
continuity of $\bar{f}$ at $p$.

To see that $F$ is confluent let $K\subset \uc$ be a subcontinuum and let $H$
be a component of $\bar{f}^{-1}(K)$. Choose a chain  of crosscuts $C_i$ such
that $\varphi(C_i)=A_i$ is a crosscut of $X$ meeting $X$ in two accessible
points $a_i$ and $b_i$, $C_i\cap \bar{f}^{-1}(K)=\0$ and $\lim C_i=H$. It
follows from the preservation of crosscuts (see Theorem~\ref{confeq}) that
$\hat{f}(C_i)$ separates $K$ from $\infty$. Hence $\hat{f}(C_i)$ must meet
$\uc$ on both sides of $K$ and $\lim \bar{f}(C_i)=K$. Hence $F(H)=\lim
\bar{f}(C_i)=K$ as required.
\end{proof}

\begin{cor}
Suppose that $f:\Complex\to\Complex$ is a perfect, oriented  map of the plane,
$X\subset\Complex$ is  a subcontinuum without cut points and $f(X)=X$. Let
$\hX$ be the component of $f^{-1}(f(X))$ containing $X$. Let $\varphi:
\disk^\infty \to\sphere\setminus T(\hX) $ and
$\psi:\disk^\infty\to\sphere\setminus T(X)$ be conformal mappings. Define
$\hat{f}:\disk^\infty\sm \varphi^{-1}(f^{-1}(X))\to\disk^\infty$ by
$\hat{f}=\psi^{-1} \circ f \circ \varphi$. Put $\uc=\partial\disk^\infty$.

Then $\hat{f}$ extends over $\uc$ to a map
$\bar{f}:\ol{\disk^\infty}\to\ol{\disk^\infty}$. Moreover
$\bar{f}^{-1}(\uc)=\uc$ and $F=\bar{f}|_{\uc}$ is a confluent map.
\end{cor}

\begin{proof}
By Lemma~\ref{cacyclic} $f=g\circ m$ where $m$ is a monotone perfect and onto
mapping of the plane with acyclic point inverses, and $g$ is an open and
perfect surjection of the plane to itself. By Lemma~\ref{finite}, $f^{-1}(X)$
has finitely many components. It follows that there exist a simply connected
open set $V$, containing $T(X)$, such that if $U$ is the component of
$f^{-1}(V)$ containing $\hX$, then $U$ contains no other components of
$f^{-1}(X)$. It is easy to see that $f(U)=V$ and that $U$ is simply connected.
Hence $U$ and $V$ are homeomorphic to $\Complex$. Then  $f|_U:U\to V$ is a
confluent map. The result now follows from Theorem~\ref{tinduced} applied to
$f$ restricted to $U$.
\end{proof}

\chapter{Partitions of domains in the sphere}\label{partKP}

\section{Kulkarni-Pinkall Partitions}\label{secKP}

Throughout this section let $K$ be a compact subset of the plane whose
complement $U=\complex\sm K$ is connected. In the interest of completeness we
define the Kulkarni-Pinkall partition of $U$ and prove the basic properties of
this partition that are essential for our work in Section~\ref{sechyp}.
Kulkarni-Pinkall \cite{kulkpink94} worked in closed $n$-manifolds. We will
follow their approach and adapt it to our situation in the plane.

We think of $K$ as a closed subset of the Riemann sphere $\sphere$, with the
spherical metric and set $\tU=\sphere\sm K= U\cup\{\infty\}$. Let $\B^\infty$
\index{ball@$\B^\infty$} be the family of closed, round  balls $B$ in $\sphere$
such that $Int(B)\subset\tU$ and $|\partial B\cap K|\ge 2$. Then $\B^\infty$ is
in one-to-one correspondence with the family $\B$ \index{ball@$\B$} of closed
subsets $B$ of $\complex$ which are the closure of a complementary component of
a straight line or a round circle in $\complex$ such that $Int(B)\subset U$ and
$|\partial B\cap K|\ge 2$.\index{maximal ball}

\begin{prop}\label{lense}
If $B_1$ and $B_2$ are two closed round balls in $\complex$ such that $B_1\cap
B_2\ne \0$ but does not contain a diameter of either $B_1$ or $B_2$, then
$B_1\cap B_2$ is contained in a ball of diameter strictly less than the
diameters of both  $B_1$ and $B_2$.
\end{prop}

\begin{proof}
Let $\partial B_1\cap \partial B_2=\{s_1,s_2\}$. Then the closed ball with
center $(s_1+s_2)/2$ and radius $|s_1-s_2|/2$ contains  $B_1\cap B_2$.
\end{proof}

If $B$ is the closed ball of minimum diameter that contains $K$, then we say
that $B$ is the \textit{smallest ball} \index{smallest ball} containing $K$. It
is unique by Proposition~\ref{lense}. It exists, since any sequence of balls of
decreasing diameters that contain $K$ has a convergent subsequence. \medbreak

We denote the \emph{Euclidean convex hull of $K$} \index{convex hull!Euclidean}
by $\ECH(K)$ \index{conve@$\ECH(K)$}. It is the intersection of all closed
half-planes (a closed half-plane is the closure of a component of the
complement of a straight line) which contain $K$. Hence $p\in\ECH(K)$ if $p$
cannot be separated from $K$ by a straight line.

Given a closed ball $B\in\B^\infty$, $\Int(B)$ is conformally  equivalent to
the unit disk in $\complex$. Hence its interior can be naturally equipped with
the hyperbolic metric. Geodesics $\rg$ \index{g@$\rg$}  in this metric are
intersections of $\Int(B)$ with round circles $C\subset\rsphere$  which
perpendicularly cross the boundary $\partial B$. For every hyperbolic geodesic
$\rg$, \index{hyperbolic!geodesic} $B\setminus \ol{\rg}$ has exactly two
components. We call the closure of such components \emph{hyperbolic half-planes
of $B$.}\index{hyperbolic!halfplane} Given $B\in\B^\infty$, the
\emph{hyperbolic convex hull of $K$ in $B$}\index{convex hull!hyperbolic} is
the intersection of all (closed) hyperbolic half-planes of $B$ which contain
$K\cap B$ and we denote it by $\HCH(B\cap K)$. \index{convea@$\HCH(B\cap K)$}

\begin{lem}\label{nonsep}
Suppose that $B$ is the smallest ball containing $K\subset\complex$
and let $c\in B$ be its center. Then $c\in\HCH (K\cap\partial B)$.
\end{lem}

\begin{proof}
By contradiction. Suppose that there exists a circle that separates the center
$c$ from $K\cap\partial B$ and crosses $\partial B$ perpendicularly. Then there
exists a line $\ell$ through $c$ such that a half-plane bounded by $\ell$
contains $K\cap \partial B$ in its interior. Let $B'=B+v$ be a translation of
$B$ by a vector $v$ that is orthogonal to $\ell$ and directed into this
halfplane. If $v$ is sufficiently small, then $B'$  contains $K$ in its
interior. Hence, it can be shrunk to a strictly smaller ball that also contains
$K$, contradicting that $B$ has smallest diameter.
\end{proof}

\begin{figure}
\includegraphics{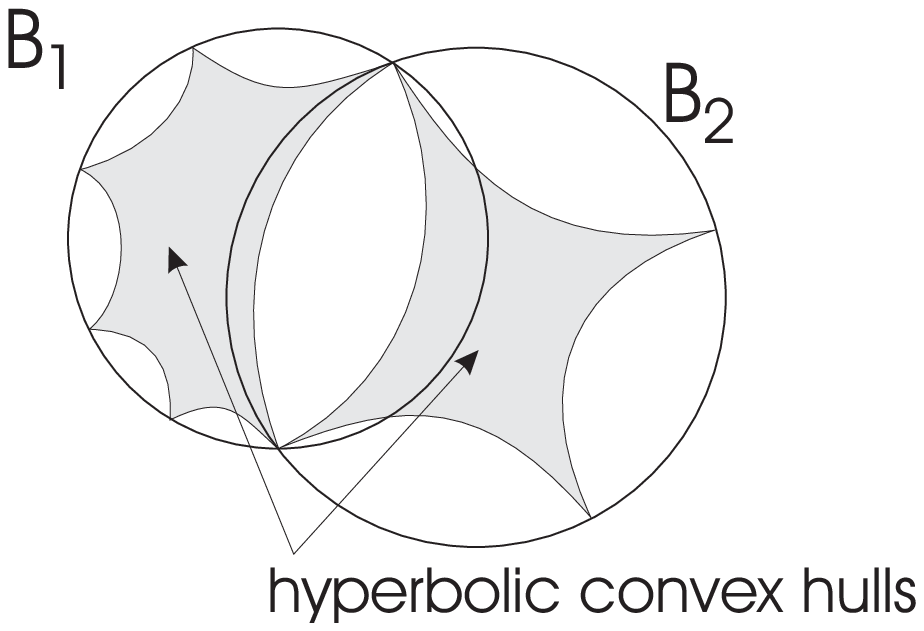}
\caption{Maximal balls have disjoint hulls.} \label{hulls}
\end{figure}

\begin{lem}\label{disjointCH}
Suppose that $B_1,B_2\in\B^\infty$ with $B_1\ne B_2$. Then
$$\HCH(B_1\cap\partial U)\cap\HCH(B_2\cap \partial
U)\subset\partial U.$$ In particular,
$\HCH(B_1\cap\partial U)\cap\HCH(B_2\cap \partial
U)$ contains at
most two points.
\end{lem}

\begin{proof}
A picture easily explains this, see Figure~\ref{hulls}. Note that $\partial
U\cap [B_1\cup B_2] \subset \partial (B_1\cup B_1)$. Therefore $B_1\cap\partial
U$ and $B_2\cap\partial U$ share at most two points. The open hyperbolic chords
between these points in the respective balls are disjoint.
\end{proof}

It follows that any point in $U^\infty(X)$ can be contained in at most one
hyperbolic convex hull. In the next lemma we see that each point of $\tU$ is
indeed contained in $\HCH(B\cap K)$ for some $B\in\B^\infty$. So $\{\tU\cap
\HCH(B\cap K) \mid B\in\B^\infty\}$ is a partition of $\tU$.

Since hyperbolic convex hulls are preserved by M\"{o}bius transformations, they
are more easy to manipulate than the Euclidean convex hulls used by Bell
(which are preserved only by M\"{o}bius transformations that fix $\infty$).
This is illustrated by the proof of the following lemma.

\begin{lem}[Kulkarni-Pinkall inversion lemma]\label{kplem} \index{Kulkarni-Pinkall!Lemma}
For any $p\in \sphere\setminus K$ there exists  $B\in\B^\infty$ such that
$p\in\HCH(B\cap K)$.
\end{lem}

\begin{proof}
We prove first that there exists $B^*\in\B^\infty$ such that no
circle which crosses $\partial B^*$ perpendicularly separates $K\cap
\partial B^*$ from $\infty$.

Let $B'$ be the smallest round ball which contains $K$ and let
$B=\ol{\complex\sm B'}$. Then $B^*=B\cup \{\infty\}\in \B^\infty$.
If $L$ is a circle which crosses $\partial B^*=\partial B'$
perpendicularly and separates $K\cap \partial B'$ from $\infty$,
then it also separates $K\cap\partial B'$ from the center $c'$  of
$B'$, contrary to Lemma~\ref{nonsep}. [To see this note that if $F$
is the M\"{o}bius transformation which fixes points in the boundary
 of $B'$ and interchanges the points $\infty$ and $c'$, then
 $F(L)=L$.  Hence it would follow that $L$ separates $c'$ from
 $K\cap\partial B'$, a contradiction with Lemma~\ref{nonsep}.]  Hence,
$\infty \in\HCH(B^*\cap K)$.

Now let $p\in \complex^\infty\sm K$. Let $M:\sphere\to\sphere$ be a M\"{o}bius
transformation such that $M(p)=\infty$. By the above argument there exists a
ball  $B^*\in\B^\infty$ such that $\infty\in \HCH(B^*\cap M(K))$. Then
$B=M^{-1}(B^*)\in\B^\infty$ and, since $M$ preserves perpendicular circles,
$p\in\HCH(B\cap K)$ as desired.
\end{proof}

From Lemmas~\ref{disjointCH} and \ref{kplem}, we obtain the following Theorem
which is a special case of a Theorem of  Kulkarni and Pinkall
\cite{kulkpink94}.

\begin{thm}\label{KPthrm}\index{Kulkarni-Pinkall!Partition} \index{KPP@$\KPP$}
Suppose that $K\subset\complex$ is a nondegenerate compact set such
that its complement $\tU$ in the Riemann sphere is  connected. Then
$\tU$ is partitioned by the family
$$\KPP=\left\{\tU\cap \HCH(B\cap K)\colon B\in\B^\infty\right\}.$$
\end{thm}

Theorem~\ref{KPthrm} is the linchpin of the theory of geometric crosscuts. An
analogue of it  was known to Harold Bell and used by him implicitly since the
early 1970's. Bell considered non-separating plane continua $K$ and he used the
equivalent notion of Euclidean convex hull of the sets $B\cap\partial U$ for
all maximal balls $B\in\B$ (see the comment following Theorem~\ref{Hypmain}).

\smallskip
 Let $B\in
\B^\infty$. If $B\cap\partial U^\infty(X)$ consists of two points $a$ and $b$,
then its (hyperbolic) hull \index{hull!hyperbolic} is an open circular segment
$\rg$ with endpoints $a$ and $b$ and perpendicular to $\partial B$. We will
call the crosscut $\rg$  a \textit{$\kp$ crosscut} or simply a
\textit{$\kp$-chord}\index{KPchord@$\kp$-chord}. If $B\cap\partial U$ contains three or more
points, then we say that the hull  $\HCH(B\cap\partial U)$ is a \textit{gap}.
\index{gap} A gap has nonempty interior. Its boundary in $\Int(B)$ is a union
of open circular segments (with endpoints in $K$), which we also call
\emph{$\kp$ crosscuts or $\kp$-chords}. We denote by $\KP$ \index{KP@$\KP$}
the collection of all open  chords  obtained as above using all
$B\in\B^\infty$. \medbreak

The following example may serve to illustrate
Theorem~\ref{KPthrm}.\\
{\bf Example}.  Let $K$ be the unit square $\{x+yi\colon
-1\leq x,y\leq 1\}$.
There are five obvious members of $\B$. These are the sets
$$\im z\geq 1,\ \im z\leq -1,\ \re z\geq 1,\ \re z\leq -1,\ |z|\geq \sqrt
2,$$ four of which are half-planes. These are the only members of
$\B$ whose hyperbolic convex hulls have non-empty interiors.
However, for this example the family $\B$ defined in the
introduction of Section~\ref{secKP} is infinite. The hyperbolic hull
of the half-plane $\im z\ge 1$ is the semi-disk $\{z \mid |z-i|\le
1, \ \im z >1\}$. The hyperbolic hulls of the other three
half-planes given above are also semi-disks. The hyperbolic hull of
$|z|\ge \sqrt{2}$ is the unbounded region whose boundary consists of
parts of  four circles lying (except for their endpoints) outside
$K$ and contained in the circles of radius $\sqrt{2}$ and having
centers at $-2,2,-2i$ and $2i$, respectively. These hulls do not
cover $U$ as there are spaces between the hulls of the half-planes
and the hull of $|z|\geq \sqrt 2$.

If $C$ is a circle that circumscribes $K$ and contains exactly two of its
vertices, such as $1\pm i$, then the exterior ball $B$ bounded by
$C$ is maximal. Now $\HCH(B\cap K)$ is a single chord and the union
of all such chords foliates the remaining spaces in $\complex\sm K$.

\begin{lem}\label{contKP}
If $\rg_i$ is a sequence of $\kp$-chords with endpoints $a_i$ and
$b_i$, and $\lim a_i=a\ne b=\lim b_i$, then $\{\ol{\rg_i}\}$ has a
convergent subsequence   and $\lim \ol{\rg_{i_j}}=C$, where
$\rg=C\sm \{a,b\}\in\KP$ is also a $\kp$-chord.
\end{lem}

\begin{proof}
For each $i$ let $B_i\in\B^\infty$ such that
$\rg_i\subset\HCH(B_i\cap K)$. Then a subsequence  $B_{i_j}$
converges to some $B\in\B^\infty$ and $\ol{\rg_{i_j}}$ converges to
a closed circular arc $C$ in $B$ with endpoints $a$ and $b$, and $C$
is perpendicular to $\partial B$. Hence $\rg=C\sm\{a,b\}\subset
\HCH(B\cap K)$. So $\rg\in\KP$.
\end{proof}

By Lemma~\ref{contKP}, the family $\KP$ of  chords has continuity
properties similar to a foliation.

\begin{lem}\label{lenseunion}
For $a,b\in \partial U^\infty$, define $C(a,b)$\index{Cab@$C(a,b)$}
as the union of all $\kp$-chords  with endpoints $a$ and $b$. Then
if $C(a,b)\ne\0$, $C(a,b)$ is either  a single chord, or
$C(a,b)\cup\{a,b\}$ is a closed disk whose boundary consists of two
$\kp$-chords contained in $C(a,b)$ together with $\{a,b\}$.
\end{lem}

\begin{proof}
Suppose $\rg$ and $\rh$ are two distinct $\kp$-chords between $a$
and $b$. Then $S=\rg\cup\rh\cup\{a,b\}$ is a simple closed curve.
Choose a point $z$ in the complementary domain $V$ of $S$ contained
in $\tU$. Since the hyperbolic hulls partition $\tU$, there exists
$B\in \B^\infty$ such that $z\in \HCH(B\cap K)$. By
Lemma~\ref{disjointCH},  $\HCH(B\cap K)$ can only intersect $S\cap
K$ in $\{a,b\}$. So $\HCH(B\cap K)\cap K=\{a,b\}$ and it follows
that $V$ is contained in $C(a,b)$.

The rest of the Lemma follows from \ref{contKP}.
\end{proof}

\section{Hyperbolic foliation of simply connected domains}\label{sechyp}

In this section we will apply the results from Section~\ref{secKP} to the case
that $K$ is a non-separating plane continuum (or, equivalently, that
$\tU=\sphere\sm K$ is simply connected). The results in this section are
essential to \cite{overtymc07,ov09} but are not used in this paper. The reader who
is only interested in the fixed point question can skip this section.

Let $\Disk$ \index{da@$\Disk$} be the open unit disk in the plane. In this
section we let $\phi:\Disk\to \sphere\sm K=\tU$ be a Riemann map onto $\tU$. We
endow $\mathbb D$ with the hyperbolic metric, which is carried to $\tU$ by the
Riemann map. We use $\phi$ and  the Kulkarni-Pinkall hulls to induce a closed
collection  $\Gamma$ of chords in $\mathbb D$ that is a hyperbolic geodesic
lamination in $\Disk$ (see \cite{thur85}).

Let $\rg\in\KP$ be a chord with endpoints $a$ and $b$. Then $a$ and
$b$ are accessible points in $K$ and $\ol{\phi^{-1}(\rg)}$ is an arc
in $\Disk$ with endpoints $z,w\in\partial\Disk$. Let $\rG$
\index{G@$\rG$}  be the hyperbolic geodesic
\index{geodesic!hyperbolic} in $\Disk$ joining $z$ and $w$. Then $G$
is an open circular arc which meets $\partial\disk$ perpendicularly.
Let $\Gamma$ \index{Gamma@$\Gamma$} be the collection of all $\rG$
such that $\rg\in\KP$. We will prove that $\Gamma$ inherits the
properties of the family  $\KP$ as described in Theorem~\ref{KPthrm}
and Lemma~\ref{contKP} (see Lemma~\ref{compactfol},
Theorem~\ref{Hypmain} and the remark following ~\ref{Hypmain}).

Since members of $\KP$  do not intersect (though their closures are arcs which
may have common endpoints)  the same is true for distinct members of $\Gamma$.
We will refer to the members of $\Gamma$ (and their images under $\phi$) as
\emph{hyperbolic chords} or \emph{hyperbolic geodesics} \index{hyperbolic
chord}. Given $\rg\in\KP$ we denote the corresponding element of $\Gamma$ by
$\rG$ and its image  $\phi(\rG)$ in $\tU$ by $\fg$. \index{g@$\fg$} Note that
$\Gamma$ is a lamination of $\mathbb{D}$ in the sense of Thurston\cite{thur85}.
By a \emph{gap} of $\Gamma$ (or of $\phi(\Gamma)$), we mean the closure of a
component of $\mathbb{D}\sm\bigcup \Gamma$ in $\mathbb{D}$ (or its image under
$\phi$ in $\tU$, respectively).

\begin{lem}[J{\o}rgensen~{\cite[p.91 and 93r]{pomm92}}]\label{jorg}\index{J{\o}rgensen Lemma}
Let $B$ be a closed round ball such that its interior is in $\tU$.
Let $\gamma\subset\mathbb D$ be a hyperbolic geodesic. Then
$\phi(\gamma)\cap B$ is connected. In particular, if $R_t$ is an
external ray in $U^\infty$ and $B\in\B^\infty$, then $R_t\cap B$ is
connected.
\end{lem}

If $a,b\in\partial\tU$, recall that $C(a,b)$ is the union of all $\kp$-chords
with endpoints $a$ and $b$. From the viewpoint of prime ends, all chords in
$C(a,b)$ are the same. That is why all the chords in $C(a,b)$ are replaced  by
a single hyperbolic chord  $\fg\in\phi(\Gamma)$. The following lemma follows.

\begin{lem}\label{sameb}
Suppose $\rg\in\kp$ and $\rg\subset \HCH(B\cap \partial\tU)$ joins
the points $a,b\in\partial\tU$ for some $B\in \B^\infty$. If
$G\in\Gamma$ is the corresponding hyperbolic geodesic,  then
$\fg=\phi(\rG)\subset B$.
\end{lem}

\begin{proof}We may
assume that the Riemann map $\phi:\Disk\to\tU$ is extended over all
points $x\in \uc$ so that $\phi(x)$ is an accessible point of $\tU$.
Let $\phi^{-1}(a)=\hat a$, $\phi^{-1}(b)=\tb$ and
$\phi^{-1}(B)=\tB$, and let $\rG$ be the hyperbolic geodesic joining
the points $\hat a$ and $\tb$ in $\Disk$. (Note that this extended
map is not necessarily continuous at points of $\ol{\disk}$
corresponding to accessible points of $K$.) Suppose, by way of
contradiction, that $x\in\rG\sm \tB$. Let $C$ be the component of
$\ol{\Disk}\sm \ol{\phi^{-1}(\rg)}$ which does not contain $x$.
Choose $a_i\to\hat a$ and $b_i\to \tb$ in $\uc\cap C$ and let $H_i$
be the hyperbolic geodesic in $\Disk$ joining the points $a_i$ and
$b_i$. Then $\lim H_i=\rG$ and $H_i\cap \tB$ is not connected for
$i$ large. Hence $\phi(H_i)\cap B$ is not connected for $i$ large.
This contradiction with Lemma~\ref{jorg} completes the proof.
\end{proof}

\begin{lem}\label{compactfol}
Suppose that $\{\rG_i\}$ is a sequence of hyperbolic chords in $\Gamma$ and
suppose that $x_i\in \rG_i$ such that $\{x_i\}$ converges to $x\in \mathbb D$.
Then there is a  unique hyperbolic chord $\rG\in\Gamma$ that contains $x$.
Furthermore, $\lim \rG_i=\ol{\rG}$.
\end{lem}

\begin{proof}
We may suppose that a subsequence sequence $\{\rG_{i_j}\}$ converges
to a hyperbolic chord $\rG$ which contains $x$. Let $\rg_i\in\KP$ so
that $\phi^{-1}(\rg_i)$ is an open arc which joins the endpoints of
$\rG_i$. By Lemma~\ref{contKP},there ia another  subsequence so that
$\lim \rg_{i_{j(t)}}=\rg\in\kp$. It follows that $\rG$ is the
hyperbolic chord joining the endpoints of $\phi^{-1}(\rg)$. Hence
$\rG\in\Gamma$. Since the above argument applies to all
subsequences, the sequence $G_i$ must converge to $G$. \end{proof}

So we have used the  family  of $\kp$-chords in $\tU$ to stratify
$\mathbb{D}$ to the family $\Gamma$ of hyperbolic chords. In
particular gaps of $\Gamma$ are no longer necessarily disjoint but
they can meet at most in a common boundary chord. By
Lemma~\ref{sameb} for each $\kp$-chord $\rg\subset \HCH(B\cap
\partial\tU)$ its associated hyperbolic chord $\fg=\phi(\rG)\subset
B$. Hence, there is a continuous deformation of $\tU$ that maps
$\bigcup \KP$ onto $\bigcup\phi(\Gamma)$, which suggests that
components of $\tU\sm \bigcup \phi(\Gamma)$ naturally correspond to
the interiors of the gaps of the Kulkarni-Pinkall partition. That
this is indeed the case is the substance of the next lemma.

\begin{lem}
There is a $1-1$ correspondence between complementary domains $Z\subset\mathbb
D\setminus \bigcup \Gamma$ and the interiors of  Kulkarni-Pinkall gaps $\HCH(
B\cap K)$. Moreover, for each gap $Z$ of $\Gamma$ there exists a unique
$B\in\B^\infty$ such that $Z$ corresponds to the interior of the $\KP$ gap
$\HCH(B\cap K)\cap \tU$ in that $\partial Z\cap \Disk=\bigcup\{G\in\Gamma\mid
\rg\in\KP \text{ and } \rg\subset \partial \HCH(B\cap K)\}$ and $\phi(Z)\subset
B$.
\end{lem}

\begin{proof}
Let $\rg$ and $\rh$ be two distinct $\kp$-chords in the boundary of
the gap $\HCH(B\cap K)$ for some $B\in\B^\infty$. Let $\{a,b\}$ and
$\{c,d\}$ be the endpoints of $\phi^{-1}(\rg)$ and $\phi^{-1}(\rh)$,
respectively.  Since $\rg$ and $\rh$ are contained in the same gap,
no hyperbolic chord of $\Gamma$ separates $G$ and $H$. Hence there
exists a gap $Z$ of $\Gamma$ whose boundary includes the hyperbolic
chords $\rG$ and $\rH$. It now follows easily that for any
$\rg'\in\KP$ which is contained in the boundary of the same gap
$\HCH(B\cap K)$, $\rG'$ is contained in the boundary of $Z$. Hence
the $\KP$ gap $\HCH(B\cap K)$ corresponds to the gap $Z$ of
$\Gamma$. Conversely, if $Z$ is a gap of $\Gamma$ in $\Disk$ then a
similar argument, together with Lemmas~\ref{contKP} and
\ref{lenseunion}, implies that $Z$ corresponds to a unique gap
$\HCH(B\cap K)$ for some $B\in\B^\infty$. The rest of the Lemma now
follows from Lemma~\ref{sameb}.
\end{proof}

So if $\tU=\sphere\sm K$ is endowed with the hyperbolic metric induced by
$\phi$, then there exists a family of geodesic chords that share the same
endpoints as elements of $\KP$. The complementary domains of $\tU\sm
\bigcup\{\fg\mid \rg\in\KP\}$ corresponds to the Kulkarni-Pinkall gaps. We
summarize the results:

\begin{thm}\label{Hypmain}
Suppose that $K\subset\complex$ is a non-separating continuum and let $\tU$ be
its complementary domain in the Riemann sphere. There exists a family
$\phi(\Gamma)$ of hyperbolic chords  in the hyperbolic metric on $\tU$ such
that for each $\fg\in\phi(\Gamma)$ there exists $B\in\B^\infty$ and $\rg\subset
\HCH(B\cap \partial\tU)$ so that $\fg$ and $\rg$ have the same endpoints and
$\fg\subset B$.  Each domain $Z$ of $\tU\setminus \phi(\Gamma)$  naturally
corresponds to a Kulkarni-Pinkall gap $\HCH(B\cap\partial \tU)$ The bounding
hyperbolic chords of $Z$ in $\tU$ correspond to the $\kp$-chords (i.e., chords
in $\KP$) of $\HCH(B\cap\partial \tU)$.
\end{thm}

In order to obtain Bell's Euclidean foliation \cite{bell76} we could have
modified the $\KP$ family as follows. Suppose that $B\in\B$. Instead of
replacing a $\kp$-chord  $\rg\in\HCH(B\cap K)$  by a geodesic in the hyperbolic
metric on $\tU$, we could have replaced it by a straight line segment; i.e, the
geodesic in the Euclidean metric. Then we would have obtained a family of open
straight line segments. In so doing we would have  replaced the gaps
$\HCH(B\cap\partial \tU)$ by $\ECH(B\cap\partial \tU)$, which is the way in
which Bell originally foliated $\ECH(K)\setminus K$. We hope that the above
argument provides a more transparent proof of Bell's result. Note that both in
the hyperbolic and Euclidean case the elements of the foliation are not
necessarily disjoint (hence we use the word \lq\lq foliate'' rather then \lq\lq
partition''). However, in both cases every point of $\tU$ is  contained in
either a unique chord or in the interior of a unique gap.

\section{Schoenflies Theorem}\label{Schoenflies}

In this short section we will show that the Schoenflies Theorem
follows immediately  from Theorem~\ref{Hypmain} (see
\cite{siebenmann2005} for some recent history of this old problem
and \cite{overtymc07,ov09} for more details and extensions of these
ideas). We want to emphasize here that no results of Chapter~\ref{c:tools} are relied upon
in Section~\ref{Schoenflies}.

\begin{thm}[Schoenflies Theorem] Suppose that $h:\uc\to \C$ is an embedding
of the unit circle in the plane and $U$ is a bounded complementary domain of $h(\uc)=S$. Then
there exists an embedding $H:\ol{\disk}\to\C$ which extends $h$.
\end{thm}
\begin{proof} Let $\B=\{B_\al\}$ be the collection of maximal open balls in $U$ such that $|\partial B_\al\cap \partial U|\ge 2$.
For each $\al$ let $F_\al=\ECH(\partial B_\al\cap\partial U)$. Let $\lam$ be the collection of
all chords in the boundaries of all the sets
$F_\al$ and let $\lam^*$ be the union of all the chords in $\lam$. Let $\mathcal{G}=\{G_\beta\}$ be the collection of all
components of $U\sm \lam^*$. By  Theorem~\ref{Hypmain} there exists
for each $z\in U$ either a unique chord $\ell\in\lam$ such that $z\in\ell$ or a unique $G_\beta\in\mathcal{G}$
such that $z\in G_\beta$. Moreover, in the latter case, there exists a unique $\al$ such that $\ol{G_\beta}=F_\al\subset\ol{B_\al}$.
Note that  all chords in $\lam$ are straight line segments and
all the sets $F_\al$ are Euclidean convex sets. Suppose that  $x_i\in\lam^*$ such that $\lim x_i=x_\infty$ and
 $x_i\in\ell_i\in\lam$. Then either $\lim \ell_i=\ell_\infty$ and $x_\infty\in\ell_\infty\in\lam$, or
 $\lim \ell_i=x_\infty\in\partial U\subset h(\uc)$.
  Now we can pull back the lamination
$\lam$ to a lamination $\mathcal{E}$ of the unit disk $\disk$ by $h$:  if $\ell\in\lam$ is a chord joining the
points $y_1,y_2\in S$, then connect the points $h^{-1}(y_1), h^{-1}(y_2)$  by the straight line segment,
denoted by $h^{-1}(\ell)$ in the unit disk.
Let $\mathcal{E}^*$ be the union of all such line segments. Note that each gap $G_\beta$ of $U$ uniquely corresponds
to a (Euclidean convex) gap $H_\beta$ of $\mathcal{E}$ (i.e., $H_\beta$ is a component of $\disk\sm \mathcal{E}^*$.

Now extend $h$ first over $\mathcal{E}^*$ by mapping each chord in $\mathcal{E}$ linearly onto the corresponding chord in
$\lam$.  For each gap $H_\beta$ ($G_\beta$) of $\mathcal{E}$ ($\lam$) let $h_\beta$ ($g_\beta$, respectively) be its
barycenter. Then it follows easily that we can extend the map $h$ continuously  by defining $H(h_\beta)=g_\beta$
for each $\beta$.
Finally extend $H$ over all of $H_\beta$ by mapping, for each $w\in\partial H_\beta$ the straight line segment
$wh_\beta$ linearly onto the straight line segment joining the points $H(w)$ and $H(h_\beta)=g_\beta$.
Then $H$ is the required extension of $h$.
\end{proof}

\section{Prime ends}\label{kpchords}

We will follow the notation from Section~\ref{secKP} in the case
that $K=T(X)$ where $X$ is a plane continuum. Here we assume, as in
the introduction to this paper,  that $f:\complex\to\complex$ takes
the continuum $X$ into $T(X)$ with no fixed points in $T(X)$, and
$X$ is minimal with respect to these properties. We apply the
Kulkarni-Pinkall partition to $\tU=\sphere\sm T(X)$. Recall that
$\KPP=\{\HCH(B\cap K)\cap \tU\mid B\in\B^\infty\}$ is the Kulkarni
Pinkall partition of $\tU$ as given by Theorem~\ref{KPthrm}.

Let $B^\infty\in\B^\infty$ be the  maximal ball such that
$\infty\in\HCH(B^\infty\cap K)$. As before we use balls on the sphere. In
particular, straight lines in the plane correspond to circles on the sphere
containing the point at infinity. The subfamily of $\KPP$ whose elements are of
diameter $\leq \delta$ in the spherical metric is denoted by $\KPP_\delta$.
\index{KPPD@$\KPP_\delta$} The subfamily of  chords in $\KP$  of diameter $\leq
\delta$ is denoted by $\KP_\delta$. \index{KP chord@$\KP_\delta$}

By Lemma~\ref{contKP} we  know that the families $\KP$ and $\KPP$ have nice continuity
properties. However, $\KP$ and $\KPP$ are not closed in
the hyperspace of compact subsets of $\complex^\infty$: a sequence
of chords or hulls may converge to a point in the boundary
of $\tU$ (in which case it  must be a null sequence).

\begin{prop}[Closedness]\label{compactness}
Let $\{\rg_i\}$ be a convergent sequence of distinct elements in  $\KP_\delta$,
then either $\rg_i$ converges to a chord $\rg$ in $\KP_\delta$ or $\rg_i$
converges to a point of $X$. In the first case, for  large $i$ and $\delta$
sufficiently small, $\var(f,\rg,T(X))=\var(f,\rg_i,T(X))$.
\end{prop}

\begin{proof}
By Lemma~\ref{contKP},  we know that the first conclusion holds if
$\rg=\lim \rg_i$ contains a point of $\tU$. Hence we only need to
consider the case when $\lim \rg_i=\rg\subset
\partial \tU\subset T(X)$. If the diameter of $\rg_i$
converged to zero, then $\rg$ is a point as desired. Assume that
this is not the case and let $B_i$ be the maximal ball that contains
$\rg_i$. Under our assumption, the diameters of $\{B_i\}$ do not
decay to zero. Let $B\in\B^\infty$ be the limit of a subsequence
$B_{i_j}$. Then $\lim \rg_i$ is a piece of a round circle which
crosses $\partial B$ perpendicularly. Hence $\lim \rg_i\cap
\Int(B)\ne\0$, contradicting the fact that $\rg\subset
\partial\tU\subset T(X)$. Note that for $\delta$ sufficiently small,
$\ol{\rg}\cap f(\ol{\rg})=\0$. Hence, $\var(f,\rg,T(X))$ and
$\var(f,\rg_i,T(X))$ are defined for all $i$ sufficiently large.
Then last statement in the Lemma follows from stability of variation
(see Section~\ref{compofvar}).
\end{proof}

\begin{cor}\label{small}
For each $\e>0$, there exist $\delta>0$ such that for all $\rg\in\KP$ with
$\rg\subset B(T(X),\delta)$, $\dm(\rg)<\e$.
\end{cor}

\begin{proof}
Suppose not, then there exist $\e>0$ and  a sequence $\rg_i$ in $\KP$ such that
$\lim \rg_i\subset X$ and $\dm(\rg_i)\geq\e$ a contradiction to
Proposition~\ref{compactness}.
\end{proof}

The proof of the following well-known proposition is omitted.

\begin{prop}\label{smallcutoff}
For each $\e>0$ there exists $\delta>0$ such that for each open arc $A$ with
distinct endpoints $a,b$ such that $\cl{A}\cap T(X)=\{a,b\}$ and
$\dm(A)<\delta$, $T(T(X)\cup A)\subset B(T(X),\e)$.
\end{prop}

\begin{prop} \label{smallgeometric}
Let $\e$, $\delta$ be as in Proposition~\ref{smallcutoff} above with
$\delta<\e/2$ and let $B\in\B^\infty$. Let $A$ be a crosscut of $T(X)$ such
that $\dm(A)<\delta$. If $x\in T(A\cup T(X))\cap \HCH(B\cap T(X))\sm T(X)$ and
$d(x,A)\geq\e$, then the radius of $B$ is less than $\e$. Hence,
$\dm(\HCH(B\cap T(X)))<2\e$.
\end{prop}

\begin{proof}
Let $z$ be the center of $B$. If  $d(z,T(X))<\e$ then $\dm(B)<2\e$
and we are done. Hence, we may assume that $d(z,T(X))\ge \e$. We
will show that this leads to a contradiction. By
Proposition~\ref{smallcutoff} and our choice of $\delta$, $z\in
\sphere\sm T(A\cup X)$. The straight line segment $\ell$ from $x$ to
$z$ must cross $T(X)\cup A$ at some point $w$. Since the segment
$\ell$ is in the interior of the maximal ball $B$, it is disjoint
from $T(X)$, so $w\in A$. Hence $d(x,w)\geq\e$ and, since $x\in B$,
$B(w,\e)\subset B$. This is  a contradiction since $A\subset
B(w,\delta)$ and $\delta<\e/2$ so $\overline A$ would be contained
in the interior of $B$ which is impossible since $A$ is a crosscut
of $T(X)$.
\end{proof}

\begin{prop}\label{crossing}
Let $C$ be a crosscut of $T(X)$ and let $A$ and $B$ be disjoint closed sets in
$T(X)$ such that $\cl{C}\cap A\not=\0\not= \cl{C}\cap B$. For each $x\in C$,
let $F_x\in\KPP$ so that $x\in F_x$. If each $\ol{F_x}$ intersects $A\cup B$,
then there exists an $F_\infty\in\KPP$ such that $\ol{F_\infty}$ intersects
$A$,  $B$ and $\ol{C}$.
\end{prop}

\begin{proof}
Let $a\in A,b\in B$ be the endpoints of $\cl C$. Let $C_a,
C_b\subset C$ be the set of points $x\in C$ such that $\ol{F_x}$
intersects $A$ or $B$, respectively. Then $C_a$ and $C_b$ are closed
subsets by Proposition~\ref{compactness}. Note that $d(A,B)>0$. If
$C_a=\0$, choose $x_i\in C$ converging to $a\in A\cap\ol{C}$.
 Let $F_{x_i}=\HCH(B_i\cap T(X))$, where $B_i\in\B$ and assume that $B_\infty=\lim B_i$.
  Then by Lemma~\ref{contKP},
$\ol{F_\infty} \cap B\ne\0 $  and $\lim
F_{x_i}\subset\ol{F_\infty}\subset \HCH(B_\infty\cap K)$. Then
$\ol{F_\infty} \cap B\ne\0 $ and $a\in A\cap \ol{C}\cap
\ol{F_\infty}$. Suppose now $C_a\ne\0\ne C_b$. Then $C_a$ and $C_b$
are closed and, since $C$ is connected, $C_a\cap C_b\ne\0$. Let
$y\in C_a\cap C_b$. Then $\ol{F_y}\cap A\ne\0\ne \ol{F_y}\cap B$ and
$y\in F_y\cap C$.
\end{proof}

Proposition~\ref{crossing} allows us to replace small crosscuts
which essentially cross the external ray $R_t$  with non-trivial
principal continuum by small nearby $\kp$-chords which also
essentially crosses $R_t$. For if $C$ is a small crosscut of $T(X)$
with endpoints $a$ and $b$ which crosses the external ray $R_t$
essentially, let $A$ and $B$ be the closures of the sets in $T(X)$
accessible from $a$ and $b$, respectively by small arcs missing
$R_t$. If the $F_\infty$ of proposition~\ref{crossing} is a gap
$\HCH(B\cap T(X))$, then a $\kp$-chord in its boundary crosses $R_t$
essentially.

Fix a Riemann map $\varphi:\disk^\infty\to \tU=\rsphere\sm T(X)$
with $\varphi(\infty)=\infty$. Recall that an external ray $R_t$ is the image
of the radial line segment with argument $2\pi t i$ under the map $\varphi$.

\begin{prop} \label{var0}
Suppose the external ray $R_t$ lands on $x\in T(X)$, and
$\{\rg_i\}_{i=1}^\infty$ is a sequence of crosscuts of $T(X)$ converging to $x$
such that there exists a null sequence of arcs $A_i\subset \complex\setminus
T(X)$ joining $\rg_i$ to $R_t$. Then for sufficiently large $i$,
$\var(f,\rg_i,T(X))=0$. \end{prop}

\begin{proof}
Since $f$ is fixed point free on $T(X)$ and $f(x)\in T(X)$, we may
choose  a small ball  $W$ with center  $x$ in $\complex$ such that
$f(\cl{W})\cap(\cl{W}\cup R_t)=\0$. For sufficiently large $i$,
$A_i\cup \rg_i\subset W$. Then for each such $i$ there exists a
junction $J_i$ starting from a point in $\rg_i$, with all of its
legs staying in $W$ close to $A_i$ until it reaches $R_t$, and then
staying close to $R_t$ to $\infty$. By our choice of $W$,
$\var(f,\rg_i,T(X))=0$.
\end{proof}

\begin{prop}\label{pomm92}
Suppose that for an external ray $R_t$ we have $R_t\cap \linebreak[4] \Int(
\ECH(T(X))) \ne \0$. Then there exists $x\in R_t$ such that the $(T(X),x)$-end
of $R_t$ is contained in $\ECH(T(X))$. In particular there exists a chord
$\rg\in\KP$ such that $R_t$ crosses $\rg$ essentially.
\end{prop}

\begin{proof}
External rays in $\tU$ correspond to geodesic half-lines starting at infinity
in the hyperbolic metric on $\sphere\sm T(X)$. Half-planes are conformally
equivalent to disks. Therefore, J{\o}rgensen's lemma applies: the intersection
of $R_t$ with a halfplane is connected, so it is a half-line. Since the
Euclidean convex hull of $T(X)$ is the intersection of all half-planes
containing $T(X)$, $R_t\cap\ECH(T(X))$ is connected.
\end{proof}

\begin{lem}\label{chordlimit}
Let $\mc{E}_t$ be a channel (that is, a prime end such that $\pr(\mc{E}_t)$ is
non-degenerate) in $T(X)$. Then for each $x\in\pr(\mc{E}_t)$, for every
$\delta>0$, there is a chain $\{\rg_i\}_{i=1}^\infty$ of chords defining
$\mc{E}_t$ selected from $\KP_\delta$ with  $\rg_i\to x\in \bd T(X)$.
\end{lem}

\begin{proof}
Let $x\in\pr(\mc{E}_t)$ and let $\{C_i\}$ be a defining chain of
crosscuts for $\pr(\mc{E}_t)$ with $\{x\}=\lim C_i$. By
Proposition~\ref{crossing}, in particular by the remark following
the proof of that proposition, there is a sequence $\{\rg_i\}$ of
$\kp$-chords such that $d(\rg_i, C_i)\to 0$ and $R_t$ crosses each
$\rg_i$ essentially. By Proposition~\ref{smallgeometric}, the
sequence $\rg_i$ converges to $\{x\}$.
\end{proof}

\begin{lem} \label{cutoff}
Suppose an external ray $R_t$ lands on $a\in T(X)$ with
$\{a\}=\pr(\mc{E}_t)\not=\im(\mc{E}_t)$.  Suppose
$\{x_i\}_{i=1}^\infty$ is a collection of points in $\tU$ with
$x_i\to x\in\im(\mc{E}_t)\sm \{a\}$ and $\phi^{-1}(x_i)\to t$. Then
there is a sequence of $\kp$-chords $\{\rg_i\}_{i=1}^\infty$ such
that for sufficiently large $i$, $\rg_i$ separates $x_i$ from
$\infty$, $\rg_i\to a$ and $\phi^{-1}(\rg_i)\to t$.
\end{lem}

\begin{proof} The existence of the chords $\rg_i$ again follows from the remark
following Proposition~\ref{crossing}. It is easy to see that $\lim
\varphi^{-1}(\rg_i)\to t$.
\end{proof}

\subsection{Auxiliary Continua}\label{auxcont}

We use $\kp$-chords to form Carath\'{e}odory loops around the continuum $T(X)$.

\begin{defn}
Fix $\delta>0$.  Define the following collections of chords:\index{KPd
chords@$\KP^{\pm}_\delta$}
$$\KP^+_\delta=\{\rg\in\KP_\delta\mid \var(f,\rg,T(X))\ge 0\}$$
$$\KP^-_\delta=\{\rg\in\KP_\delta\mid \var(f,\rg,T(X))\le 0\}$$
$$\KP_\delta=\KP_\delta^+\cup\KP_\delta^-$$
To
each collection of chords above, there corresponds an auxiliary
continuum defined as follows: $$T(X)_\delta=T(X\cup(\cup
\KP_\delta))$$ \index{TXd@$T(X)_\delta$} $$T(X)^+_\delta=T(X\cup(\cup \KP^+_\delta))$$
$$T(X)^-_\delta=T(X\cup(\cup \KP^-_\delta))$$
\end{defn}

\begin{prop} \label{boundary}
Let $Z\in\{T(X)_\delta,T(X)^+_\delta,T(X)^-_\delta\}$, and correspondingly
$\mc{W}\in\{\KP_\delta,\KP^+_\delta,\KP^-_\delta\}$. Then the following hold:
\begin{enumerate}
\item $Z$ is a non-separating
plane continuum.

\item $\bd Z\subset T(X)\cup(\cup\mc{W})$.

\item Every
accessible point $y$ in $\bd Z$ is either a point of $T(X)$ or a point
interior to a chord $\rg\in\mc{W}$.

\item  If $y\in\partial Z\cap\rg$  with $\rg\in\mc{W}$, then $y$ is accessible,
$\rg\subset \partial Z$ and $\partial Z$ is locally connected at
each point of $\rg$. Hence, if $\varphi:\disk^\infty\to\rsphere\sm
Z$ is the Riemann map and $R_t$ is an external ray landing at $y$,
then $\varphi$ extends continuously to an open interval in $\uc$
containing $t$. Moreover, if $y\in \partial Z\cap [\ol{\rg}\sm\rg]$,
then $\vp$ extends continuously over  a half open $J\subset \uc$
with endpoint $t$ so that $\vp(J)\subset \ol{\rg}$.
\end{enumerate}
\end{prop}

\begin{proof}
By Proposition~\ref{compactness}, $T(X)\cup(\cup \mc{W})$ is compact.
Moreover, $T(X)\cup(\cup \mc{W})$  is connected since each crosscut
$A\in\mc{W}$ has endpoints in $T(X)$.  Hence, the topological hull
$T(T(X)\cup(\cup \mc{W}))$ is a non-separating plane continuum, establishing
(1).

Since $Z$ is the topological hull of $T(X)\cup(\cup \mc{W})$, no boundary
points can be in complementary domains of $T(X)\cup(\cup \mc{W})$.  Hence, $\bd
Z\subset T(X)\cup(\cup \mc{W})$, establishing (2). Conclusion (3) follows
immediately.

Suppose $y\in\partial Z\cap\rg$ with $\rg\in\mc{W}$. Then $\Sh(\rg)\subset Z$
and there exists $y_i\in\complex\sm Z$ such that $\lim y_i=y$. We may assume
that all the points $y_i$ are on the ``same side'' of the arc $\rg$  (i.e.,
$y_i\in\complex\sm \Sh(\rg)$).  This side of $\rg$ is either (1) a limit of
$\kp$-chords $\rg_j$, or (2) there exists a gap $\HCH(B\cap X)$ on this side
with $\rg$ in its boundary. In case (1), $\rg\subset \Sh(\rg_j)$ and, since
$y_i\in \complex\sm Z$ for all $i$, $\rg_j\not\in \mc{W}$. Hence each
$\rg_j\subset \complex\sm Z$ for all $j$. It follows that every point of $\rg$
is accessible, $\rg\subset \partial Z$ and $\partial Z$ is locally connected at
each point of $\rg$. In case (2) there exists a chord $ \rg'\ne\rg$ in the
boundary of $\HCH(B\cap X)$ which separates $\rg$ from infinity. Then
$\rg'\not\in\mc{W}$ and  the interior of $\HCH(B\cap X)\subset \complex\sm Z$.
Hence the same conclusion follows.

The last part of (4) follows from  the proof of Carath\'eodory's theorem (see
\cite{pomm92}).
\end{proof}

\begin{prop}\label{LC}
$T(X)_\delta$ is locally connected; hence, $\bd T(X)_\delta$ is a
Ca\-ra\-th\'e\-o\-do\-ry loop.
\end{prop}

\begin{proof}
Suppose that $T(X)_\delta$ is not locally connected. Then $T(X)_\delta$ has a
non-trivial impression and there exist $0<\e<\delta/2$ and a chain $A_i$ of
crosscuts of $T(X)_\delta$ such that $\dm( \Sh(A_i))>\e$ for all $i$. We may
assume that $\lim A_i=y\in T(X)_\delta$.

By Proposition~\ref{boundary} (4) we may assume $y\in X$. Choose
$z_i\in\Sh(A_i)$ such that $ d(z_i,y)>\e$. We can enlarge the
crosscut $A_i$ of $T(X)_\delta$ to a crosscut $C_i$ of $T(X)$ as
follows. Suppose that $A_i$ joins the points $a^+_i$ and $a^-_i$ in
$T(X)_\delta$. If $a^+_i\in T(X)$, put $y^+_i=a^+_i$. Otherwise
$a^+_i$ is contained in a chord $\rg^+_i\in\KP_\delta$, with
endpoints in $T(X)$, which is contained in $T(X)_\delta$. Since
$\lim A_i=y$, we can select one of these endpoints and call it
$y^+_i$ such that $d(y^+_i,a^+_i)\to 0$. Define $\rg^-_i$ and
$y^-_i$ similarly. Then $\rg^+_i\cup A_i\cup \rg^-_i$ contains a
crosscut $C_i$ of $T(X)$ joining the points $y^+_i$ and $y^-_i$ such
that $\lim C_i=y$. We claim that $z_i\in\Sh(C_i)$. To see this note
that, since $z_i\in \Sh(A_i)$, there exists a half-ray $R_i\subset
\complex\setminus T(X)_\delta$ joining $z_i$ to infinity such that
$|R_i\cap A_i|$ is an odd number and each intersection is
transverse. Since $R_i\cap C_i=R_i\cap A_i$ it follows that
$z_i\in\Sh(C_i)$. Let $\HCH(B_i\cap X)$ be the unique hull of the
Kulkarni-Pinkall partition $\KPP$ which contains $z_i$. Since
$\dm(C_i)\to 0$ and $d(z_i,y)>\e$, it follows from
Proposition~\ref{smallgeometric} that $\dm(\HCH(B_i\cap
X))<2\e<\delta$. This contradicts the fact that
$z_i\in\complex\setminus T(X)_\delta$ and completes the proof.
\end{proof}

\part{Applications of basic theory}

\chapter{Description of main results of Part 2}\label{descr2}

We begin by describing the results obtained in Part 2. These results
are applications of the tools developed in Part 1. We will say that
a continuum $X$ is \emph{decomposable}\index{continuum!decomposable}
if there exist two proper subcontinua $A, B$ of $X$ such that
$X=A\cup B$. A continuum which is not decomposable, is called
\emph{indecomposable}.\index{continuum!indecomposable}

\section{Outchannels}

In Chapter~\ref{outcha} we will study outchannels. Outchannels were
introduced by Bell to establish that a minimal counterexample to the
Plane Fixed Point Problem must be an indecomposable continuum. In
Chapter~\ref{outcha} we will recover this result and strengthen it by
showing that the outchannel in a minimal counterexample to the Plane
Fixed Point Problem is unique: there exists exactly one prime end
$\mc{E}_t$ which corresponds to a dense channel with non zero
variation. It will follow that the variation of this channel must be
$-1$ while all other small crosscuts, which do not cross this channel
essentially, must have variation zero. Let us assume that $f:\C\to
\C$ with a forward invariant non-separating continuum $X$ presents a
(possibly existing) minimal counterexample to the Plane Fixed Point
Problem.

We construct a specific locally connected (but not invariant)
continuum $X'\supset X$ by adding small crosscuts to $X$. This will
be done in a careful way; we will only add Kulkarni-Pinkall crosscuts
from $\kp$. This construction is used to show that if there is a minimal
counterexample $(X, f)$ to the Plane Fixed Point Problem, then there
exists a continuum $Z$ such that the following facts hold.

\begin{enumerate}
\item $Z\supset X$;
\item there exists a one-to-one map $\vp:\reals \to Z$,
\item $\vp(\reals)$ is the set of accessible points of $Z$,
\item as $t\to\infty$, $\vp(t)$ and $\vp(-t)$ run along
opposite sides of the outchannel.
\end{enumerate}

Moreover, the same construction is important in the proof of the
uniqueness of the outchannel.

These ideas are also applied in \cite{blokoverto1}. There it was
shown that in certain cases a minimal subcontinuum $X$ without a
fixed point must be fully invariant.  As an important tool it was
shown in that paper that the map $f$ can be modified on $\C\sm X$ to
a map $g$  such that $g(R_t)=R_t$, $g$ maps points on $R_t$ closer to
infinity and $g$ locally interchanges the two sides of $R_t$. Here
$R_t$ is the conformal external ray which represents the prime end
corresponding to the outchannel. Note that if $X$ is fully invariant
then a prime end which corresponds to the outchannel has the property
that, in a defining sequence $\{C_i\}$ of crosscuts of the prime end
$f(C_i)$ separates $C_i$ from infinity in $U^\infty(X)$ (thus
justifying the name ``outchannel'').

Suppose that $f:\complex\to \complex$ is a perfect map, $X$ is a
continuum, $f$ has no fixed point in $T(X)$ and $X$  is minimal with
respect to $f(X)\subset T(X)$. Fix $\eta>0$ such that for each
$\kp$-chord $\rg\subset T(X)_\eta$, $\ol{\rg}\cap f(\ol{\rg})=\0$ and
$f$ is fixed point free on $T(X)_\eta$. In this case we will say that
$\eta$ \emph{defines variation near} \index{defines variation near
$X$} $X$ and that the triple $(f,X,\eta)$ satisfies the
\emph{standing hypothesis}. \index{standing
hypothesis}\index{fXeta@$(f,X,\eta)$}\label{standh} As usual, for a
continuum $X$ let $\tU=\rsphere\sm T(X)$.

\begin{defn} [Outchannel]\label{outch}
Suppose that the triple  $(X,f,\eta)$ satisfies the standing
hypothesis. An {\em outchannel} \index{outchannel} of the
non-separating plane continuum $T(X)$ is a prime end $\mc{E}_t$ of
$\tU$ such that for some chain $\{\rg_i\}$ of crosscuts defining
$\mc{E}_t$, $\var(f,\rg_i,T(X))\not= 0$ for every $i$. We call an
outchannel $\mc{E}_t$ of $T(X)$ a {\em geometric outchannel}
\index{outchannel!geometric} if and only if for sufficiently small $\delta$,
every chord in $\KP_\delta$, which crosses  $\mc{E}_t$ essentially,
has nonzero variation. We call a geometric outchannel {\em negative}
\index{geometric outchannel!negative}(respectively, {\em positive})
\index{geometric outchannel!positive} (starting at $\rg\in\kp$) if and only if
every $\kp$-chord  $\rh \subset T(X)_\eta \cap \ol{\Sh(\rg)}$, which
crosses $\mc{E}_t$ essentially, has negative (respectively, positive)
variation.
\end{defn}


\section{Fixed points in invariant continua}

In this Section we describe the results obtained in
Section~\ref{sec:fxpt} of Chapter~\ref{ch:fxpt}. The main result of
Section~\ref{sec:fxpt} solves the Plane Fixed Point Problem in the
affirmative for positively oriented maps of the plane. Namely, the
following theorem is proven.

\setcounter{chapter}{7}\setcounter{section}{1}\setcounter{thm}{2}

\begin{thm}
Suppose $f:\Complex\to\Complex$ is a positively oriented map and $X$ is a
continuum such that $f(X)\subset T(X)$. Then there exists a point $x_0\in T(X)$
such that $f(x_0)=x_0$.
\end{thm}

\setcounter{chapter}{5}\setcounter{section}{2}\setcounter{thm}{10}

\section{Fixed points in non-invariant continua -- the case of dendrites}

As described in Chapter~\ref{intro}, in the rest of
Chapter~\ref{ch:fxpt} we want to extend Theorem~\ref{fixpoint} to at
least some non-invariant continua. We are motivated by the interval
case in which to conclude that there exists a fixed point in an
interval it is enough to know that the endpoints of the interval map
in opposite directions, and the invariantness of the interval itself
is not crucial.

In Section~\ref{sec:dendr} we extend, in the spirit of the interval case,
a well-known result according to which a map of a dendrite into itself has a
fixed point (Theorem~\ref{th:dendr}, see \cite{nadl92}). We show the existence
of fixed points in non-invariant dendrites and, with some additional
conditions, obtain also results related to the number of periodic points of
$f$. To state the precise results we need some definitions.

\begin{defn}[Boundary scrambling for dendrites]\index{boundary scrambling!for dendrites}\label{bouscr}
Suppose that $f$ maps a dendrite $D_1$ to a dendrite $D_2\supset D_1$. Put
$E=\ol{D_2\sm D_1}\cap D_1$ (observe that $E$ may be infinite). If for each
\emph{non-fixed} point $e\in E$, $f(e)$ is contained in a component of
$D_2\sm\{e\}$ which intersects $D_1$, then we say that $f$ has the
\emph{boundary scrambling property} or that it \emph{scrambles the boundary}.
Observe that if $D_1$ \emph{is} invariant then $f$ automatically scrambles the
boundary.
\end{defn}

The following theorem is the first result obtained in Section~\ref{sec:dendr}.



\setcounter{chapter}{7}\setcounter{section}{2}\setcounter{thm}{1}

\begin{thm}
Let $f:D_1\to D_2$ be a map, where $D_1$ and $D_2$ are dendrites and $D_1\subset D_2$.
The following claims hold.

\begin{enumerate}

\item
If $a, b\in D_1$ are such that $a$ separates $f(a)$ from $b$ and $b$ separates
$f(b)$ from $a$, then there exists a fixed point $c\in (a, b)$. Thus, if
$e_1\ne e_2\in E$ are such that each $f(e_i)$ belongs to a component of $D_2\sm \{e_i\}$
disjoint from $D_1$ then there is a fixed point $c\in (e_1, e_2)$.

\item
If $f$ scrambles the boundary, then $f$ has a fixed point.

\end{enumerate}

\end{thm}

\setcounter{chapter}{5}\setcounter{section}{3}\setcounter{thm}{1}

To give the next definition we recall that if $x\in Y$ then
the \emph{valence of $Y$ at $x$}, \index{valence} $\val_Y(x)$, \index{val@$\val_Y(x)$} is
defined as the number of connected components of $Y\sm \{x\}$, and $x$ is said
to be a \emph{cutpoint (of $Y$)} \index{cutpoint} if $\val_Y(x)>1$.

\begin{defn}[Weakly repelling periodic points]\index{periodic point!weakly repelling}\label{wkrep}
In the situation of Definition~\ref{bouscr}, let $a\in D_1$ be a fixed point and suppose that there exists a component $B$
of $D_1\sm \{a\}$ such that arbitrarily close to $a$ in $B$ there exist fixed
cutpoints of $D_1$ or points $x$ separating $a$ from $f(x)$. Then we say that
$a$ is a \emph{weakly repelling fixed point (of $f$ in $B$)}. A periodic point
$a\in D_1$ is said to be simply \emph{weakly repelling} \index{weakly repelling}if there exists $n$ and a
component $B$ of $D_1\sm \{a\}$ such that $a$ is a weakly repelling fixed point
of $f^n$ in $B$.
\end{defn}

We use the notions introduced in Definition~\ref{wkrep} to prove
Theorem~\ref{infprpt}.

\setcounter{chapter}{7}\setcounter{section}{2}\setcounter{thm}{5}

\begin{thm}\label{sumdend-1}
Suppose that $f:D\to D$ is continuous where $D$ is a dendrite and all its periodic points are weakly
repelling. Then $f$ has infinitely many periodic cutpoints.
\end{thm}

This theorem is applied in Theorem~\ref{lamwkrp} where it is shown that  if
$g:J\to J$ is a \emph{topological polynomial} on its dendritic Julia set (e.g.,
if $f$ is a complex polynomial with a dendritic Julia set) then it has
infinitely many periodic cutpoints.

\setcounter{chapter}{5}\setcounter{section}{3}\setcounter{thm}{2}
\section{Fixed points in non-invariant continua -- the planar case}

In parallel with the dendrite case, we want to extend
Theorem~\ref{fixpoint} to a larger class of maps of the plane and
non-invariant continua such that certain ``boundary'' conditions are
satisfied. This is accomplished in Section~\ref{sec:fxpt-noni}.

\begin{defn}\index{boundary scrambling!for planar continua}\label{scracon}
Suppose that $f:\C\to\C$ is a positively oriented map and $X\subset \C$ is a
non-separating continuum. Suppose that there exist $n\ge 0$ disjoint
non-separating continua $Z_i$ such that the following properties hold:

\begin{enumerate}

\item $f(X)\sm X\subset \cup_i Z_i$;

\item for all $i$,  $Z_i\cap X=K_i$ is a non-separating continuum;


\item for all $i$, $f(K_i)\cap [Z_i\sm K_i]=\0$.

\end{enumerate}

\noindent Then the map $f$ is said to \emph{scramble the boundary (of $X$}). If
instead of (3) we have

\begin{enumerate}

\item[(3a)] for all $i$, either $f(K_i)\subset K_i$, or
$f(K_i)\cap Z_i=\0$

\end{enumerate}

\noindent then we say that $f$  \emph{strongly scrambles the boundary (of
$X$)}; clearly, if $f$ strongly scrambles the boundary of $X$, then it
scrambles the boundary of $X$. In either case, the continua $K_i$ are called
\emph{exit continua (of $X$)}. \index{exit continuum}
\end{defn}

Observe that if in Definition~\ref{scracon}  $n=0$, then $X$ must be invariant
(i.e., $f(X)\subset X$).

\begin{rem}\label{zxgrow}
Since $Z_i$ and $Z_i\cap X=K_i\ne \0$ are non-separating continua and sets
$Z_i$ are pairwise disjoint, then $X\cup (\bigcup Z_i)$ is a non-separating
continuum. Loosely, scrambling the boundary means that $f(X)$ can only ``grow''
off $X$ \emph{within} the sets $Z_i$ and \emph{through} the sets $K_i\subset X$
while any set $K_i$ itself cannot be mapped outside $X$ within $Z_i$, with more
specific restrictions upon the dynamics of $K_i$'s in the case of strong
scrambling.
\end{rem}

The following theorem extends Theorem~\ref{fixpoint} onto some non-invariant continua.


\setcounter{chapter}{7}\setcounter{section}{3}\setcounter{thm}{2}

\begin{thm} 
In the situation of ~\ref{scracon},
if $f$ is a positively oriented map which strongly scrambles the boundary of $X$, then $f$ has a
fixed point in $X$.
\end{thm}

\setcounter{chapter}{5}\setcounter{section}{4}\setcounter{thm}{2}

We specify the above theorem for positively oriented maps with isolated fixed
points as follows. Given a non-separating continuum $X\subset \C$, a positively
oriented map $f$ and a fixed point $p\in X$, we define what it means that $f$
\emph{repels outside $X$ at $p$} (see Definition~\ref{repout}; basically, it
means that there exists an invariant external ray to $X$ which lands at $p$ and
along which the points are repelled away from $p$ by $f$). We also need the next
definition which is closely related to that of the index of the map on a simple
closed curve.

\begin{defn}\label{indpt}
Suppose that $f:\C\to\C$ is a positively oriented map with isolated fixed points and
$x$ is a fixed point of $f$.
Then the \emph{local index of $f$ at $x$}, \index{local index} \index{index!local} denoted by
$\ind(f, x)$, \index{index@$\ind(f,x)$} is defined as $\ind(f,S)$
where $S$ is a small simple closed curve around $x$.
\end{defn}

Then we prove the following theorem.



\setcounter{chapter}{7}\setcounter{section}{4}\setcounter{thm}{7}

\begin{thm}
Suppose that $f:\C\to\C$ is a positively oriented map with isolated fixed points, and
$X\subset\C$ is a non-separating continuum or a point. Suppose that the conditions (1)-(3) in
\ref{scracon} are satisfied. Moreover, suppose that
the following conditions hold.

\begin{enumerate}

\item
For each fixed point $p\in X$ we have that $\ind(f, p)=1$ and $f$ repels
outside $X$ at $p$.

\item
The map $f$ scrambles the boundary of $X$. Moreover, for each $i$
either $f(K_i)\cap Z_i=\0$, or there exists a neighborhood $U_i$ of
$K_i$ with $f(U_i\cap X)\subset X$.

\end{enumerate}

Then $X$ is a point.
\end{thm}
\setcounter{chapter}{5}\setcounter{section}{4}\setcounter{thm}{1}

\section{The polynomial case}\label{polycase}

Theorems~\ref{lamwkrp} and \ref{locrot} apply to polynomials acting
on the complex plane. These theorems allow us to obtain  corollaries
dealing with the existence of periodic points in certain parts of the
Julia set of a polynomial and with the degeneracy of certain continua
(e.g., impressions). To discuss this we need the following standard
notation.

Suppose that $P:\C\to\C$ is a complex polynomial of degree $d$. A
$P$-periodic point $a$ of period $n$ is called \emph{repelling}
\index{periodic point!repelling} if $|(P^n)'(a)|>1$,
\emph{parabolic} if $(P^n)'(a)$ is a root of unity\index{periodic
point!parabolic} (i.e., for an appropriate $k$ we will have
$[{(P^{n})'(a)}]^k=1$) and \emph{irrational neutral} if
$(p^n)'(a)=e^{2\pi \alpha i}$ with $\alpha$ irrational. The closure
of the union of all repelling periodic points of $P$ is called the
\emph{Julia set} of $P$ \index{Julia set} and is denoted by $J_P$
\index{JP@$J_P$}. Then the set $\iu(J_P)=\iu$ (i.e., the unbounded
component of $\C\sm J_P$) is called the \emph{basin of attraction of
infinity} and the set $K_P=\C\sm \iu=T(J_P)$ \index{KP@$K_P$}is
called the ``filled-in'' Julia set.

Components of $\C\sm J_P$ are called \emph{Fatou domains}.
\index{domain!Fatou} A Fatou domain is said to be \emph{attracting
(Siegel, respectively)}\index{domain!attracting}
\index{domain!Siegel} if it contains a periodic point which is
attracting (irrational neutral, respectively); an irrational neutral
periodic point like that is said to be a \emph{Siegel (periodic)
point}\index{periodic point!Siegel}. A bounded periodic Fatou domain
is said to be \emph{parabolic} \index{domain!parabolic} if it
contains no periodic points (in this case all its points converge to
the same parabolic periodic orbit which meets the boundary of the
domain). Finally, an irrational neutral periodic point which belongs
to $J_P$ is said to be a \emph{Cremer (periodic)
point}\index{periodic point!Cremer}.

The set $\iu$ is foliated by so-called \emph{(conformal) external rays} $R_\al$
\index{external ray!smooth}\index{conformal!external ray}
of arguments $\al\in \uc$. By \cite{douahubb85}, if the degree of $P$ is $d$
and $\sigma_d:\C\to\C$ is defined by $\sigma_d(z)=z^d$, then
$P(R_\al)=R_{\sigma_d(\al)}$. Denote by $C_*$ the set of all preimages of critical points in
$U^\infty(J_P)$ ($C_*=\0$ if $J_P$ is connected). If $J_P$ is connected,  each
$\al\in\uc$ corresponds to a unique external ray and all external rays are  smooth and
pairwise disjoint.   In general $R_\al$ is smooth and unique if and only if
$R_\al\cap C_*=\0$. Other external rays are one-sided limits of smooth rays; it follows
that they are non-smooth and there are at most countably many of them (in fact,
for each $\al\in\uc$ there exist at most two external rays $R_\al^\pm$ with
argument $\al$ and each is a one sided limit of smooth external rays, see
\cite{lepr} for further details).\index{external ray!non smooth}

It is known that two distinct external rays are not homotopic in the
complement of $K_P$ (with the landing point fixed  under the
homotopies). Given an external ray $R_\al$ of $K_P$, we denote by
$\Pi(R_\al)=\ol{R_\al}\sm R_\al$ the \emph{principal continuum of
$R_\al$}. Given a set $\RR$ of external rays, we extend the above
notation by setting $\Pi(\RR)=\bigcup_{R\in \RR} \Pi(R)$. Now we are
ready to give the following technical definition (see
Figure~\ref{puzzle-piece} for an illustration).

\begin{defn}[General puzzle-piece]\label{genpuz}
Let $P:\C\to\C$ be a polynomial. Let $X\subset K_P$ be  a
non-separating subcontinuum or a point such that the following holds.


\begin{enumerate}

\item 
There exists $m\ge 0$ and $m$
pairwise disjoint non-separating continua/points $E_1\subset X, \dots,
E_m\subset X$.
\item There exist $m$  finite sets of external rays $A_1=\{R_{a^1_1}, \dots, R_{a^1_{i_1}}\},
\dots, A_m=\{R_{a^m_1}, \dots, R_{a^m_{i_m}}\}$ 
 with $i_k\ge 2, 1\le k\le m$.

\item 
We have $\Pi(A_j)\subset E_j$ (so the set $E_j\cup
(\cup^{i_j}_{k=1} R_{a^j_k})=E'_j$ is closed and connected).

\item 
$X$ intersects a unique component $C_X$ of $\C\sm \cup E'_j$ \index{CX@$C_X$}.

\item For each Fatou domain $U$ either $U\cap X=\0$ or $U\subset X$.

\end{enumerate}

We call such $X$ a \emph{general puzzle-piece}\index{general puzzle-piece} and
call the continua $E_i$ the \emph{exit continua}\index{exit continuum!for a
general puzzle-piece} of $X$. For each $k$, the set $E'_k$ divides the plane
into $i_k$ open sets which we will call \emph{wedges (at $E_k$)}\index{wedge
(at an exit continuum)}; denote by $W_k$ the wedge which contains $X\sm E_k$
(it is well-defined by (4) above).
\end{defn}

Note that if $m=0$, $\bigcup E'_j=\0$ and $C_X=\C$; so, any
non-separating continuum in $K_P$ with the empty set of exit continua
satisfying (5) is a general puzzle-piece. Observe also, that there is
a natural situation in which general puzzle-pieces can occur. Suppose
that $J_P$ is connected, conditions (1) - (3) are satisfied, and all
continua $E_j$ are contained in $J_P$ while the continuum $X$ is not
yet defined. Suppose that there exists a component $C$ of $\C\sm
\bigcup E'_j$ such that the boundary of $C$ meets every $E_j, 1\le
j\le m$. Let $X=(C\cap K_P)\cup (\bigcup E_j)$. Then it is easy to
see that $X$ is a general puzzle-piece. However, our definition
allows for a wider variety of general puzzle-pieces (like, e.g.,
non-separating invariant subcontinua of $J_P$).

For convenience call a fixed point $x$ of a
polynomial $P$ \emph{non-rotational}\index{fixed point!non-rotational}
if there is a fixed external ray landing at $x$ (it follows that
each such point is either repelling or parabolic).
We are ready to state the main result of Section~\ref{complappl}.



\setcounter{chapter}{7}\setcounter{section}{5} \setcounter{thm}{1}

\begin{thm}
Let $P$ be a polynomial with filled-in Julia set $K_P$ and let $Y$
be a non-degenerate periodic component of $K_P$ such that $P^p(Y)=Y$.
Suppose that $X\subset Y$ is a non-degenerate general puzzle-piece
with $m\ge 0$ exit continua $E_1, \dots, E_m$ such that $P^p(X)\cap
C_X\subset X$ and either $P^p(E_i)\subset W_i$, or $E_i$ is a
$P^p$-fixed point.
Then at least one of the following claims holds:

\begin{enumerate}

\item $X$ contains a  $P^p$-invariant parabolic domain,

\item $X$ contains a $P^p$-fixed point  which is neither repelling nor
    parabolic, or

\item $X$ has an external ray $R$ landing at a repelling or parabolic
    $P^p$-fixed point such that $P^p(R)\cap R=\0$ (i.e., $P^p$ locally
    rotates at some parabolic or repelling $P^p$-fixed point).

\end{enumerate}

Equivalently, suppose that $Y$ is a non-degenerate periodic
component of $K_P$ such that $P^p(Y)=Y$, $X\subset Y$ is a general
puzzle-piece with $m\ge 0$ exit continua $E_1, \dots, E_m$ such that
$P^p(X)\cap C_X\subset X$ and either $P^p(E_i)\subset W_i$, or $E_i$
is a $P^p$-fixed point; if, moreover, $X$ contains only
non-rotational $P^p$-fixed points and does not contain
$P^p$-invariant parabolic domains, then it is degenerate.
\end{thm}

\setcounter{chapter}{5}\setcounter{section}{0} \setcounter{thm}{17}

We also prove in Corollary~\ref{degimpr} that an impression of an invariant
external ray, to the filled in Julia set, which contains only repelling or parabolic periodic points is
degenerate.

\chapter{Outchannels and their properties}\label{outcha}

\section{Outchannels} \label{out}
In this section we will always let $f:\C\to\C$ be a continuous
function. Suppose that $X$ is a minimal continuum such that
$f(X)\subset T(X)$ and $f$ has no fixed point in $T(X)$.  We show
that $X$ has at least one \textit{negative outchannel}. We will
always assume that $(f,X,\eta)$ satisfies the standing hypothesis
(see Definition~\ref{outch} and the paragraph preceding \ref{outch})
and see Section~\ref{auxcont} for the notation $T(X)_\delta^\pm$).
In particular, $f$ is fixed point free on $T(X)_\eta$. Note that for
each $\kp$-chord $\rg$ in $T(X)_\eta$,
$\var(f,\rg,T(X))=\var(f,\rg)$ is defined.

\begin{lem} \label{localarcs}
Suppose that $(f,X,\eta)$ satisfy the standing hypothesis and $\delta\le\eta$.
Let $Z\in\{T(X)^+_\delta,T(X)^-_\delta\}$. Fix a Riemann map
$\varphi:\disk^\infty\to\rsphere\sm Z$ such that $\varphi(\infty)=\infty$.
Suppose $R_t$ lands at $x\in\bd Z$. Then there is an open interval
$M\subset\bd\disk^\infty$ containing $t$ such that $\varphi$ can be extended
continuously over $M$. \end{lem}

\begin{proof}
Suppose that $Z=T(X)_\delta^-$ and $R_t$ lands on $x\in\partial Z$.
By proposition~\ref{boundary} we may assume that $x\in X$.  Note
first that the family of chords in $\kp^-_\delta$ form a closed
subset of the hyperspace of $\complex\sm X$, by
Proposition~\ref{compactness}. By symmetry, it suffices to show that
we can extend $\psi$ over an interval $[t',t]\subset \uc$ for
$t'<t$.

Let $\phi:\disk^\infty\to \complex\sm T(X)$ be the Riemann map for
$T(X)$. Then there exists $s\in \uc$ so that the external ray $R_s$
of $\complex\sm T(X)$ lands at $x$. Suppose first that there exists
a chord $\rg\in \kp^-_\delta$ such that $G=\vp^{-1}(\rg)$ has
endpoints $s'$ and $s$ with $s'<s$. Since $\kp^-_\delta$ is closed,
there exists a minimal $s"\le s'<s$ such that there exists a chord
$\rh\in\kp^-_\delta$ so that $H=\vp^{-1}(\rh)$ has endpoints $s"$
and $s$. Then $\rh\subset \partial Z$ and $\phi$ can be extended
over an interval $[t',t]$ for some $t'<t$, by
Proposition~\ref{boundary} (4).

Suppose next that no such chord $\rg$ exists. Choose a junction $J_x$ for
$T(X)^-_\delta$ and a neighborhood $W$ of $x$ such that $f(W)\cap [W\cup
J_x]=\0$. We will first show that there exists $\nu\le \delta$ such that $x\in
\partial T(X)_\nu$.   For suppose that this is not
the case. Then there exists a sequence $\rg_i\in\KP$ of chords such
that $x\in\Sh(\rg_{i+1})\subset\Sh(\rg_i)$, $\lim \rg_i=x$ and
$\var(f,\rg_i)>0$ for all $i$. This contradicts
Proposition~\ref{var0}. Hence $x\in\partial T(X)_\nu$ for some
$\nu>0$.  We may assume that $\nu$ is so small that any chord of
$\kp_\nu$ with endpoint $x$ is contained in $W$.

By Proposition~\ref{LC}, the boundary of $T(X)_\nu$ is a simple
closed curve $S$ which must contain $x$. If there exists a chord
$\rh\in\kp_\nu$  with endpoint $x$ such that $H$  has endpoints $s'$
and $s$ with $s'<s$ then,  since $\rh\subset W$, $f(\rh)\cap
J_x=\0$, $\var(f,\rh)=0$ and $\rh\in \kp^-_\delta$, a contradiction.
Similarly, all chords $\rh$  close to $x$ in $S$ so that $H$ has
endpoints less than $s$ and which are contained in $W$ have
$\var(f,\rh)=0$ by Proposition~\ref{var0}. Hence a small interval
$[x',x]\subset S$, in the counterclockwise order on $S$ is contained
in $T(X)^-_\nu$. It now follows easily that a similar arc exists in
the boundary of $T(X)^-_\delta$ and the desired result follows.
\end{proof}

By a \emph{narrow strip} we mean the image of an embedding
$h:\{(x,y)\in\Complex\mid x\ge 0 \text{ and } -1< y< 1\}\to
\Complex$ such that $h$ has a continuous extension over the closure
of its domain and $\lim_{x\to\infty}\dm (h(\{x\}\times [-1,1]))=0$.

\begin{figure}

\includegraphics{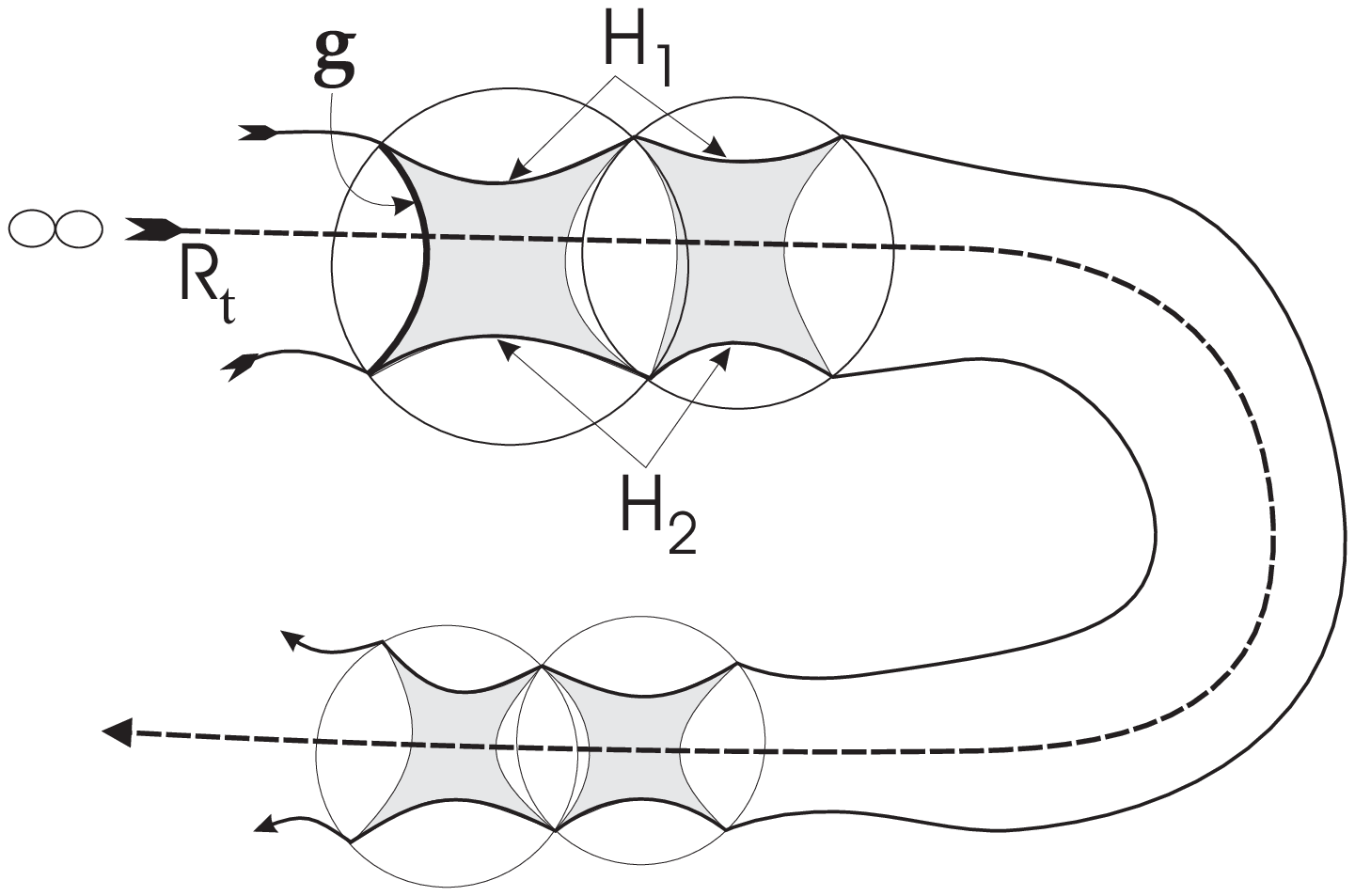}

\caption{The strip $\fS$ from Lemma~\ref{geomoutchannel}} \label{strip}

\end{figure}

\begin{lem} \label{geomoutchannel}
Suppose that $(f,X,\eta)$ satisfy the standing hypothesis. If there
is a chord $\rg\subset T(X)_\eta$ of $T(X)$ of negative
(respectively, positive) variation, such that there is no fixed
point in $T(T(X)\cup \rg)$, then there is a negative (respectively,
positive) geometric outchannel $\mc{E}_t$ of $T(X)$ starting at
$\rg$.

Moreover, if $\mc{E}_t$ is a positive (negative) geometric outchannel starting
at the $\kp$-chord $\rg$, and $\fS=\bigcup \{\HCH(B\cap T(X)) \mid \HCH(B\cap
T(X))\subset T(X)_\eta\cap \ol{\Sh(\rg)} \text{ and a chord in } \HCH(B\cap
T(X)) \text{ crosses } R_t \text{ essentially}\}.$ Then $\fS$ is an infinite
narrow strip in the plane whose remainder is contained in $T(X)$ and which is
bordered by a $\kp$-chord and  two halflines $H_1$ and $H_2$ (see
figure~\ref{strip}).
\end{lem}

\begin{proof}
Without loss of generality, assume $\var(f,\rg,T(X))=\var(f,\rg)<0$.
If $\rg$ is such that for any chord $\rh\subset T(X\cup\rg)$,
$\rh\subset T(X)_\eta$,  put $\rg'=\rg$. Otherwise consider the
boundary of $T(X)_\delta$ ($\delta<\eta$) which is locally connected
by Proposition~\ref{LC} and, hence, a Carath\'eodory loop. Then a
continuous extension $g:\uc\to\partial T(X)_\delta$ of the Riemann
map $\phi:\disk^\infty\to \rsphere\sm T(X)_\delta$ exists. Whence
the boundary of $T(X)_\delta$ contains a sub-path $A=g([a,b])$,
which is contained in $\ol{\Sh(\rg)}$, whose endpoints coincide with
the endpoints of $\rg$. Note that for each component $C$ of $A\sm
X$, $\var(f,C)$ is defined. Then it follows from
Proposition~\ref{crossunder}, applied to a Carath\'eodory path, that
there exists a component $C=\rg'$ such that $\var(f, \rg')<0$. Note
that $\rg'$ is a $\kp$-chord contained in the boundary of
$T(X)_\delta$.  By taking $\delta$ sufficiently small we can assume
that for any chord $\rh\subset  \ol{\Sh(\rg')}$, $\rh\subset
T(X)_\eta$.

To see that a geometric outchannel, starting with $\rg'$ exists,
note that for any chord $\rg''\subset \ol{\Sh(\rg')}$ with
$\var(f,\rg'',X)<0$, if $\rg''=\lim \rg_i$, then there exists $i$
such that for any chord $\rh$ which separates $\rg_i$ and $\rg''$ in
$U^\infty(X)$, $\var(f, \rh,X)=\var(f,\rg'',X)<0$. This follows
since $f(\rg)$ is close to $f(\rh)$ and, hence, crosses a junction
$J_v$ in the same way (we can slightly change the junction $J_v$
with $v\in\rg''$ to a junction with vertex in $\rh$ without changing
the crossings of the images of the crosscuts with the junction). If
$\rg''$ is isolated on the side closest to $X$, then $\rg''\subset
\HCH(B\cap T(X))$, where $\HCH(B\cap T(X))$ is a gap, such that
$\rg''$ separates $\HCH(B\cap T(X))\sm\rg''$ from infinity in
$U^\infty(X)$. Again by Proposition~\ref{crossunder}, there exists
$\rh\ne\rg''$ in $\HCH(B\cap T(X))$ such that $\rg''$ separates
$\rh$ from infinity in $U^\infty(X)$ and $\var(f,\rh,X)<0$.  It
follows from these two facts that there exists a maximal family of
$\kp$ crosscuts, all of which have negative variation and are such
that that for any three members of the family, one separates the
other two in $U^\infty(X)$. Hence this maximal family determines a
geometric outchannel. Each chord $\rh$ in this family corresponds to
a unique maximal ball $B_\rh$. It is now not difficult to see that
the union of all the sets $\HCH(B_\rh \cap T(X))$ is a narrow strip.
\end{proof}

\subsection{Invariant Channel in $X$}
We are now in a position to prove Bell's principal result on any possible
counter-example to the fixed point property, under our standing hypothesis.

\begin{lem} \label{invchannel}
Suppose $\mc{E}_t$ is a geometric outchannel of $T(X)$ under
$f$. Then the principal continuum $\pr(\mc{E}_t)$ of $\mc{E}_t$ is
invariant under $f$. So $\pr(\mc{E}_t)=X$. \end{lem}

\begin{proof}
Let $x\in\pr(\mc{E}_t)$. Then for some chain $\{\rg_i\}_{i=1}^\infty$ of
crosscuts defining $\mc{E}_t$ selected from $\KP_\delta$, we may suppose
$\rg_i\to x\in \bd T(X)$ (by Lemma~\ref{chordlimit}) and
$\var(f,\rg_i,X)\not=0$ for each $i$. The external ray $R_t$ meets all $\rg_i$
and there is, for each $i$, a junction from $\rg_i$  which ``parallels" $R_t$.
Since $\var(f,\rg_i,X)\not=0$, each $f(\rg_i)$ intersects $R_t$. Since
$\dm(f(\rg_i))\to 0$, we have $f(\rg_i)\to f(x)$ and $f(x)\in\pr(\mc{E}_t)$. We
conclude that $\pr({\mc{E}_t})$ is invariant.\end{proof}

\begin{thm} [Dense channel, Bell] \label{densechannel}\index{channel!dense}
If $(X,f,\eta)$ satisfy our standing hypothesis then $T(X)$
contains a negative geometric outchannel; hence, $\bd \tU=\bd T(X)=X=f(X)$ is
an indecomposable continuum.
\end{thm}

\begin{proof}
By Lemma~\ref{LC} $\bd T(X)_\eta$ is a Carath\'eodory loop. Since $f$ is fixed
point free on $T(X)_\eta$, $\ind(f,\bd T(X)_\eta)=0$. Consequently, by
Theorem~\ref{I=V+1} for Carath\'eodory loops, $\var(f,\bd T(X)_\eta)=-1$. By
the summability of variation on $\bd T(X)_\eta$, it follows that on some chord
$\rg\subset \bd T(X)_\eta$, $\var(f,\rg,T(X)) <0$. By
Lemma~\ref{geomoutchannel}, there is a negative geometric outchannel $\mc{E}_t$
starting at $\rg$.

Since $\pr(\mc{E}_t)$ is invariant under $f$ by
Lemma~\ref{invchannel}, it follows that $\pr(\mc{E}_t)$ is an
invariant subcontinuum of $\bd \tU\subset\bd T(X)\subset X$. So by
the minimality condition in our Standing Hypothesis, $\pr(\mc{E}_t)$
is dense in $X$.  It then follows from a theorem of Rutt
\cite{rutt35} that $X$ is an indecomposable continuum.
\end{proof}

\begin{thm}\label{nicebd} Assume that
 $(X,f,\eta)$ satisfy  our standing hypothesis and $\delta\le
\eta$. Then the boundary of $T(X)_\delta$ is a simple closed curve.
The set of accessible points in the boundary of each of
$T(X)^+_\delta$ and $T(X)^-_\delta$ is an at most countable union of
pairwise disjoint continuous one-to-one images of $\real$.
\end{thm}

\begin{proof}
By Theorem~\ref{densechannel}, $X$ is indecomposable, so it has no cut points.
By Proposition~\ref{LC}, $\bd T(X)_\delta$ is a Carath\'eodory loop. Since $X$
has no cut points, neither does $T(X)_\delta$.  A Carath\'eodory loop without
cut points is a simple closed curve.

Let $Z\in\{T(X)_\delta^+,T(X)_\delta^-\}$ with $\delta\le \eta$.  Fix a Riemann
map $\phi:\disk^\infty\to\rsphere\sm Z$ such that $\phi(\infty)=\infty$.
Corresponding to the choice of $Z$, let
$\mc{W}\in\{\KP^+_\delta,\KP^-_\delta\}$.  Apply Lemma~\ref{localarcs} and find
the maximal collection $\mc{J}$ of disjoint open subarcs of $\bd\disk^\infty$
over which $\phi$ can be extended continuously. The collection $\mc{J}$ is
countable. Since $X$ has no cutpoints the extension is one-to-one over $\cup\mc
J$. Since angles that correspond to accessible points are dense  in
$\bd\disk^\infty$, so is $\cup\mc{J}$. If $Z=T(X)_\delta^+$, then it is
possible that $\cup\mc{J}$ is all of $\bd\disk^\infty$ except one point, but it
cannot be all of $\bd\disk^\infty$ since there is at least one negative
geometric outchannel by Theorem~\ref{densechannel}.
\end{proof}

Theorem~\ref{nicebd} still leaves open the  possibility that
$Z\in\{T(X)_\delta^+,T(X)_\delta^-\}$ has a very complicated boundary. The set
$C=\bd\disk^\infty\setminus\cup\mc J$ is compact and zero-dimensional. Note
that $\phi$ is discontinuous at points in $C$. We may call $C$ the set of
outchannels of $Z$. In principle, there could be an uncountable set of
outchannels, each dense in $X$. The one-to-one continuous images of half lines
in  $\real$ lying in $\bd Z$ are the ``sides" of the outchannels. If two
elements $J_1$ and $J_2$ of the collection $\mc{J}$ happen to share a common
endpoint $t$, then the prime end $\mc{E}_t$ is an outchannel in $Z$, dense in
$X$, with images of half lines $\phi(J_1)$ and $\phi(J_2)$ as its sides.  It
seems possible that an endpoint $t$ of $J\in\mc{J}$ might have a sequence of
elements $J_i$ from $\mc{J}$ converging to it.  Then the outchannel $\mc{E}_t$
would have only one (continuous) ``side." Such exotic possibilities are
eliminated in the next section.

In the proposition below we summarize several of the results in this
section and show that an arc component $K$  of the set of accessible
points of the boundary of $T(X)_{\delta}^-$ is efficient in
connecting close points in $K$. Note that it will follow later from
Theorem~\ref{outchannel} that there are no chords of positive
variation. Hence $T(X)^-_\delta=T(X)_\delta$ which is  always a
simple closed curve.

\begin{prop}\label{smallarc}
Suppose that $(X,f,\eta)$ satisfy  our standing hypothesis, that the
boundary of $T(X)_\delta^-$ is not a simple closed curve,
$\delta\le\eta$ and that $K$ is an  arc component of the boundary of
$T(X)_\delta^-$ so that $K$ contains an accessible point. Let
$\varphi:\disk^\infty\to \rsphere\sm T(X)_\delta^-$ be a conformal
map such that $\varphi(\infty)=\infty$. Then:

\begin{enumerate}
\item \label{extpsi}
$\varphi$ extends continuously and injectively to a map
$\tilde{\varphi}:\tilde{\disk}^{\infty}\to \tilde{U}^{\infty}$,
where $\tilde{\disk}^{\infty}\sm \disk^\infty$ is a dense and open
subset of $\uc$ which contains $K$ in its image.  Let
$\tilde{\varphi}^{-1}(K)=(t',t)\subset S^1$ with $t'<t$ in the
counterclockwise order on $\uc$. Hence $\tilde{\varphi}$ induces an
order $<$ on $K$. If $x<y\in K$, we denote by $\langle x,y\rangle $
the subarc of $K$ from $x$ to $y$ and by $\langle x,\infty\rangle
=\cup_{y>x} \langle x,y\rangle $.

\item
$\mc{E}_t$ and $\mc{E}_{t'}$ are positive geometric outchannels of $T(X)$.

\item
Let $R_t$ be the external ray of $T(X)_\delta^-$ with argument $t$.
There exists $s\in R_t$, $B\in\B^\infty$ and $\rg\in\kp$ such that
$s\in \rg\subset \HCH(B\cap X)$ and $s$ is the last point of $R_t$
in $\HCH(B\cap X)$ (from $\infty$), $\rg$ crosses $R_t$ essentially
and for each $B' \in \B^\infty$ with $\HCH(B'\cap X)\sm X\subset
\Sh(\rg)$, $\dm(B')<\delta$.

\item\label{conhul}
There exists $\hx\in K$ such that if $B'\in\B^\infty$ with
$\Int(B')\subset \Sh(\rg)$, then $\HCH(B'\cap X)\cap \langle
\hx,\infty\rangle $ is a compact ordered subset of $K$ so that if
$C$ is $\KP$-crosscut in the boundary of $\HCH(B'\cap X)$ with both
endpoint in $K$, then $C\subset K$.

\item \label{esscross}
Let $\B^\infty_t\subset\B^\infty$ be the collection of all $B\in\B^\infty$ such
that $ R_t$ crosses a chord in the boundary of $\HCH(B\cap X)$ essentially and
$\Int(B)\subset\Sh(\rg)$. Then $\fS=\bigcup_{B\in \B^\infty_t}  \HCH(B\cap X)$
is a narrow strip \index{narrow strip} in the plane, bordered by two halflines
$H_1$ and $H_2$, which compactifies on $X$ and one of $H_1$ or $H_2$ contains
the set $\langle \hx',\infty\rangle $ for some $\hx'\in K$.\\
In particular,  if $\max(\hx,\hat{x}')<p<q$ and $\dm(\langle
p,q\rangle )
>2\delta$, then there exists a chord $\rg\in\KP$ such that one
endpoint of $\rg$ is in $\langle p,q\rangle $ and $\rg$  crosses
$R_t$ essentially.
\end{enumerate}

An analogous conclusion holds for $T(X)^+_\delta$ since its boundary
cannot be a simple closed curve (clearly $T(X)_\delta$ must contain
a crosscut of negative variation).
\end{prop}

\begin{proof}
By Proposition~\ref{boundary} and Theorem~\ref{nicebd}, and its proof,
$\varphi$ extends continuously and injectively to a map
$\tilde{\varphi}:\tilde{\disk}^{\infty}\to \tilde{U}^{\infty}$ and
(\ref{extpsi}) holds.

By Lemma~\ref{localarcs}, the external ray $R_t$ does not land. Hence there
exist a  chain $\rg_i$ of $\KP_\delta$ chords which define the prime end
$\mc{E}_t$. If for any $i$ $\var(f,\rg_i)\le 0$, then $\rg_i\subset
T(X)_\delta^-$ a contradiction with the definition of $t$. Hence
$\var(f,\rg_i)>0$ for all $i$ sufficiently small and $\mc{E}_t$ is a positive
geometric outchannel by the proof of Lemma~\ref{geomoutchannel}. Hence (2)
holds.

The proof of (3) is straightforward and is left to the reader.

Suppose that the endpoints of $\rg$ are $e$ and $f$ with $f\in K$.
Choose $\hx>f$ in $K$ so that $\hx$ is the endpoint of a $\KP$
crosscut  which is contained in $\HCH(B\cap X)$ with $B\subset
\sh(\rg)$.  Let $B'\in\B^\infty$ with $\Int(B')\subset\Sh(\rg)$,
$\hx\not\in B'$ and $\langle \hx,\infty\rangle \sm \HCH(B'\cap X)$
not connected.  Suppose $\langle a, b\rangle$ is a bounded component
of $\langle \hx,\infty\rangle\sm \HCH(B'\cap X)$ with endpoints in
$B'$. Note that there must exist a chord $\rh\in\kp$ with endpoints
$a$ and $b$.  If $\var(f,\rh)\le 0$ we are done. Hence
$\var(f,\rh)>0$. By Lemma~\ref{geomoutchannel}, there is a geometric
outchannel $\mc{E}_{t"}$ starting at $\rh$. This outchannel
disconnects the arc $\langle a,b\rangle $ between $a$ and $b$, a
contradiction. Hence (4) holds.

Next choose $\hat{x}'\in K$ such that each point of  $\langle
\hat{x}',\infty\rangle $ is accessible from $\Sh(\rg)$. Then each
subarc $\langle p,q\rangle $ of $\langle \hat{x},\infty\rangle $ of
diameter bigger than $2\delta$ cannot be contained in a single
element of the $\KPP$ partition. Hence there exists a $\kp$-chord
$\rg$ which crosses $R_t$ essentially and has one endpoint in
$\langle p,q\rangle $.

Note that for each chord $\rh\subset \Sh(\rg)$ which crosses $R_t$
essentially, $\var(f,\rh)>0$. By Lemma~\ref{geomoutchannel},
$\bigcup_{B\in \B^\infty_t} \HCH(B\cap X)$ is a  strip in the plane,
bordered by two halflines $H_1$, $H_2$, which compactify on $X$.
These two halflines, consist of chords in $\kp_\delta$ and points in
$X$,   one of which, say $H_1$ meets $\langle \hx',\infty\rangle $.
If $\langle \hx',\infty\rangle $ is not contained in $H_1$ then, as
in the proof of (4), there exists a chord $\rh\subset H_1$ with
$\var(f,\rh)>0$ joining two points of $x,y\in\langle
\hx,\infty\rangle $. As above this leads to a contradiction
and the
proof is complete.
\end{proof}

\section{Uniqueness of the Outchannel}

Theorem~\ref{densechannel} asserts the existence of at least one negative
geometric outchannel which is dense in $X$.  We show below that there is
exactly one geometric outchannel, and that its variation is $-1$.  Of course,
$X$ could have other dense channels, but they are ``neutral" as far as
variation is concerned.

\begin{thm} [Unique Outchannel] \label{outchannel}\index{outchannel!uniqueness}
If $(X,f,\eta)$ satisfy the standing hypothesis then there exists a
unique geometric outchannel $\mc{E}_t$ for $X$, which is dense in $X=\bd T(X)$.
Moreover, for any sufficiently small chord $\rg$ in any chain defining
$\mc{E}_t$, $\var(f,\rg,X)=-1$, and for any sufficiently small chord $\rg'$ not
crossing $R_t$ essentially, $\var(f,\rg',X)=0$.
\end{thm}

\begin{figure}\label{f:outcha}

\includegraphics{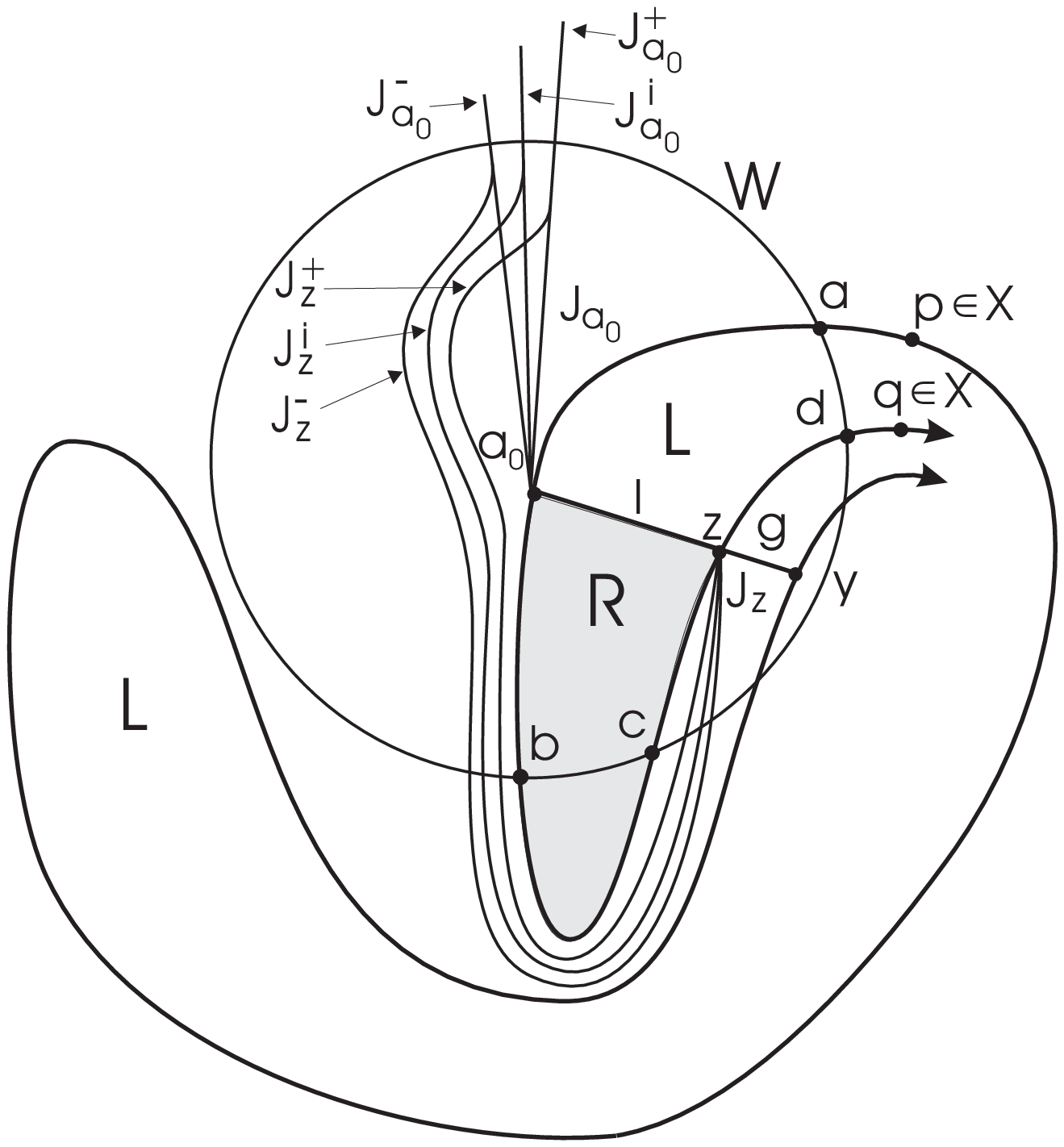}

\caption{Uniqueness of the negative outchannel.} \label{outpic}

\end{figure}

\begin{proof}
Suppose by way of contradiction that $X$ has a positive outchannel. Let
$0<\delta\le\eta$ such that  if $M\subset T(B(T(X),2\delta))$ with
$\dm(M)<2\delta$, then $f(M)\cap M=\0$. Since $X$ has a positive outchannel,
$\bd T(X)_{\delta}^-$ is not a simple closed curve. By Theorem~\ref{nicebd}
$\bd T(X)_{\delta}^-$ contains an arc component $K$ which is the one-to-one
continuous image of $\real$. Note that each point of $K$ is accessible.

Let $\varphi:\disk^\infty\to V^\infty=\complex\setminus
T(X)_{\delta}^-$ a conformal map.  By Proposition~\ref{smallarc},
$\varphi$ extends continuously and injectively to a map
$\tilde{\varphi}:\tilde{\disk}^{\infty}\to \tilde{V}^{\infty}$,
where $\tilde{\disk}^{\infty}\sm \disk^\infty$ is a dense and open
subset of $\uc$ which contains $K$ in its image. Then
$\tilde{\varphi}^{-1}(K)=(t',t)\subset \uc$ is an open arc with
$t'<t$ in the counterclockwise order on $\uc$ (it could be that
$(t',t)=\uc\sm\{t\}$ and $t=t'$). By abuse of notation, let $<$
denote the order in $K$ induced by $\tilde{\varphi}$ and for $x<y$
in $K$, denote the arc in $K$ with endpoints $x$ and $y$ by $\langle
x,y\rangle$. For $x\in K$, let $\langle x,\infty\rangle =\cup_{y>x}
\langle x,y\rangle $

Let $\mc{E}_t$ be the prime-end corresponding to $t$. By
Proposition~\ref{smallarc}, $\pr(\mc{E}_t)$ is a positive geometric
outchannel and, hence, by Lemma~\ref{invchannel}, $\pr(\mc{E}_t)=X$.
Let $R_t=\varphi(re^{it})$, $r>1$,  be the external conformal ray
corresponding to the prime-end $\mc{E}_t$ of $T(X)^-_\delta$. Since
$\cl{R_t}\sm R_t=X$ and  the small chords $\rg_x$ which define
$\pr({\mc{E}}_t)$ have at least one endpoint in $K$, cross $R_t$
essentially at $x$ and have diameter going to zero if the endpoint
in $K$ moves in the positive direction along $K$,  it follows that
$X=\ol{R_t}\sm R_t=\linebreak \ol{\langle x,\infty\rangle }\ \sm
\langle x,\infty\rangle $.

By Proposition~\ref{smallarc}, there is  $s\in R_t$ such that if
$B\in\B^\infty$ such that $\HCH(B\cap X)\cap [(X,s)$-end of
$R_t]\ne\0$, then $\dm(B)<\delta/2$.  Let
$\B^\infty_t=\{B\in\B^\infty\mid \HCH(B\cap X) \text{ contains a
chord } \rg \text{ such that } \rg \text{ crosses the } (X,s)
\text{-end of }  R_t \text{ essentially}\}$.

By Proposition~\ref{smallarc} there exists  $\hx\in K\cap X$ such
that for each arc $A\subset \langle \hx,\infty\rangle $ with
diameter $>2\delta$, there is a $\kp$-chord $\rg$ which contains a
point of $A$ as an endpoint and crosses $R_t$ essentially.

Let $a_0\in K\cap X$ so that $a_0>\hx$ and $J_{a_0}$ is a junction
of $T(X)_\delta^-$. Let $W$ be an open disk, with simple closed
curve boundary, about $a_0$ such that $\dm(W)<\delta/4$ and
$f(\ol{W})\cap [\ol{W}\cup J_{a_0}]=\0$. Let $a<a_0<b$ in
$K\cap\partial W$ such that $\langle a,b\rangle $ is the component
of $K\cap\ol{W}$ which contains $a_0$. We may suppose that $\langle
b,\infty\rangle \cap W$ is contained in one component of
$W\sm\langle a,b\rangle $ since one side of $K$ is accessible from
$\complex\sm T(X)^-_\delta$ and $a_0\in X$. If $a\in X$, let $p=a$.
If not, then there exists a $\kp$-chord $\rh\subset K$ such that
$a\in\rh$. Let $p$ be the endpoint of $\rh$ such that $p<a$. See Figure~\ref{f:outcha}.

Since $X\subset \ol{\langle x,\infty\rangle }$ there are components
of $\langle b,\infty\rangle \cap W$ which are arbitrarily close to
$a_0$. Choose $b<c<d$ in $K$ so that the $\langle c,d\rangle $ is
the closure of a component of $W\cap\langle b,\infty\rangle $ such
that:

\begin{enumerate}
\item $a$ and $d$ lie in the same component of $\partial W\sm\{b,c\}$.

\item
There exists $z\in\langle c,d\rangle \cap X\cap W$ and an arc
$I\subset \{a_0,z\}\cup [W\sm\langle p,d\rangle ]$ joining $a_0$ to
$z$.

\item
There is a $\kp$-chord $\rg\subset W$ with $z$ and $y$ as endpoints which
crosses $R_t$ essentially. Hence, $\var(f,\rg)>0$.

\item $\dm(f(\rg))<d(J^+_{a_0}\sm W,J^i_{a_0}\sm W)$.
\end{enumerate}

Conditions (1) and (2) follow because $J_{a_0}$ is a connected and
closed set from $a_0$ to $\infty$ in $\{a_0\}\cup[\complex\sm
T(X)_\delta^-]$ and the ray $\langle b,\infty\rangle $ approaches
both $a_0$ and $p$. Conditions (3) and (4) follow from
Proposition~\ref{smallarc}. If $d\in X$, put $q=d$. Otherwise, let
$q\in\langle d,\infty\rangle $ such that there is a $\kp$-chord
$\rh\subset K$ containing $d$ with endpoint $q$.

By  Corollary~\ref{bumpingscc}, there exists a bumping arc $A'$ of
$T(X)$ from $p$ to $q$ such that variation is defined on each
component of $A'\sm X$, $S'=A'\cup \langle p,q\rangle $ is a simple
closed curve with $T(X)\subset T(S')$ and $f$ is fixed point free on
$T(S')$. Since $\ol{\rg}\cap X=\{z,y\}$, we may assume that $A'\cap
\ol{\rg}=\{y\}$. Let $C$ be the arc in $\partial W$ from $a$ to $d$
disjoint from $b$. The arc $A'$ may enter $W$ and intersect $I$
several times. However, in this case $A'$ must enter $W$ through
$C$.  Since we  want to apply the Lollipop lemma, we will modify the
arc $A'$ to a new arc $A$ which is disjoint from $I$.

Let $A$ be the set of points in $A'\cup C$ accessible from $\infty$
in $\complex\sm [S'\cup C]$. Then $A$ is a bumping arc from $p$ to
$q$, $A\cap I=\0$, $\var(f,A)$ is defined,  $S=A\cup \langle
p,q\rangle $ is a simple closed curve with $T(X)\subset T(S)$ and
$f$ is fixed point free on $T(S)$. Note that $y\in A$. Then the
Lollipop lemma applies to $S$ with $R=T(\langle a_0,z\rangle \cup
I)$ and $L=T(I\cup \langle z,q\rangle \cup A\cup \langle
p,a_0\rangle )$.

\smallskip
Claim: $f(z)\in R$.  Hence by Corollary~\ref{corlol}, $\langle
a_0,z\rangle $ contains a chord $\rg_1$ with $\var(f,\rg_1)<0$.

\smallskip
\noindent Proof of Claim. Note that the positive direction along
$\rg$ is from $z$ to $y$.  Since $z,y\in X$, $\{f(z),f(y)\}\subset
X\subset T(S)=R\cup L$. Choose a junction $J_z$ such that
$J_{a_0}\sm W\subset J_z$ and $J_z$ runs close to $\langle
a_0,z\rangle$ on its way to $\rg$. In particular we may assume that
$J_z\cap R=\{z\}$. Since $\rg$ crosses $R_t$ essentially,
$\var(f,\rg)>0$. For $*\in\{-,i,+\}$, let $C^*_z$ be the union of
components of $J^*_z\sm W$ which are disjoint from $J^*_{a_0}$. Then
$C^i_z$ separates $R\cup C^+_z$ from $L\cup C^-_z$ in $\complex\sm
W$ (see figure~\ref{outpic}).  Since $f(\rg)\cap J_{a_0}=\0$, if
$f(z)\not\in R$, $\var(f,\rg)\le 0$, a contradiction. Hence $f(z)\in
R$ (and, in fact, $f(y)\in L$) as desired.

Since  $f(z)\in R$, $\langle a_0,z\rangle $ contains a chord $\rg_1$
with $\var(f,\rg_1)<0$. Repeating the same argument, replacing $a_0$
by $z$ we obtain a second chord $\rg_2$ contained in $\langle
z,\infty\rangle $ such that $\var(f,\rg_2)<0$.

We will now show that the existence of two distinct chords $\rg_1$
and $\rg_2$ in $K$ with variation $<0$ on each leads to a
contradiction. Recall that $a_0\in\ol{\langle b,\infty\rangle }$.
Hence we can find $y'\in\,\langle b,\infty\rangle $ with $y'\in X$
such that $\rg_1\cup\rg_2\subset\  \langle a_0,y'\rangle $ and there
exists a small arc $I'\subset W$  such that $I'\cap \langle
a_0,y'\rangle =\{a_0,y'\}$. Since $f(I')\cap J_{a_0}=\0$,
$\var(f,I')=0$. We may also assume that $f$ is fixed point free on
$T(S'')$, where $S''=I'\cup \langle a_0,y'\rangle $. Since $\langle
a_0,y'\rangle $ contains both $\rg_1$, $\rg_2$ and no chords of
positive variation, $\var(f,\langle a_0,y'\rangle )\le -2$ and
$\var(f,S'')\le -2$. Then $\ind(f,S'')=\var(f,S'')+1\le -1$ a
contradiction with Theorem~\ref{fpthm}. Hence $X$ has no positive
geometric outchannel.

By Theorems~\ref{densechannel} and ~\ref{I=V+1}, $X$ has exactly one
negative outchannel and its variation is $-1$.
\end{proof}

Note that the following Theorem  follows from Lemma~\ref{smallarc} and
Theorem~\ref{outchannel}.

\begin{thm}
Suppose that $X$ is a minimal counterexample to the Plane  Fixed
Point Problem. Then there exists $\delta>0$ such that the continuum
$Y=T(X)_\delta^+$ is a non-separating continuum, $f$ is fixed point
free on $Y$ and all accessible points of $Y$  are contained in one
arc component $K$ of the boundary of $Y$.  In other words, $Y$ is
homeomorphic to a disk with exactly one  channel removed which
corresponds to the unique geometric outchannel of variation $-1$ of
$X$. This channel compactifies on $X$.  The sides of this channel
are halflines  consisting entirely of chords of zero variation and
points in $X$. There exist arbitrarily small homeomorphisms of tails
of these halflines to a tail of $R_t$ which is the external ray
corresponding to this channel.
\end{thm}

\chapter{Fixed points}\label{ch:fxpt}

In this chapter we study fixed points in invariant and non-invariant
continua under positively oriented maps. We also obtain corollaries
dealing with complex polynomials (the applications of these
corollaries to complex dynamics are described in
Chapter~\ref{intro}.)

\section{Fixed points in invariant continua}\label{sec:fxpt}

In this section we will consider a positively oriented map of the
plane. As we shall see below, a  straight forward application of the
tools developed above will give us the desired fixed point result.
We will often assume, by way of contradiction, that $f:\Complex \to
\Complex$ is a positively oriented map, $X$ is a  plane continuum
such that $f(X)\subset T(X)$ and $T(X)$ contains no fixed points of
$f$.

\begin{lem}\label{closed}
Let $f:\Complex\to\Complex$ be a map and $X\subset\complex$ a continuum such
that $f(X)\subset T(X)$. Suppose $C=(a,b)$ is a crosscut of the continuum
$T(X)$. Let  $v\in (a,b)$ be a point and $J_v$ be a junction such that $J_v\cap
(X\cup C)=\{v\}$. Then there exists an arc $I$ such that $S=I\cup C$ is a
simple closed curve, $T(X)\subset T(S)$ and $f(I)\cap J_v= \0$.
\end{lem}

\begin{proof}
Since $f(X)\subset T(X)$ and $J_v\cap X=0$, it is clear that there exists an
arc $I$ with endpoints $a$ and $b$ sufficiently close to  $T(X)$ such that
$I\cup C$ is a simple closed curve, $T(X)\subset T(I\cup C)$ and $f(I)\cap
J_v=\0$. This completes the proof.
\end{proof}

\begin{cor}\label{posvaro}
Suppose $X\subset\complex$ is a continuum, $f:\Complex\to \Complex$ a
positively oriented map such that $f(X)\subset T(X)$. Then for each crosscut
$C$  of $T(X)$ such that $f(\cl C)\cap \cl C=\0$, $\var(f,C)\geq 0$
\end{cor}

\begin{proof}
Suppose that $C=(a,b)$ is a crosscut of $T(X)$ such that $f(\cl C)\cap \cl
C=\0$ and $\var(f,C)\not=0$. Choose a junction $J_v$ such that $J_v\cap (X\cup
C)=\set{v}$ and $v\in C\setminus X$. By Lemma~\ref{closed}, there exists an arc
$I$ such that $S=I\cup C$ is a simple closed curve and $f(I)\cap J_v=\0$.
Moreover, by choosing $I$ sufficiently close to $X$, we may assume that
$v\in\Complex\setminus f(S)$. Hence $\var(f,C)=\text{Win}(f,S,v)\geq 0$ by the
remark following Definition~\ref{vararc}.
\end{proof}

\begin{thm}\label{fixpoint}\index{fixed point!for positively oriented maps}
Suppose $f:\Complex\to\Complex$ is a positively oriented map and $X$ is a
continuum such that $f(X)\subset T(X)$. Then there exists a point $x_0\in T(X)$
such that $f(x_0)=x_0$.
\end{thm}

\begin{proof}
Suppose  we are given a continuum $X$ and $f:\Complex\to \Complex$  a
positively oriented map such that $f(X)\subset T(X)$. Assume that $f|_{T(X)}$
is fixed point free. Choose a simple closed curve $S$ such that $X\subset T(S)$
and points $a_0<a_1<\ldots<a_n$ in $S\cap X$ such that for each $i$
$C_i=(a_i,a_{i+1})$ is a sufficiently small crosscut of $X$, $f(\ol{C_i})\cap
\ol{C_i}=\0$ and  $f|_{{T(S)}}$ is fixed point free. By Corollary~\ref{posvaro},
$\var(f,C_i)\geq 0$ for each $i$. Hence by Theorem~\ref{I=V+1}, $\ind(f,S)=\sum
\var(f,C_i) +1 \geq 1$. This contradiction with Theorem~\ref{fpthm} completes
the proof.
\end{proof}

\begin{cor} \index{point of period two!for oriented maps}
Suppose $f:\Complex\to\Complex$  is a perfect, oriented map and $X$ is a
continuum such that $f(X)\subset T(X)$. Then there exists a point $x_0\in
T(X)$ of period at most  2.
\end{cor}

\begin{proof}
By Theorem~\ref{orient}, $f$ is either positively or negatively
oriented. In either case, the second iterate $f^2$ is positively
oriented and must have a fixed point in $T(X)$ by
Theorem~\ref{fixpoint}.
\end{proof}

\section{Dendrites}\label{sec:dendr}

Here we generalize Theorem~\ref{th:dendr} on the existence of fixed points in
invariant dendrites to non-invariant dendrites. We also show that in certain
cases the dendrite must contain infinitely many periodic cutpoints. Given two
points $a, b$ of a dendrite we denote by $[a, b], (a, b], [a, b), (a, b)$ the
unique closed, semi-open and open arcs connecting $a$ and $b$ in the dendrite.
Unless specified otherwise, the situation considered in this subsection is as
follows: $D_1\subset D_2$ are dendrites and $f:D_1\to D_2$ is a continuous map.
Set $E=\ol{D_2\sm D_1}\cap D_1$. In other words, $E$ consists of points at
which $D_2$ ``grows'' out of $D_1$. Observe that more than one component of
$D_2\sm D_1$ may ``grow'' out of a point $e\in E$. We assume that $D_1$ is
non-degenerate.

As an important tool we will need the following retraction closely related to
the described above situation.

\begin{defn}\label{retr}
For each $x\in D_2$ there exists a unique arc (possibly a point) $[x,d_x]$ such
that $[x,d_x]\cap D_1=\{d_x\}$. Hence there exists a natural monotone
retraction $r:D_2\to D_1$ defined by $r(x)=d_x$, and the map $g=g_f=r\circ
f:D_1\to D_1$  is a continuous map of $D_1$ into itself. We call the map $r$
the \emph{natural retraction (of $D_2$ onto $D_1$)}\index{natural retraction of
dendrites} and the map $g$ the \emph{retracted (version of) $f$}.
\end{defn}

The map $g$ is designed to make $D_1$ invariant so that
Theorem~\ref{th:dendr} applies to $g$ and allows us to conclude that
there are $g$-fixed points. Theorem~\ref{fxpt-dend-1} extends the
result for $\R$ claiming that if there are points $a<b$ in $\R$
mapped by $f$ in different directions, then there exists a fixed
point $c\in (a, b)$ (see Introduction, Subsection 1.1). Let us recall
the notion of the boundary scrambling property which is first
introduced in Definition~\ref{bouscr}.

\setcounter{chapter}{5}\setcounter{section}{3}\setcounter{thm}{0}

\begin{defn}[Boundary scrambling for dendrites]
Suppose that $f$ maps a dendrite $D_1$ to a dendrite $D_2\supset D_1$. Put
$E=\ol{D_2\sm D_1}\cap D_1$ (observe that $E$ may be infinite). If for each
\emph{non-fixed} point $e\in E$, $f(e)$ is contained in a component of
$D_2\sm\{e\}$ which intersects $D_1$, then we say that $f$ has the
\emph{boundary scrambling property} or that it \emph{scrambles the boundary}.
Observe that if $D_1$ \emph{is} invariant then $f$ automatically scrambles the
boundary.
\end{defn}

We are ready to prove the following theorem.

\setcounter{chapter}{7}\setcounter{section}{2}\setcounter{thm}{1}

\begin{thm}\label{fxpt-dend-1}
The following claims hold.

\begin{enumerate}

\item
If $a, b\in D_1$ are such that $a$ separates $f(a)$ from $b$ and $b$
separates $f(b)$ from $a$, then there exists a fixed point $c\in (a,
b)$. Thus, if $e_1\ne e_2\in E$ are such that $f(e_i)$ belongs to a
component of $D_2\sm \{e_i\}$ disjoint from $D_1$ then there is a
fixed point $c\in (e_1, e_2)$.

\item
If $f$ scrambles the boundary, then $f$ has a fixed point in $D_1$.

\end{enumerate}

\end{thm}

Observe, that the fixed points found in (1) are cutpoints of $D_1$
(and hence of $D_2$).

\begin{proof}
(1) Set $a_0=a$. Then we find a sequence of points $a_{-1}, a_{-2},
\dots$ in $(a, b)$ such that $f(a_{-n-1})=a_{-n}$ and $a_{-n-1}$
separates $a_{-n}$ from $b$. Clearly, $\lim_{n\to \iy} a_{-n}=c\in
(a, b)$ is a fixed point as desired (by the assumptions $c$ cannot be
equal to $b$). If there are two points $e_1\ne e_2\in E$ such that
$f(e_i)$ belongs to a component of $D_2\sm \{e_i\}$ disjoint from
$D_1$ then the above applies to them.

(2) Assume that there are no $f$-fixed points $e\in E$. By Theorem~\ref{th:dendr}
$g_f=g$ has a fixed point $p\in D_1$. It follows from the fact that $f$
scrambles the boundary that points of $E$ are not $g$-fixed. Hence $p\notin E$.

In general, a $g$-fixed point is not necessarily an $f$-fixed point. In fact,
it follows from the construction that if $f(x)\ne g(x)$, then $f$ maps $x$ to a
point belonging to a component of $D_2\sm D_1$ which ``grows'' out of $D_1$ at
$r\circ f(x)=g(x)\in E$. Thus, since $g(p)=p$ but $p\nin E$, then
$g(p)=f(p)=p$.
\end{proof}

\begin{rem} It follows from Theorem~\ref{fxpt-dend-1} that the only behavior
of points in $E$ which does not force the existence of a fixed point in $D_1$
is when one point $e\in E$ maps into a component of $D_2\sm \{e\}$ disjoint
from $D_1$ whereas any other point $e'\in E$ maps into the component of $D_2\sm
\{e'\}$ which is not disjoint from $D_1$.
\end{rem}


Now we suggest conditions under which a map of a dendrite has
infinitely many periodic cutpoints; the result will then apply in
cases related to complex dynamics. Let us recall the notion of a
weakly repelling periodic point which is first introduced in
Definition~\ref{wkrep}.

\setcounter{chapter}{5}\setcounter{section}{3}\setcounter{thm}{1}

\begin{defn}[Weakly repelling periodic points]
In the situation of Definition~\ref{bouscr}, let $a\in D_1$ be a
fixed point and suppose that there exists a component $B$ of $D_1\sm
\{a\}$ such that arbitrarily close to $a$ in $B$ there exist fixed
cutpoints of $D_1$ or points $x$ separating $a$ from $f(x)$. Then we
say that $a$ is a \emph{weakly repelling fixed point (of $f$ in
$B$)}. A periodic point $a\in D_1$ is said to be simply \emph{weakly
repelling} \index{weakly repelling}if there exists $n$ and a
component $B$ of $D_1\sm \{a\}$ such that $a$ is a weakly repelling
fixed point of $f^n$ in $B$.
\end{defn}

\setcounter{chapter}{7}\setcounter{section}{2}\setcounter{thm}{3}

Now we can prove Lemma~\ref{wkrppow}.

\begin{lem}\label{wkrppow}
Let $a$ be a fixed point of $f$ and $B$ be a component of $D_1\sm \{a\}$. Then
the following two claims are equivalent:

\begin{enumerate}

\item $a$ is a weakly repelling fixed point for $f$ in $B$;

\item either there exists a sequence of fixed cutpoints of $f|_B$,
converging to $a$, or, otherwise, there exists a point $y\in B$ which separates
$a$ from $f(y)$ such that there are no fixed cutpoints in the component of
$B\sm \{y\}$ containing $a$ in its closure (in the latter case for any $z\in
(a, y]$ the point $z$ separates $f(z)$ from $a$ and each backward orbit of $y$
in $(a, y]$ converges to $a$).

\end{enumerate}

In particular, if $a$ is a weakly repelling fixed point for $f$ in $B$ then $a$
is a weakly repelling fixed point for $f^n$ in $B$ for any $n\ge 1$.
\end{lem}

\begin{proof}
Let us show that (2) implies (1). We may assume that there exists a point $y\in
B$ which separates $a$ from $f(y)$ such that there are no fixed cutpoints in
the component $W$ of $B\sm \{y\}$ containing $a$ in its closure. Choose a point
$z\in (a, y)$. Since there are no fixed cutpoints of $f$ in $W$,
Theorem~\ref{fxpt-dend-1}(1) implies that $f(z)$ cannot be separated from $y$
by $z$. Hence $f(z)$ is separated from $a$ by $z$, and $a$ is weakly repelling
for $f$ in $B$. Moreover, we can take preimages of $y$ in $(a, y]$, then take
their preimages even closer to $a$, inductively. Any so constructed backward orbit of
$y$ in $(a, y]$ converges to $a$ because it converges to a fixed point of $f$
in $[a, y]$ and $a$ is the only such fixed point.

Now, suppose that (1) holds. We may assume that there exists a neighborhood $U$
of $a$ in $B$ such that there are no fixed cutpoints of $f$ in $U$. If $a$ is
weakly repelling in $B$ for $f$, we can choose a point $y\in U$ so that $y$
separates $a$ from $f(y)$ as desired.

It remains to prove the last claim of the lemma. Indeed, we may
assume that there is no sequence of $f^n$-fixed cutpoints in $B$
converging to $a$. Choose a neighborhood $U$ of $a$ which contains
no $f^n$-fixed cutpoints in $U\cap B$. By (2) we can choose a point
$y\in U\cap B$ such that $y$ separates $a$ from $f(y)$ so that there
is a sequence of preimages of $y$ under  $f$ which converges to $a$
monotonically. Choosing the $n$-th preimage $z$ we will see that $z$
separates $a$ from $f^n(z)$ with other parts of the second set of
conditions of (2) also fulfilled. By the  above  $a$ is weakly
repelling for $f^n$ in $B$ as desired.
\end{proof}

Let $B$ be a component of $D_1\sm \{a\}$ where $a$ is fixed. Suppose that $a$
is a weakly repelling fixed point for $f$ in $B$ which is not a limit of fixed
cutpoints of $f$ in $B$. Since the set of all vertices of $D_2$ together with
their images under $f$ and powers of $f$ is countable (see Theorem 10.23
\cite{nadl92}), we can choose $y$ from Lemma~\ref{wkrppow} so that $y$ and all
cutpoints $x$ in its backward orbit have $\val_{D_2}(x)=2$. From now on to
each fixed point $a$ which is weakly repelling for $f$ in a component $B$ of
$D_1\sm \{a\}$, but is not a limit point of fixed cutpoints in $B$, we
associate a point $x_a\in B$ of valence $2$ in $D_2$ separating $a$ from
$f(x_a)$ and such that all cutpoints in the backward orbit of $x_a$ are of
valence $2$ in $D_2$. We also associate to $a$ a semi-neighborhood $U_a$ of $a$
in $\ol{B}$ which is the component of $\ol{B}\sm \{x_a\}$ containing $a$. We
choose $x_a$ so close to $a$ that the diameter of $U_a$ is less than one third
of the diameter of $B$.

The next lemma shows that in some cases a fixed point $p$ from
Theorem~\ref{fxpt-dend-1}(2) can be chosen to be a cutpoint of $D_1$. Recall
that an endpoint of a continuum $X$ is a point $a$ such that the number
$\val_X(a)$ of components of $X\sm \{a\}$ equals $1$.

\begin{lem}\label{fxctpt}
Suppose that $f$ scrambles the boundary.  Then either there is a
fixed point of $f$ which is a cutpoint of $D_1$, or, otherwise,
there exists a fixed endpoint $a$ of $D_1$ such that if  $C_a$ is
the component of $D_2\sm \{a\}$ containing $D_1\sm \{a\}$, then $a$
is not weakly repelling for $f$ in $C_a$.
\end{lem}

\begin{proof}
Suppose that $f$ has no fixed cutpoints. By Theorem~\ref{fxpt-dend-1}(2), the
set of fixed points of $f$ is not empty. Hence we may assume that \emph{all}
fixed points of $f$ are endpoints of $D_1$ and, by way of contradiction,  $f$
is weakly repelling at any such fixed point $a$ in the component of $D_2\sm
\{a\}$ containing $D_1\sm\{a\}$. Suppose $a$ and $b$ are distinct fixed points of $f$. Let us
show that either $U_a\subset U_b$, or $U_b\subset U_a$, or $U_a\cap U_b=\0$.
Set $\dia(D_1)=\e$.

Recall that $x_a, x_b$ are cutpoints of $D_2$ of valence $2$. Now,
first we assume that $b\in U_a$. If $x_b\nin U_a$, then $U_a\subset
U_b$ as desired. Suppose that $x_b\in U_a$. We will show that
$x_b\in (b, x_a]$. Indeed, otherwise $U_b$ would contain the
component $Q$ of $D_1\sm \{x_a\}$, not containing $a$. However, by
the choice of the size of $U_a$ we see that $\dia(Q)\ge 2\e/3$ and
therefore $\dia(U_b)\ge 2\e/3$, a contradiction with the choice of
the size of $U_b$. Hence $x_b\in (b, x_a]$ which implies that
$U_b\subset U_a$. Now assume that $b\nin U_a$ and $a\nin U_b$. Then
it follows that $x_a, x_b\in [a, b]$ and that the order of points in
$[b, a]$ is $b, x_b, x_a, a$ which implies that $U_b\cap U_a=\0$.

Consider an open covering of the set of all fixed points $a\in D_1$ by their
neighborhoods $U_a$ and choose a finite subcover. By the above we may assume
that it consists of pairwise disjoint sets $U_{a_1}, \dots, U_{a_k}$.
Consider the component $Q$ of $D_1\sm\{x_{a_1},\dots,x_{a_k}\}$
whose endpoints are the points $x_{a_1},\dots,x_{a_k}$ and perhaps
some endpoints of $D_1$. Then $f|_Q$ is fixed point free. On the
other hand, $f|_{D_1}$ scrambles the boundary, and hence it is easy
to see that $f|_Q$, with $Q$ considered as a subdendrite of $D_2$,
scrambles the boundary too, a contradiction with
Theorem~\ref{fxpt-dend-1}.
\end{proof}

Lemma~\ref{fxctpt} is helpful in the next theorem.

\begin{thm}\label{infprpt}Let $D$ be a dendrite.
Suppose that $f:D\to D$ is continuous and all its periodic points are weakly repelling.
Then $f$ has infinitely many periodic cutpoints.
\end{thm}

\begin{proof}
By way of contradiction we assume that there are finitely many
periodic cutpoints of $f$. Let us show that each endpoint $b$ of $D$
with $f(b)=b$ is a weakly repelling fixed point for $f$. Since the
only component of $D\sm \{b\}$ is $D\sm \{b\}$, we will not be
mentioning this component anymore. By the assumptions of the Theorem
$b$ is weakly repelling for some power $f^m$ with $m\ge 1$. Then by
Lemma~\ref{wkrppow} and by the assumption we can choose a point
$x\ne b$ such that (1) $x$ is not a vertex or endpoint of $D$, (2)
for each point $z\in (b, x]$ we have that $z$ separates $b$ from
$f^m(z)$, and (3) the component $U$ of $D\sm \{x\}$ containing $b$,
contains no periodic cutpoints.

On the other hand, by way of contradiction we assume that $b$ is not
weakly repelling for $f$. Then, again by Lemma~\ref{wkrppow}, no
point $z\in (b, x]$ is such that $z$ separates $b$ from $f(z)$. The
idea is to use this in order to find a point $y\in (b, x]$ which
\emph{does not} separate $b$ from $f^m(y)$, a contradiction. To find
$y$ we apply the following construction. First, observe that there
exists a point $d_1\in (b, x)$ such that $f([b, x])\supset [b,
f(x)]\supset [b, d_1]$. Let $X_1=\{z\in [b, x] \mid f(z)\in [b,x]\}$
be the set of points mapped into $[b, x]$ by $f$. Then
$f(X_1)\supset [b, d_1]$ and all points of $X_1$ map towards $b$ on
$[b, x]$.
We can apply the same observation to $(b, d_1]$ instead of $(b, x]$.
In this way we obtain a point $d_2\in [b, d_1)$ and a set
$X_2=\{z\in [b,d_1] \mid f(z)\in X_1\}$ such that $[b, d_2]\subset
f^2(X_2)\subset [b, d_1]$ and all points of $f(X_2)$ are mapped
towards $b$ by $f$. Repeating this argument, we will find points of
$(b, x]$ mapped towards $b$ and staying on $(b, x]$ for $m$ steps in
a row. This contradicts the previous paragraph and proves that if
$b$ is weakly repelling for $f^m$, then it is weakly repelling for
$f$. Now by Lemma~\ref{wkrppow} $b$ is weakly repelling for $f^n$
for all $n\ge 1$.

Let $f$ have finitely many periodic cutpoints $a^1, \dots, a^k$ of
$f$. For each $a^i$ there exists $N_i$ such that $a^i$ is fixed for
$f^{N_i}$ and there exists a component $B^i$ of $D\sm \{a^i\}$ such
that $a^i$ is weakly repelling for $f^{N_i}$ in $B^i$. Set
$N=N_1\cdot \dots \cdot N_k$. Then it follows from
Lemma~\ref{wkrppow} that each $a^i$ is fixed for $f^N$ and weakly
repelling for $f^N$ in $B^i$. Observe that, as we showed above, the
endpoints of $D$ which are fixed under $f^N$ are in fact weakly
repelling for $f^N$. Without loss of generality we may use $f$ for
$f^N$ in the rest of the proof.

Let $A=\cup^k_{i=1} a^i$ and let $B$ be a component of $D\sm A$. Then $\ol{B}$
is a subdendrite of $D$ to which the above tools apply: $D$ plays the role of
$D_2$, $\ol{B}$ plays the role of $D_1$, and $E$ is exactly the boundary
$\bd B$ of $B$ (by the construction $\bd B\subset A$). Suppose that each
point $a\in \bd B$ is weakly repelling \emph{in $B$}. Then all fixed points of
$f$ in $B$ are endpoints of $B$, and all of them are weakly repelling for $f$.
Thus, by Lemma~\ref{fxctpt} there exists a fixed cutpoint in $B$, a
contradiction. Hence for some $a\in \bd B$ we have that $a$ is \emph{not}
weakly repelling in $\ol{B}\sm\{a\}$. By the assumption there exists a component, say,
$C'$, of $D\sm \{a\}$ disjoint from $B$ such that $a$ \emph{is} weakly repelling in $C'$. Let $C$
be the component of $D\sm A$ non-disjoint from $C'$ with $a\in \bd C$.

We can now apply the same argument to $C$. If all boundary points of $C$ are
weakly repelling for $f$ \emph{in $C$}, then by Lemma~\ref{fxctpt} $C$ will
contain a fixed cutpoint, a contradiction. Hence there exists a point $d\in A$
such that $d$ is \emph{not} weakly repelling for $f$ in $C$ and a component $F$
of $D\sm A$ whose closure meets $\ol{C}$ at $d$, and $d$ \emph{is} weakly
repelling in $F$. Note that $\ol{B}\cap \ol{F}=\0$. Clearly, after finitely many steps this process will have to
end (recall, that $D$  is a dendrite), ultimately leading to a component $Z$ of
$D\sm A$ such that all fixed points of $f$ in $\ol{Z}$ are endpoints of $\ol{Z}$ at which
$f$ is weakly repelling. Again, Lemma~\ref{fxctpt} applies to $\ol{Z}$ and there
exists a fixed cutpoint in $\ol{Z}$, a contradiction.
\end{proof}

An important application of Theorem~\ref{infprpt} is to dendritic
\emph{topological Julia sets}. They can be defined as follows.
Consider an equivalence relation $\sim$ on the unit circle
$\ucirc\subset\C$. Equivalence classes of $\sim$ will be called
\emph{($\sim$-)classes} and will be denoted by boldface letters. A
$\sim$-class consisting of two points is called a \emph{leaf}
\index{leaf}; a class consisting of at least three points is called
a \emph{gap} \index{gap}.
Fix an integer $d>1$ and define the map $\si_d:\uc\to\uc$ by
$\si_d(z)=z^d$, where $z$ is a complex number with $|z|=1$. Then the
equivalence $\sim$ is said to be a \emph{($d$-)invariant
lamination}\index{lamination!invariant} (this is more restrictive
than Thurston's definition in \cite{thur85}) if:

\noindent (E1) $\sim$ is \emph{closed}: the graph of $\sim$ is a closed subset of
$\ucirc \times \ucirc$;

\noindent (E2) $\sim$ defines a \emph{lamination}\index{lamination}, i.e., it is \emph{unlinked}:\index{unlinked}
if $\g_1$ and $\g_2$ are distinct $\sim$-classes, then their convex hulls
$\ch(\g_1), \ch(\g_2)$ in the unit disk $\disk$ are disjoint,

\noindent (D1) $\sim$ is \emph{forward invariant}: for a $\sim$-class $\g$, the set
$\si_d(\g)$ is a $\sim$-class too

\noindent which implies that

\noindent (D2) $\sim$ is \emph{backward invariant}: for a $\sim$-class $\g$, its
preimage $\si_d^{-1}(\g)=\{x\in \ucirc: \si_d(x)\in \g\}$ is a union of
$\sim$-classes;

\noindent (D3) for any gap $\g$, the map $\si_d|_{\g}: \g\to \si_d(\g)$ is a
\emph{map with positive orientation}, i.e., for every connected
component $(s, t)$ of $\ucirc\setminus \g$ the arc $(\si_d(s), \si_d(t))$ is a
connected component of $\ucirc\setminus \si_d(\g)$.

The lamination in which all points of $\uc$ are equivalent is said to be
\emph{degenerate}.\index{lamination!degenerate} It is easy to see that if a forward invariant lamination
$\sim$ has a $\sim$-class with non-empty interior then $\sim$ is degenerate. Hence
equivalence classes of any non-degenerate forward invariant lamination are
totally disconnected.

Let $\sim$ define an invariant lamination. A $\sim$-class $\g$ is periodic if $\si^n(\g)=\g$ for some $n\ge 1$.
Let $p: \ucirc\to J_\sim=\ucirc/\sim$ be the quotient map of $\ucirc$ onto its
quotient space $J_\sim$.  We can extend the equivalence relation $\sim$
to an equivalence relation $\approx$ of the entire plane by defining $x\approx y$ if either $x$ and $y$ are contained in the
convex hull of one equivalence class of $\sim$, or $x=y$. Then the quotient map $m:\C\to\C/\approx$ is a monotone map
whose point inverses are convex continua or points. Note that $p(\ucirc)=\ucirc/\sim=m(\ol{\disk})=\ol{\disk}/\approx$.
Let $f_\sim:J_\sim \to J_\sim$ be the  map induced by
$\sigma_d$. We call $J_\sim$ a \emph{topological Julia set} \index{topological!Julia set} and the induced map
$f_\sim$ a \emph{topological polynomial}.\index{topological!polynomial}
Recall that a \emph{branched covering map $f:X\to Y$}\index{branched covering map} is a finite-to-one
and open map for which  there exists a finite set $F\subset Y$ such that $f|_{X\setminus f^{-1}(F)}$
is a covering map.  Note that $f_\sim$ is a branched
covering map, and in particular, $f_\sim$ has finitely many critical points
(i.e., points where $f$ is not locally one-to-one). It
is easy to see that if $\g$ is a $\sim$-class then $\val_{J_\sim}(p(\g))=|\g|$
where by $|A|$ we denote the cardinality of a set $A$.

\begin{thm}\label{lamwkrp}
Suppose that the topological Julia set $J_\sim$ is a dendrite and
$f_\sim:J_\sim\to J_\sim$ is a topological polynomial. Then all periodic points
of $f_\sim$ are weakly repelling and $f_\sim$ has infinitely many periodic
cutpoints.
\end{thm}

\begin{proof}
Suppose that $x$ is an $f_\sim$-fixed point and set $\g=p^{-1}(x)$. Then
$\si_d(\g)=\g$. Suppose first, that $x$ is an endpoint of $J_\sim$. Then $\g$ is
a singleton. Choose $y\ne x\in J_\sim$. Then the unique arc $[x, y]\subset
J_\sim$ contains points $y_k\to x$ of valence $2$ because there are no more
than countably many vertices of $J_\sim$ (see Theorem 10.23 in \cite{nadl92}).
It follows that $\sim$-classes $p^{-1}(y_k)$ are leaves separating $\g$ from
the rest of the circle and repelled from $\g$ under the action of $\si_d$
which is expanding. Hence
$f_\sim(y_i)$ is separated from $x$ by $y_i$ and so $x$ is weakly repelling.

Suppose that $x$ is not an endpoint. Choose a \emph{very small}
connected neighborhood $U$ of $x$. It is easy to see that each
component $A$ of $U\sm \{x\}$ corresponds to a single non-degenerate
chord $\ell_A$ in the boundary of the Euclidean convex hull,
$\ch(\g)=G$, of $\g$. Recall that $\uc=\reals\sm \mathbb{Z}$ and
that the endpoints $a_A$ and $b_A$ of $\ell_A$ are points in $\uc$.
Denote by $\si_d(\ell_A)$ the chord with endpoints $\si_d(a_A)$ and
$\si_d(b_A)$ and by $|\ell_A|=\min\{|a_A-b_A|, 1-|a_A-b_A|\}$, the
\emph{length of $\ell_A$}. Since $f_\sim$ is a branched covering
map, for each component $A$ of $U\sm \{x\}$ there exists a unique
component $B=h(A)$ of $U\sm \{x\}$ such that $f_\sim(A)\cap B\ne
\0$. This defines a map $h$ from the set $\A$ of all components of
$U\sm \{x\}$ to itself. It follows that for each chord
$\ell_A\subset\bd G$, $\si_d(\ell_A)$ is a non-degenerate chord in
$\bd G$.

Suppose that there exist $\ell_A\subset\bd G$ and $n>0$ such that
$\si_d^n(\ell_A)=\ell_A$. Then it follows that the endpoints of $\ell_A$ are
fixed under $\si_d^{2n}$. Connect $x$ to a point $y\in A$ with the arc $[x,
y]$, and choose, as in the first paragraph, a sequence of points $y_k\in [x,
y], y_k\to x$ of valence $2$. Then again by the expanding properties of
$\si_d^{2n}$ it follows that $f_\sim(y_i)$ is separated from $x$ by $y_i$ and so
$x$ is weakly repelling (for $f^{2n}_\sim$ in $A$).

It remains to show that there must exist a component $A$ of
$U\sm\{x\}$ with $\si_d^n(\ell_A)=\ell_A$ for some $n>0$. Clearly
$\bd G$ can contain at most finitely many chords $\ell_A$ such that
its length $\Le(\ell_A)\ge 1/(2(d+1))$. If $\Le(\ell)<1/(2(d+1))$,
then $\Le(\si_d(\ell))=d \cdot \Le(\ell)$  (i.e. $\si_d$ expands the
length of small leaves by the factor $d$).

Since the family of chords in the boundary of $G$ is forward
invariant and for each chord $\ell_A$ with $\Le(\ell_A)<1/(2(d+1))$,
$\Le(\si_d(\ell_A))=d \cdot \Le(\ell_A)$, such a periodic chord must
exist (since if this is not the case there must exist an infinite
number of distinct leaves in the boundary of $G$ of length bigger
than $1/(2(d+1))$, a contradiction.

Hence all periodic points of $f_\sim$ are weakly repelling and by
Theorem~\ref{infprpt} $f_\sim$ has infinitely many periodic cutpoints.
\end{proof}

\section[Non-invariant continua and positively oriented maps]
{Non-invariant continua and positively oriented maps of the plane}\label{sec:fxpt-noni}

In this subsection we will extend Theorem~\ref{fixpoint} and obtain
a general fixed point theorem which shows that if a non-separating
plane continuum, not necessarily invariant, maps in an appropriate
way, then it contains a fixed point. However we begin with
Lemma~\ref{posvar} which gives a sufficient condition for the
non-negativity of the variation of an arc.

\begin{lem}\label{posvar}
Let $f:\C\to\C$ be positively oriented, $X$ a continuum and
$C=[a,b]$ a bumping arc of $X$ such that $f(a), f(b)\in X$ and
$f(C)\cap C=\0$. Let $v\in C\sm X$, and let $J_v$ be a junction at
$v\in C$ defined as in Definition~\ref{vararc}. If there exists a
continuum $K$ disjoint from $J_v$ such that\, $C$ is a bumping arc of
$K$ and $f(K)\cap J_v\subset \{v\}$ (e.g., $K$ can be a subcontinuum
of $X$ with $a, b\in K$), then $\var(f,C)\geq 0$.
\end{lem}

\begin{proof}
Consider two cases. First, let $f(K)\cap J_v=\0$. Choose an arc $I$
very close to $K$ so that $S=I\cup C$ is a bumping simple closed curve around
$K$ and by continuity $f(I)\cap J_v=\0$. Then $v\nin f(S)$. Moreover,
since $f(I)$ is disjoint from $J_v$, 
it is easy to see that $\var(f,C)=\text{win}(f,S,v)\geq 0$ (recall that $f$ is
positively oriented). Suppose now that $f(K)\cap J_v=\{v\}$. Then we can
perturb the junction $J_v$ slightly in a small neighborhood of $v$, obtaining a
new junction $J_d$ such that intersections of $f(C)$ with $J_v$ and $J_d$ are
the same (and, hence, yield the same variation) and $f(K)\cap J_d=\0$. Now
proceed as in the first case.
\end{proof}

Let us recall the notion of (strongly) scrambling the boundary of a
planar continuum introduced in Definition~\ref{scracon} (it extends
the notion of scrambling the boundary from maps of dendrites to maps
of the plane).

\setcounter{chapter}{5}\setcounter{section}{4}\setcounter{thm}{0}

\begin{defn}
Suppose that $f:\C\to\C$ is a positively oriented map and $X\subset \C$ is a
non-separating continuum. Suppose that there exist $n\ge 0$ disjoint
non-separating continua $Z_i$ such that the following properties hold:

\begin{enumerate}

\item $f(X)\sm X\subset \cup_i Z_i$;

\item for all $i$,  $Z_i\cap X=K_i$ is a non-separating continuum;


\item for all $i$, $f(K_i)\cap [Z_i\sm K_i]=\0$.

\end{enumerate}

\noindent Then the map $f$ is said to \emph{scramble the boundary (of $X$}). If
instead of (3) we have

\begin{enumerate}

\item[(3a)] for all $i$, either $f(K_i)\subset K_i$, or
$f(K_i)\cap Z_i=\0$

\end{enumerate}

\noindent then we say that $f$  \emph{strongly scrambles the boundary
(of $X$)}; clearly, if $f$ strongly scrambles the boundary of $X$,
then it scrambles the boundary of $X$. In either case, the continua
$K_i$ are called \emph{exit continua (of $X$)}.
\end{defn}

\setcounter{chapter}{7}\setcounter{section}{3}\setcounter{thm}{1}

We will always use the same notation (for $X$, $Z_i$ and $K_i$)
introduced in Definition~\ref{scracon} unless explicitly stated
otherwise. Let us make a few remarks. First, even though we use the
notion only for  positively oriented maps $f$, the definitions can be
given for all continuous functions. Also, observe, that in the
situation of Definition~\ref{scracon} if $X$ is invariant then $f$
automatically strongly scrambles the boundary because the set of exit
continua can be taken to be empty. We will also agree that the choice
of the sets $Z_i$ is optimal in the sense that if $(f(X)\sm X)\cap
Z_i=\0$ for some $i$, then the set $Z_i$ will be removed from the
list. In particular, all continua $Z_i$ contain points from $f(X)\sm
X$ and hence all continua $K_i$ have points from $\ol{f(X)\sm X}\cap
X$.

By Remark~\ref{zxgrow} $X\cup (\bigcup Z_i)$ is a non-separating continuum.
Suppose that we are in the situation of the previous section, $D_1\subset D_2$
are dendrites, $E=\ol{D_2\sm D_1}\cap D_1=\{z_1, \dots, z_l\}$ is finite, and
$f:D_1\to D_2$ scrambles the boundary in the sense of the previous section. For
each $z_i$ consider the union of all components of $D_2\sm D_1$ whose closures
contain $z_i$, and denote by $Z_i$ the closure of their union. In other words,
$Z_i$ is the closed connected piece of $D_2$ which ``grows'' out of $D_1$ at
$z_i$. Then $Z_i\cap D_1=\{z_i\}$, each $Z_i$ is a dendrite itself, and $f$
strongly scrambles the boundary in the sense of the new definition too. This
explains why we use similar terminology in both cases.

From now on we fix a positively oriented map $f$. Even though some
of the main applications of the results are to polynomial maps, this
generality is well justified because in some arguments (e.g., when
dealing with parabolic points) we have to locally perturb our map to
make sure that the \emph{local index} at a parabolic fixed  point
equals $1$, and this leads to the loss of analytic properties of the
map (see Lemma~\ref{pararepel}). Let us now prove the following
technical lemma.

\begin{lem}\label{posvar+}
Suppose that $f$ is positively oriented, scrambles the boundary of
$X$, $Q$ is a bumping arc of $X$ such that its endpoints map back
into $X$ and $f(Q)\cap Q=\0$. Then $\var(f, Q)\ge 0$.
\end{lem}

\begin{proof}
Suppose first that $Q\sm \bigcup Z_i\ne \0$ and choose $v\in Q\sm
\bigcup Z_i$. Since $v\in Q\sm \bigcup Z_i$ and $X\cup (\bigcup
Z_i)$ is non-separating, there exists a junction $J_v$, with $v\in
Q$, such that $J_v\cap[X\cup Q\cup (\bigcup Z_i)]= \{v\}$ and,
hence, $J_v\cap f(X)\subset\{v\}$. Now the desired result follows
from Lemma~\ref{posvar}.

Suppose now that $Q\sm \bigcup Z_i=\0$. Then $Q\subset Z_i$ for some
$i$ and so $Q\cap X\subset K_i$. In particular, both endpoints of
$Q$ belong to $K_i$. Choose a point $v\in Q$. Then again there is a
junction connecting $v$ and infinity outside $X$ (except possibly
for $v$). Since all sets $Z_j, j\ne i$ are positively distant from
$v$ and $X\cup (\bigcup_{i\ne j} Z_i)$ is non-separating, the
junction $J_v$ can be chosen to avoid all sets $Z_j, j\ne i$. Now,
by part (3) of Definition~\ref{scracon}, $f(K_i)\cap J_v\subset
\{v\}$, hence by Lemma~\ref{posvar} $\var(f,Q)\ge 0$.
\end{proof}

Lemma~\ref{posvar+} is applied in Theorem~\ref{fixpt} in which we
show that a map which strongly scrambles the boundary has fixed
points. In fact, Lemma~\ref{posvar+} is a major technical tool in
our other results too. Indeed, suppose that a positively oriented
map $f$ scrambles the boundary of $X$. If we can construct a bumping
simple closed curve $S$ around $X$ which has a partition into
bumping arcs (\emph{links of $S$}) whose endpoints map into $X$ (or
at least into $T(S)$) and whose images are disjoint from themselves,
then Lemma~\ref{posvar+} would imply that the variation of $S$ is
non-negative. By Theorem~\ref{I=V+1} this would imply that the index
of $S$ is positive. Hence by Theorem~\ref{fpthm} there are fixed
points in $T(S)$. Choosing $S$ to be sufficiently tight around $X$
we see that there are fixed points in $X$. Thus, the construction of
a tight bumping simple closed curve $S$ with a partition satisfying
the above listed properties becomes a major task.

For the sake of convenience we now sketch the proof of Theorem~\ref{fixpt}
which allows us to emphasize the main ideas rather than details. The main steps
in constructing $S$ are as follows. First we assume by way of contradiction
that $f$ has no fixed points in $X$. By Theorem~\ref{fixpoint} then
$f(X)\not\subset X$ and $f(K_i)\not\subset K_i$ for any $i$. By the definition
of strong scrambling then $f(K_i)$ is ``far away'' from $Z_i$ for any $i$.
Choose a tight bumping simple closed curve $S$ around $X$ with very small
links. We need to construct a partition of $S$ into bumping arcs whose
endpoints map into $X$ (or at least into $T(S)$) and whose images are disjoint
from themselves. Since there are no fixed points in $X$, we may assume that all
links of $S$ move off themselves. However some of them may have endpoint(s)
mapping outside $X$ which prevents the corresponding partition from being the one
we are looking for. So, we enlarge these links by consecutive concatenating
them to each other until the images of the endpoints of these concatenations
are inside $X$ \emph{and these concatenations still map off themselves} (the
latter needs to be proven which is a big part of the proof of
Theorem~\ref{fixpt}).

The bumping simple closed curve $S$ then remains as before, but the partition
changes because we enlarge some links. Still, the construction shows that the
new partition is satisfactory, and since $S$ can be chosen arbitrarily tight,
this implies the existence of a fixed point in $X$ as explained before. Thus, a
new development is that we are able to construct a partition of $S$ which has
all the above listed necessary properties having possibly \emph{very long}
links.

To achieve the goal of replacing some links in $S$ by their concatenations we
consider the links with at least one endpoint mapped outside $X$ in detail
(indeed, Lemma~\ref{posvar+} already applies to all other links) and use the
fact that $f$ strongly scrambles the boundary. The idea is to consider
consecutive links of $S$ with endpoints mapped into $Z_i\sm X$. Their
concatenation is a connected piece of $S$ with endpoints (and a lot of other
points) belonging to $X$ and mapping into one $Z_i$. If we begin the
concatenation right before the images of links enter $Z_i\sm X$ and stop it
right after the images of the links exit $Z_i\sm X$ we will have one condition
of Lemma~\ref{posvar+} satisfied because the endpoints of the thus constructed
new ``big'' concatenation link $T$ of $S$ map into $X$.

We need to verify that $T$ moves off itself under $f$. This is easy to see for
the end-links of $T$: each end-link has the image ``crossing'' into $X$ from
$Z_i\sm X$, hence the images of end-links are close to $K_i$. However the set
$K_i$ is mapped ``far away'' from $Z_i$ by the definition of strong scrambling
and because none of the $K_j$'s is invariant by the assumption. This implies that
the end-links themselves must be far away from $K_i$ (and hence from $Z_i$). If
now we move from link to link inside $T$ we see that those links cannot
approach $Z_i$ too closely because if they do, they will have to ``be close to
$K_i$'', and their images will have to be close to the image of $K_i$ which is
far away from $Z_i$, a contradiction with the fact that all links in $T$ have
endpoints which map  into $Z_i\sm X$. In other words, the dynamics of $K_i$
prevents the new bigger links from getting even close to $Z_i$ under $f$ which
shows that $T$ moves off itself as desired (after all, the images of new bigger
links are close to the set $Z_i\sm X$).

Given a compact set $K$ denote by $B(K, \e)$ the open set of all points whose
distance to $K$ is less than $\e$. By $d(\cdot, \cdot)$ we denote the distance
between two points or sets.

\begin{thm} \label{fixpt}Suppose $f:\C\to\C$ is
 positively oriented, $X$ is a non-separating continuum and $f$ strongly scrambles
  the boundary of $X$. Then
$f$ has a fixed point in $X$.
\end{thm}

\begin{proof}
If $f(X)\subset X$ then the result follows from
Theorem~\ref{fixpoint}. Similarly, if there exists $i$ such that
$f(K_i)\subset K_i$, then $f$ has a fixed point in $K_i\subset X$
and we are also done. Hence we may assume that $f(X)\sm X\ne\0$,
there are $m>0$ sets $Z_i, i=1, \dots, m$, $(f(X)\sm X)\cap Z_i\ne
\0$ for any $i$, and $f(K_i)\cap Z_i=\0$ for all $i$ (making these
claims we rely upon the fact that $f$ strongly scrambles the
boundary). Suppose that $f|_X$ is fixed point free. Then there
exists $\e>0$ such that for all $x\in X$, $d(x,f(x))>\e$. We may
assume that $2\e<\min\{d(Z_i,Z_j)\mid i\ne j\}$. We now choose
constants $\eta', \eta, \da$ and a bumping simple closed curve $S$
(whose initial links are crosscuts) of $X$ so that the following
holds.

\begin{enumerate}

\item $0<\eta'<\eta<\da<\e/3$.

\item
For each $x\in X\cap B(K_i,3\da)$ we have $d(f(x), Z_i)>3\da$.

\item
For each $x\in X\sm B(K_i,3\da)$ we have $d(x,Z_i)>3\eta$.

\item \label{44} For each $i$ there is a point $x_i\in X$ with $f(x_i)=z_i\in Z_i$ and $d(z_i, X)>3\eta$.
Since by Theorem~\ref{orient} \ $\partial f(X)\subset f(\partial X)$
and $X$ is non-separating, we may assume that $x_i\in\partial X$.

\item $X\subset T(S)$ and $A=X\cap S=\{a_0<\dots<a_n<a_{n+1}=a_0\}$ with
    points of $A$ numbered in the positive circular order around $S$.

\item $f|_{T(S)}$ is fixed point free.

\item For any $Q_i=(a_i, a_{i+1})\subset S$,
    $\dia(\ol{Q_i})+\dia(f(\ol{Q_i}))<\eta$.

\item For any $x, y\in X$ with $d(x, y)<\eta'$ we have $d(f(x),
    f(y))<\eta$.

\item $A$ is an $\eta'$-net in $\bd X$ (i.e., the Hausdorff distance between
$A$ and $\partial X$ is less than $\eta'$).

\end{enumerate}

Observe that $\ol{Q_i}$ is contained in the closed ball centered at $a_i$ of
radius $\dia(\ol{Q_i})$ and $f(\ol{Q_i})$ is contained in the closed ball centered
at $f(a_i)$ of radius $\dia(f(\ol{Q_i}))$; hence by (7) and since $d(x,
f(x))>\e$ for all $x\in X$ we see that $\ol{Q_i}\cap f(\ol{Q_i})=\0$ for every
$i$ (we rely on the triangle inequality here too).

\begin{figure}
\begin{picture}(307,231)
\put(0,0){\scalebox{0.5}{\includegraphics{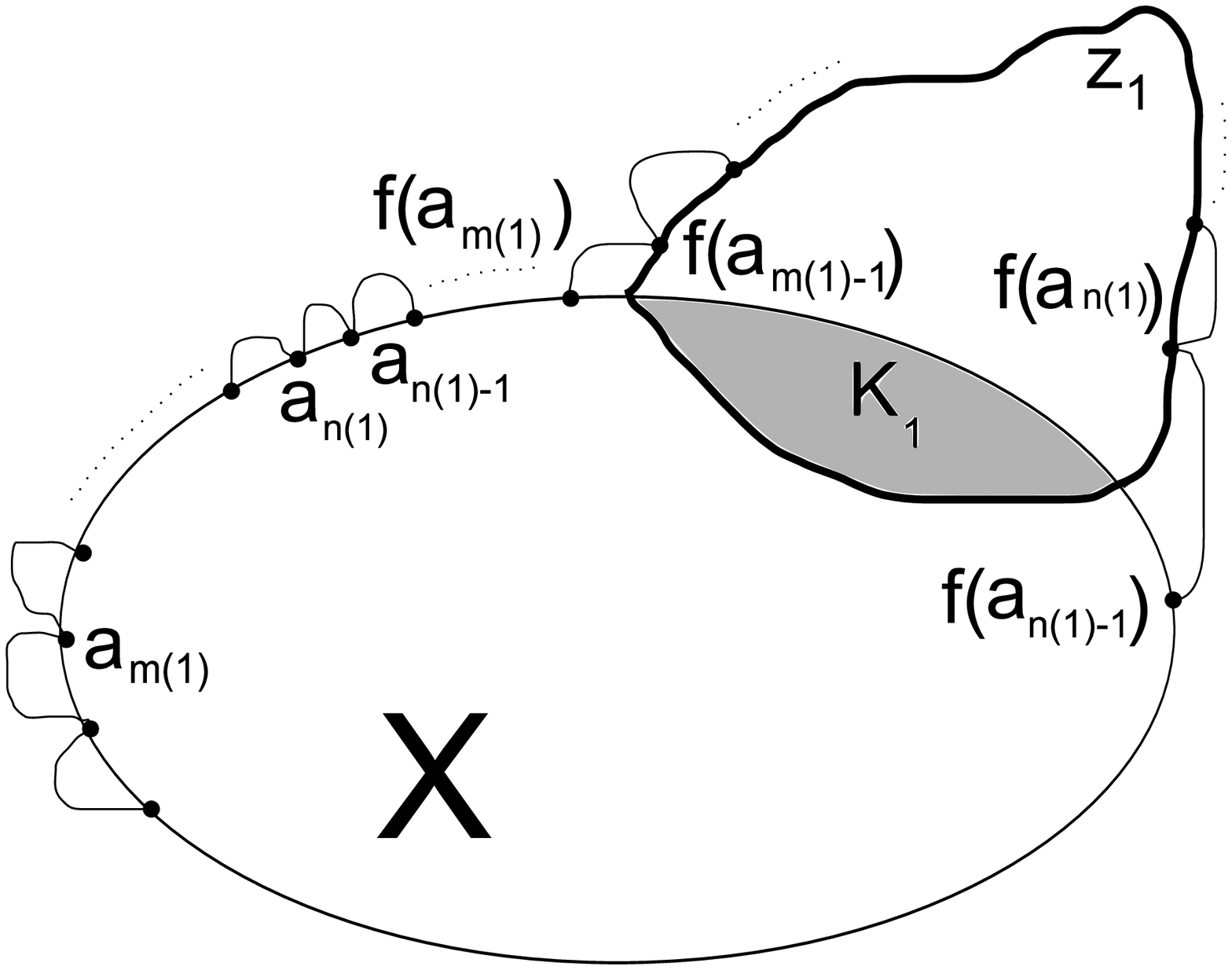}}}
\end{picture}
\caption{Replacing the links $[a_{n(1)-1}, a_{n(1)}],$ $\dots,$ $[a_{m(1)-1},
a_{m(1)}]$ by a single link $[a_{n(1)-1}, a_{m(1)}]$.}
\end{figure}

\noindent \textbf{Claim 1.} \emph{There exists a point $a_j\in A$ such that
$f(a_j)\in X\sm \bigcup_i \ol{B(Z_i, \eta)}$.}

\noindent \emph{Proof of Claim 1}. Set $\ol{B(Z_i, 3\eta)}=T_i$. We
will show that there exists a point $x\in \bd X$ with $f(x)\in X\sm
\bigcup T_i$. Indeed, suppose first that $m=1$. Then by (2) and the
assumption that $f(K_i)\cap Z_i=\0$ for each $i$ we have
$f(K_1)\subset X\sm T_1$, and we can choose any point of $K_1\cap
\bd X$ as $x$. Now, suppose that $m\ge 2$. Observe that by the
choice of $\e$ and by (1) the compacta $T_i$ are pairwise disjoint.
 By (\ref{44}) for each $i$ there
are points $x_i$ in $\bd X$ such that $f(x_i)\in Z_i\subset T_i$.
Since the sets $f^{-1}(T_i)\cap X$ are pairwise disjoint non-empty
compacta we see that the set $V=\bd X\sm \bigcup f^{-1}(T_i)$ is
non-empty (because $\bd X$ is a continuum). Now we can choose any
point of $V$ as $x$.

Notice now that by the choice of $A$ (see (9)) we can find a point $a_j$ such
that $d(a_j, x)<\eta'$ which by (8) implies that $d(f(a_j), f(x))<\eta$ and
hence $f(a_j)\in X\sm (\bigcup \ol{B(Z_i, \eta)})$ as desired.\qed

By Claim 1, we assume without loss of generality, that $f(a_0)\in X\sm
\bigcup \ol{B(Z_i, \eta)}$. Now, by (4) there exists a point $x_1$ such that
$f(x_1)=z_1$ is more than $3\eta$-distant from $X$. We can find $a_l\in A$ such
that $d(a_l, x_1)<\eta'$ and hence by (8) and by the triangle inequality $f(a_l)$
is at least $2\eta$-distant from $X$. On the other hand, $f(a_0)\in X$. By the
choice of $a_0$, we can find minimal $n(1)<m(1)$ from $\{0, 1, \dots, m+1\}$ such that
the following claims hold. Without loss of generality, $n(1)=1$.

\begin{enumerate}

\item $f(a_{n(1)-1})\in X$.

\item $f(a_r)\in f(X)\sm X$ for all $r$ with $n(1)\le r\le m(1)-1$
(and so, since $\dia(f(Q_u))<\e/3$ for any $u$ and
$d(Z_s,Z_t)>2\e$ for all $s\ne t$, there exists $i(1)\in\{1,\dots,n\}$ with
$f(a_r)\in Z_{i(1)}$ for all $n(1)\le r\le m(1)-1$).

\item $f(a_{m(1)})\in X$.

\end{enumerate}

Consider the arc $Q'=[a_{n(1)-1},a_{m(1)}]\subset S$ and show that
$f(Q')\cap Q'=\0$. As we walk along $Q'$ and mark the $f$-images of
points $a_{n(1)-1}, a_{n(1)}, \dots, a_{m(1)}$, we begin in $X$ at
$f(a_{n(1)-1})$, then enter $Z_{i(1)}\sm X$ and walk inside it, and then exit
$Z_{i(1)}\sm X$ at $f(a_{m(1)})\in X$. Since every step in this walk is rather
short (by (7) $\dia(Q_i)+\dia(f(Q_i))<\eta$), we see that $d(f(a_{n(1)-1}),
Z_{i(1)}\sm X)<\eta$ and $d(f(a_{m(1)}), Z_{i(1)}\sm X)<\eta$. On the other
hand for each $r, n(1)\le r\le m(1)-1$, we have $f(a_r)\in Z_{i(1)}\sm X$.
Thus, $d(f(a_r), Z_{i(1)})<\eta$ for each $n(1)-1\le r\le m(1)$. Since by (7)
for each link $Q$ of $S$ we have $\dia(Q)+\dia(f(Q))<\eta$, we now see by the
triangle inequality that $d(f(Q'), Z_{i(1)})<2\eta$.

This implies that for $n(1)-1\le r\le m(1)$, $d(a_r, K_{i(1)})>3\delta$
(because otherwise by (2) $f(a_r)$ would be farther away from $Z_{i(1)}$ than
$3\delta>\eta$, a contradiction) and so $d(a_r, Z_{i(1)})>3\eta$ (because
$a_r\in X\sm B(K_{i(1)}, 3\delta)$ and by (3)). Since by (7) for each link of
$S$ we have $\dia(Q)+\dia(f(Q))<\eta$, then $d(Q', Z_{i(1)})>2\eta$.

Therefore $f(Q')\cap Q'=\0$. This allows us to replace the original
division of $S$ into links $Q_0, \dots, Q_{m(1)-1}$ by a new one in which $Q'$ plays
the role of a new link; in other words, we simply delete the points
$\{a_{n(1)},\dots, a_{m(1)-1}\}$ from $A$. Thus, $Q'$ is a bumping arc whose
endpoints map back into the continuum $X$ and such that $f(Q')\cap Q'=\0$.
Therefore $Q'$ satisfies the conditions of Lemma~\ref{posvar+}, and so
$\var(f,Q')\ge 0$. Observe also that for $Q'$ the associated continuum $Z_{i(1)}$
is well-defined because the distance between distinct continua $Z_i$ is greater
than $2\e$. Replace the string of links $\{Q_0,\dots,Q_{m(1)-1}\}$ in $S$ by the single link
$Q'=Q'_0$ which has as endpoints $a_{n(1)-1}$ and $a_{m(1)}$. Continuing in the same manner and moving along $S$, in the end we
obtain a finite set $A'=\{a_0=a'_0<a'_1<\dots<a'_k\}\subset A$ such that for
each $i$ we have $f(a'_i)\in X\subset T(S)$ and for each arc
$Q'_i=[a'_i,a'_{i+1}]$ we have $f(Q'_i)\cap Q'_i=\0$. In other words, we will
construct a partition of $S$ satisfying all the required properties: its links
are bumping arcs whose endpoints map back into $X$ and whose images are disjoint
from themselves. As outlined after Lemma~\ref{posvar+}, this yields a
contradiction. More precisely, by Theorem~\ref{I=V+1}, $\ind(f, S)=\sum
\var(f,Q'_i)+1$, and since by Lemma~\ref{posvar+}, $\var(f,Q'_i)\ge 0$ for
all $i$, $\ind(f, S)\ge 1$ contradicting the fact that $f$ is fixed point free
in $T(S)$ (see Theorem~\ref{fpthm}).
\end{proof}

\section{Maps with isolated fixed points}\label{sec:fxpt-noniso}

In this section we assume that all maps $f:\complex\to\complex$ are
positively oriented maps with isolated fixed points.

\begin{defn}\label{crit}
Given a map $f:X\to Y$ we say that $c\in X$ is a \emph{critical point of $f$}
\index{critical point}
if for each neighborhood $U$ of $c$, there exist $x_1\ne x_2\in U$ such that
$f(x_1)=f(x_2)$. Hence, if $x$ is not a critical point of $f$, then $f$ is
locally one-to-one near $x$.
\end{defn}

If a point $x$ belongs to a non-degenerate continuum collapsed to a point under $f$ then $x$ is critical;
also any point which is an accumulation point of collapsing continua is
critical. However in these cases the map near $x$ may be monotone. A more
interesting case is when the map near $x$ is not monotone; then $x$ is a
\emph{branchpoint}\index{branchpoint of $f$}  of $f$ and it is critical even if there are no collapsing
continua close by. One can define the \emph{local degree} \index{local degree}$\deg_f(a)$
\index{degree@$\deg_f(a)$} as the
number of components of $f^{-1}(y)$ non-disjoint from a small neighborhood of
$a$ ($y$ then should be chosen close to $f(a)$). It is well-known that for a
positively oriented map $f$ and a point $a$ which is a component of
$f^{-1}(f(a))$ the local degree $\deg_f(a)$ equals the winding number $\win(f,
S, f(a))$ for any small simple closed curve $S$ around $a$. Then branchpoints are
exactly the points at which the local degree is greater than $1$. Notice that
since we do not assume any smoothness, a \emph{critical} point may well be both
fixed (periodic) and topologically repelling in the sense that some small neighborhoods of
$c=f(c)$ map over themselves by $f$.

Let us recall the notion of a local index of a map at a point which
is first introduced in Definition~\ref{indpt}.

\setcounter{chapter}{5}\setcounter{section}{4}\setcounter{thm}{2}

\begin{defn}
Suppose that $f:\C\to\C$ is a positively oriented map with isolated fixed points and
$x$ is a fixed point of $f$.
Then the \emph{local index of $f$ at $x$}, denoted by $\ind(f, x)$,
is defined as $\ind(f,S)$
where $S$ is a small simple closed curve around $x$.
\end{defn}

\setcounter{chapter}{7}\setcounter{section}{4}\setcounter{thm}{1}

It is easy to see that, since $f$ is positively oriented and has isolated fixed
points,  the local index is well-defined, i.e. does not depend on the
choice of $S$. By modifying a translation map one can give an example of a
homeomorphism of the plane which has exactly one fixed point $x$ with local
index $0$. Still in some cases the local index at a fixed point must be
positive.

\begin{defn}\label{toprepat}Let $f:\C\to\C$ be a map.
A fixed point $x$ is said to be \emph{topologically repelling} \index{topologically!repelling} if there exists
a sequence of simple closed curves $S_j\to \{x\}$ such that $x\in
\Int(T(S_j))\subset T(S_j)\subset \Int(T(f(S_j))$. A fixed point $x$ is said to
be \emph{topologically attracting}\index{topologically!attracting} if there exists a sequence of simple closed
curves $S_j\to \{x\}$ not containing $x$ and such that $x\in
\Int(T(f(S_j))\subset T(f(S_j))\subset \Int(T(S_j))$.
\end{defn}

\begin{lem}\label{ind1}Let $f:\C\to\C$ be a positively oriented map
with isolated fixed points. If $a$ is a topologically repelling
fixed point then we have that $\ind(f, a)=\deg_f(a)\ge 1$. If
however $a$ is a topologically attracting fixed point then $\ind(f,
a)=1$.
\end{lem}

\begin{proof}
Consider the case of the repelling fixed point $a$. Then it follows that, as
$x$ runs along a small simple closed curve $S$ with $a\in T(S)$, the vector
from $x$ to $f(x)$ produces the same winding number as the vector from $a$ to
$f(x)$. As we remarked before, it is well-known that this winding number equals
$\deg_f(a)$; on the other hand, $\ind(f, S)>0$ since $f$ is positively oriented
and has isolated fixed points. The argument for an attracting fixed point is
similar.
\end{proof}

If however a fixed point $x$ is neither topologically repelling nor
topologically attracting, then $\ind(f, x)$ could be greater than $1$ even in
the non-critical case. Indeed, by definition $\ind(f, x)$ coincides with
the winding number of $f(z)-z$ on a small simple closed curve $S$ around $x$ with
respect to the origin. If, e.g., $f$ is rational and $f'(x)\ne 1$ then this
implies that $\ind(f, x)=1$. However if $f'(x)=1$ then $\ind(f, x)$ is the
multiplicity at $x$ (i.e., the local degree of the map $f(z)-z$ at $x$). This
is related to the following useful theorem. It is a version of the more general
topological \emph{argument principle}.

\begin{thm}\label{argupr}
Suppose that $f$ is positively oriented and has isolated fixed points. Then for
any simple closed curve $S\subset\C$, which contains no fixed points of $f$,
its index equals the sum of local indices
taken over all fixed points in $T(S)$. 
In particular if for each fixed point
$p\in T(S)$ we have that $\ind(f, x)=1$ then $\ind(f, S)$ equals the number
$n(f, S)$ of fixed points inside $T(S)$.
\end{thm}

Theorem~\ref{argupr} implies Theorem~\ref{fpthm} for positively oriented maps
with isolated fixed points (indeed, if $\ind(f, S)$ $\ne$ $0$ then by
Theorem~\ref{argupr} there must exist fixed points in $T(S)$), and actually
provides more information. By the above analysis, Lemma~\ref{ind1} and
Theorem~\ref{argupr}, $\ind(f, S)$ equals the number $n(f, S)$ of fixed points
inside $T(S)$ if all $f$-fixed points in $T(S)$ are either topologically
attracting, or such that $f$ has a complex derivative $f'$ at $x$, and
$f'(x)\ne 1$; if $f$-fixed points can also be topologically repelling, then
$\ind(f, S)\ge n(f, S)$

In the spirit of the previous parts of the paper, we are still concerned with
finding $f$-fixed points inside non-invariant continua of which $f$ (strongly)
scrambles the boundary. However we now specify the types of fixed points we are
looking for. Thus, the main result of this subsection proves the existence of
specific fixed points in non-degenerate continua satisfying the appropriate
boundary conditions and shows that in some cases such continua must be
degenerate. It is in this form that we apply the result later  in this subsection.

Recall that an essential crossing of an external ray $R$ and a
crosscut $Q$ was defined in Definition~\ref{essential}; there an
external ray $R_t$ is said to \emph{cross} a crosscut $Q$
\emph{essentially} if and only if there exists $x\in R_t$ such that
the bounded component of $R_t\sm \{x\}$ is contained in the bounded
complementary domain of $T(X)\cup Q$. The fact that a crosscut
crosses a ray essentially can be similarly restated in the language
of the uniformization plane (i.e., if the ray and the crosscut are
replaced by their counterparts in the uniformization plane while
$X$ is replaced by the unit disk in the uniformization plane).

For the next definition we need to make an observation. Suppose
that $f:\C\to\C$ is a map and $D$ is a closed Jordan disk with
interior non-disjoint from a continuum  $X$ such that $f(D\cap
X)\subset X$ and $f(\ol{\bd D\sm X})\cap D=\0$. Suppose in addition
that $|\partial D\cap X|\ge 2$. Then the closure of any component
$Q$ of $\bd D\sm X$  is a bumping arc whose endpoints map back into
$X$ and such that $f(Q)\cap Q=\0$ (indeed, $f(\ol{\bd D\sm X})\cap
D=\0$ implies that $f(Q)\cap Q=\0$). Thus, $\var(f,Q)$ is
well-defined.

\begin{defn}\label{repout}Let $f$ be a positively oriented map and $X$ a continuum.
If $f(p)=p$ and $p\in \bd X$ then we say that $f$ \emph{repels
outside $X$ at $p$} \index{repels outside $X$ at $p$} provided there
exists a ray $R\subset \C\sm X$ from $\infty$ which lands on $p$ and
a sequence of simple closed curves $S^j$ bounding closed disks $D^j$
such that $D^1\supset D^2\supset\dots, $ $p\in\Int(D^j)$, $\cap
D^j=\{p\}$, $f(D^1\cap X)\subset X$, $f(\ol{S^j\sm X})\cap D^j=\0$
and for each $j$ there exists a component $Q^j$ of $S^j\sm X$ such
that $Q^j\cap R\ne\0$ and $\var(f,Q^j)\ne 0$.
\end{defn}



Definition~\ref{repout} gives some information about dynamics around $p$.

\begin{lem}\label{dynatp}
Suppose that $f:\C\to\C$ is a positively oriented map, $X$ a non-separating continuum,
$p\in\partial X$ such that $f(p)=p$ and  $f$ repels outside $X$ at $p$. If $R$ is
the ray from the Definition~\ref{repout} then $f(Q^j)\cap R\ne 0$. Moreover, if
$f$ scrambles the boundary of $X$, then $\var(f, Q^j)>0$.
\end{lem}

Thus, even though $R$ above may be non-invariant, there are crosscuts approaching $p$
which are mapped by $f$ ``along $R$ farther away from $p$''.

\begin{proof}
Take $z\in R\cap Q^j$ so that $(z, \iy)_R\cap Q^j=\0$. Choose a junction with
$[z, \iy]_R$ (the subray of $R$ running from $z$ to infinity)
playing the role of $R_i$ and two other rays close to $[z, \iy]_R$
on both sides. Then $f(Q^j)\cap R=\0$ implies $\var(f, Q^j)=0$, a
contradiction. The second claim follows by Lemma~\ref{posvar+}.
\end{proof}

We will use the following version of uniformization. Let $X$ be a
non-separating continuum and $\vp:\disk_\iy\to\C\sm X$  an onto
conformal map such that $\vp(\iy)=\iy$ (here
$\disk_\iy=\C\sm\ol{\disk}$ is the complement of the closed unit
disk). Thus, we choose the uniformization, under which the
\emph{complement} $\C\sm \ol{\disk}$ of the closed unit disk
corresponds to the \emph{complement} $\C\sm X$ of $X$. Of course,
the same can be considered on the two-dimensional sphere $\sphere$
which is sometimes more convenient. Notice, that since $\ol{\disk}$
is a non-separating continuum in $\C$, we can use for it the usual
terminology (crosscuts, shadows, etc). Also recall that the shadow
of a crosscut $C$ of a nonseparating continuum  $X$ is the bounded
component of $\complex\sm [X\cup C]$ (and \emph{not} its closure).

Then, given a crosscut $C$ of $X$ with endpoints $x, y$, we can associate to
its endpoints external angles as follows. It is well known \cite{miln00} that
$\vp^{-1}(C)$ is a crosscut of the closed unit disk with endpoints $\al,\be$.
It follows that  we can extend $\vp$ by defining $\vp(\al)=x$ and
$\vp(\be)=y$.  Note that this extension is not necessarily continuous. In this
case we say that $\al$ \emph{corresponds} to $x$ and $\be$ \emph{corresponds}
to $y$. There is a unique arc $I\subset \uc$ with endpoints $\al, \be,$
contained in the shadow of $\vp^{-1}(C)$. Assuming that the positive
orientation on $I$ is from $\al$ to $\be$, we choose the appropriate
orientation of $C$ (i.e., in this case from $x$ to $y$) and call such an
oriented $C$ \emph{positively oriented}. \index{positively oriented arc}

Observe that in this situation if $D$ is a disk around a point $x\in X$ then
components of $\bd D\sm X$ are crosscuts of $X$ whose $\vp$-preimages are
crosscuts of $\ol{\disk}$ in the uniformization plane. However, these
preimage-crosscuts in  $\disk_\iy$ may be located all over $\uc$.

The next theorem is the main result of this subsection.

\begin{thm}\label{locrot}
Suppose that $f$ is a  positively oriented map of the plane with only isolated fixed points,
$X\subset\C$ is a non-separating continuum or a point, and the following
conditions hold.

\begin{enumerate}


\item[(1h)]\label{fV}
For each fixed point $p\in X$ we have that $x\in \bd X$, $\ind(f, p)=1$ and $f$ repels
outside $X$ at $p$.

\item[(2h)]\label{Sc} The map $f$ scrambles the boundary of
    $X$. Moreover, using the notation from
    Definition~\ref{scracon} it can be said that for each $i$
    either $f(K_i)\cap Z_i=\0$, or there exists a neighborhood
    $U_i$ of $K_i$ with $f(U_i\cap X)\subset X$.

\end{enumerate}

Then $X$ is a point.
\end{thm}

\begin{proof}
Suppose that $X$ is not a point. Since $f$ has isolated fixed points, there
exists a simply connected neighborhood $V$ of $X$ such that all fixed points
$\{p_1,\dots,p_m\}$ of $f|_{\ol{V}}$ belong to $X$. The idea of the proof is to
construct a tight bumping simple closed curve $S$ such that $X\subset
T(S)\subset V$ and $\var(f, S)\ge m$. Hence $\ind(f, S)=\var(f, S)+1\ge m+1$
while by Theorem~\ref{argupr} our assumptions imply that $\ind(f, S)=m$, a
contradiction.

First we need to make a few choices of neighborhoods and constants; we assume
that there are $n$ exit continua $K_1, \dots, K_n$. We also assume that they
are numbered so that for $1\le i\le n_1$ we have $f(K_i)\cap Z_i=\0$ and for
$n_1<i\le n$ we have $f(K_i)\cap Z_i\ne \0$. Choose for each $i$ a small
neighborhood $U_i$ of $K_i$ as follows.

\begin{center} \emph{CHOOSING NEIGHBORHOODS $U_i$ OF EXIT CONTINUA $K_i$} \end{center}

\begin{enumerate}

\item By assumption (2) of the theorem we may assume that
$f(U_i\cap X)\subset X$ for each $i$ with $n_1<i\le n$.

\item By continuity we may assume that $d(U_i\cup Z_i,f(U_i))>0$ for $1\le i\le n_1$.

\item We may assume that $T(X\cup \bigcup \ol{U_i})\subset V$ and
$\ol{U_i}\cap\ol{U_k}=\0$ for all $i\ne k$.

\item We may assume that every fixed point of $f$
contained in $\ol{U_i}$ is contained in $K_i$.

\end{enumerate}

Let $\{p_1,\dots,p_t\}$ be all fixed points of $f$ in $X\sm
\bigcup_i K_i$ and let $\{p_{t+1},$ $\dots,$ $p_m\}$ be all the
fixed points contained in $\bigcup K_i$. Observe that then by part
(4) of the choice of neighborhoods $U_i$ we have $p_i\in X\sm
\ol{\bigcup U_s}$ if $1\le i\le t$. Also, it follows that for each
$j, t+1\le j\le m,$ there exists a unique $r_j, 1\le r_j\le n,$ such
that $p_j\in K_{r_j}$. For each fixed point $p_j\in X$ we rely upon
Definition~\ref{repout} and, as specified in that definition, choose
a ray $R_j\subset \C\sm X$ landing on $p_j$. Now we choose closed
disks $D_j$ around each $p_j$ from Definition~\ref{repout} so that,
in addition to properties from Definition~\ref{repout} (listed below
as (6), (9) and (10)), they satisfy the following  conditions.

\begin{figure}
\begin{picture}(307,231)
\put(100,0){\scalebox{0.3}{\includegraphics{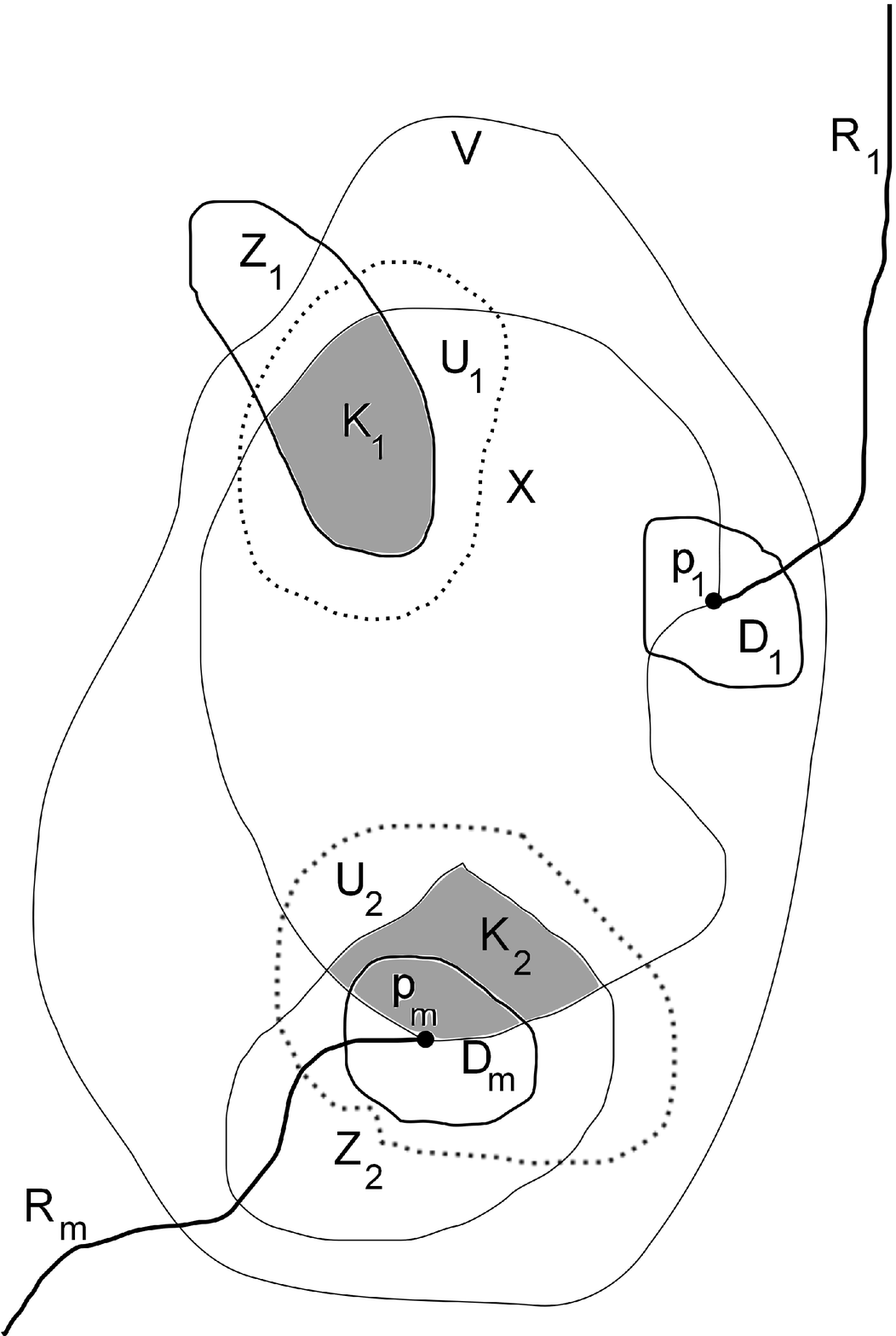}}}
\end{picture}
\caption{Illustration to the proof of Theorem~\ref{locrot}.}
\end{figure}

\begin{center} \emph{THE CHOICE OF CLOSED DISKS $D_j$} \end{center}

\begin{enumerate}\setcounter{enumi}{4}
\item $D_j\cap R_i=\0$ for all $i\ne j$.

\item $f(\ol{S_j\sm X})\cap D_j=\0$.

\item $T(X\cup \bigcup_j D_j)\subset V$.

\item $[D_j\cup f(D_j)]\cap [D_k\cup f(D_k)]=\0$ for all $j\ne k$.

\item  $f(D_j\cap X)\subset X$ (this is possible because $f$ repels outside
    $X$ at each fixed point of $f$ and by Definition~\ref{repout}).


\item Denote by $Q(j,s)$ all components of $S_j\sm X$; then there exists
  a component,  $Q(j,s(j))$,  of $S_j\sm X$, with $\var(f, Q(j,s(j)))>0$ and
    $Q(j,s(j))\cap R_j\ne\0$ (this is possible by Definition~\ref{repout}
    and Lemma~\ref{dynatp}).


\item $[D_j\cup f(D_j)]\cap \bigcup \ol{U_i}=\0$ for all $1\le j\le t$.

\item If $t<j\le m$ then $[D_j\cup f(D_j)]\subset U_{r_j}$.

\end{enumerate}

Let us use our standard uniformization $\vp:\disk_\iy\to\C\sm X$ described
before the statement of the theorem. It serves as an important tool; in
particular it allows us to pull crosscuts from the plane containing $X$ to
$\disk_\iy$ and introduce the appropriate orientation on all these crosscuts.

\smallskip

\noindent \textbf{Claim A.} \emph{Suppose that $W\subset D_j$ is a Jordan disk
around $p_j$ (e.g., $W$ may coincide with $D_j$) and $C$ is a crosscut which is
a component of $\bd W\sm X$. Then the shadow of $Q(i, s(i)), i\ne j$ is not
contained in the shadow of $C$ (thus, the shadows of $Q(j,s(j))$ and $Q(i,
s(i))$) are disjoint). Moreover, if $W$ is sufficiently small, then the shadow
of $Q(j, s(j))$ is not contained in the shadow of $C$ either.}

\smallskip

\noindent \emph{Proof of Claim A.} By condition (8) from the choice
of the disks $D_j$ those disks are pairwise disjoint. Hence all the
crosscuts $Q(r, t)$ are pairwise disjoint, and $C\cap Q(i,
s(i))=\0$. If the shadow of $Q(i, s(i))$ is contained in the shadow
of $C$, then the ray $R_i$ intersects $C$, contradicting condition
(5) from the choice of the disks $D_j$. Hence $\sh(Q(j,s(j))\cap
\sh(Q(i,s(i))=\0$ for $i\ne j$ (otherwise, because the crosscuts are
pairwise disjoint, one of the shadows would contain the other one,
impossible by the just proven). Now, if $\sh(Q(j, s(j))\subset
\sh(C)$, then, in $\disk_\iy$, $\vp^{-1}(C)$ \emph{shields}
$\vp(Q(j, s(j))$ from infinity and must be, together with $C$, of a
bounded away from zero size. Hence, if $W$ is very small, this
cannot happen. \qed

\smallskip

Now we define another collection of disks around the points $p_j$. By Claim A for
each $j$ we choose a small Jordan disk $D'_j$ from Definition~\ref{repout}
around $p_j$ so that no shadow $\sh(Q(i, s(i))$ is contained in the shadow of
any crosscut $C$ which is a component of $(\bd D'_j)\sm X$. In particular, for
each such $C$, $\ol{f(C)}\cap \ol{C}=\0$. Let us now choose a few constants.

\begin{center} \emph{THE CHOICE OF CONSTANTS $\eta<\da<\e$} \end{center}

\begin{enumerate}\setcounter{enumi}{12}

\item Choose $\e>0$ such that for all $x\in X\sm \bigcup D'_j$, $d(x,f(x))>3\e$
and for each crosscut $C$ of $X$ of diameter less than $\e$ with at least one
endpoint outside of $\bigcup D_j$ we have that $f(C)$ is
disjoint from $C$ (observe that outside any given neighborhood of $\{p_1,
\dots, p_m\}$ all points of $X$ move under $f$ by a bounded away from zero distance).

\item Choose $\da>0$ so that the following several inequalities hold:

\begin{enumerate}

\item $3\da<\e$,

\item $3\da<d(Z_i,Z_j)$ for all $i\ne j$,

\item \label{eU} $3\da<d(Z_i, [X\cup f(X)]\sm [Z_i\cup U_i])$ for each $i$,

\item $3\da<d(K_i,\C\sm U_i)$ for each $i$,

\item if $f(K_i)\cap Z_i=\0$, then $3\da<d(f(U_i),Z_i\cup U_i)$.

\end{enumerate}

\item By continuity choose $\eta>0$ such that for each set $H\subset V$ of
    diameter less than $\eta$ we have $\dia(H)+\dia(f(H))<\da$ and that
    $d(D'_i, D'_j)>\eta, i\ne j$.

\end{enumerate}

By (11) and (13) above, if a set $H\subset V$ is of diameter at most
$\eta$ and $H\not\subset \bigcup D'_i$, then $f(H)\cap H=\0$.
Indeed, otherwise let $x\in H\sm \bigcup D'_i$ and $y\in H$ be such
that $f(y)\in H$. Then by (13) $d(x, f(x))>3\e$ while  by (15) and
the triangle inequality $d(x, f(x))\le d(x,
f(y))+d(f(y),f(x))<\da<\e/3$, a contradiction.

\smallskip

Consider the family $\mathcal{E}_X$ of all components of the sets
$(\bd D'_i)\sm X$, and the crosscuts $Q(i, s(i))$. By condition (8)
from the choice of the closed disks $D_j$, the disks $D_j$ are
pairwise disjoint; hence, the crosscuts in $\mathcal{E}_X$ are
pairwise disjoint.

Let $T$ be the topological hull  $T=T(X\cup (\bigcup D'_j) \cup
\bigcup Q(i, s(i)))$. Then $T$ is a non-separating continuum. Call
$C\in \mathcal{E}_X$ an \emph{unshielded (crosscut of $X$)} if it is
a part of $\bd T$ and denote the family of all such crosscuts by
$\mathcal{E}^u_X$. By Claim A all crosscuts $Q(i, s(i))$ are
unshielded. Call $\vp$-preimages of unshielded crosscuts
\emph{unshielded (crosscuts of $\ol{\disk}$)} and denote their
family by $\mathcal{E}^u_\disk$. Clearly, any two unshielded
crosscuts have disjoint shadows.

For each $C\in\mathcal{E}_X^u$, let $C_\disk=\vp^{-1}(C)$. Note that there are at most
finitely many crosscuts $C\in\mathcal{E}^u_X$ with $\dia(C)\ge \eta/30$. Let
$C^1,\dots,C^q$ be the collection of all crosscuts $Q(i,s(i))$ and all
crosscuts in $\mathcal{E}^u_X$ with diameters at least $\eta/30$. By definition~\ref{repout},
$f(\ol{C^j})\cap\ol{C^j}=\0$ for each $j$. Then the crosscuts
$C^j_\disk=\vp^{-1}(C^j)$ are all pairwise disjoint and have disjoint shadows.
Hence we may assume that, if the endpoints of $C^j_\disk$ are $\al_j,\be_j$,
then $\al_1<\be_1<\dots<\al_q<\be_q<\al_{q+1}=\al_1$ in the positive circular
order around $\uc$.

For each $i$, $1\le i\le q$, choose a finite chain of crosscuts
$F^i_j$ in $\disk_\iy$ with endpoints $\ga_j,\ga_{j+1}$ where
$\be_i=\ga_1<\ga_2<\dots<\ga_k=\al_{i+1}$ so that all closures of crosscuts from
the collection $\{C^1_\disk,\dots,C^q_\disk\}\cup  \{F^i_j\}_{i,j}$,
$1\le j\le k(i)$ (except for the adjacent crosscuts which share
endpoints) and their shadows are pairwise disjoint (this can be easily done,
e.g. because accessible points on the boundary of $X$ are dense),
$\vp(F^i_j)=G^i_j$ is a crosscut of $X$ and $\dia(G^i_j)<\eta/30$
for all $i,j$. In addition we may assume that non-adjacent $G^i_j$
have disjoint sets of endpoints. Let $Y=T(\bigcup C^i_\disk\cup
\bigcup F^i_j)$. Then $Y$ is a Jordan disk whose boundary is a
simple closed curve $\widehat S''\subset \ol{\disk_\iy}$. Let
$S''=\vp(\widehat S'')$. The set $S''\cap X$ is finite. It
partitions $S''$ into links which include all $C^i$'s. However, some
links of the form $G^i_j$ may be very close to a fixed point of $f$
and may not move off themselves under $f$. Hence we modify $S''$ as
follows.

\smallskip
\noindent \textbf{Claim B.} \emph{There exists a bumping simple closed curve $S$ such that:
\begin{enumerate}
\item[(1b)]  $
\cup_{j=1}^q  C^j
\subset S$.
\item[(2b)]  all components of $S\sm [X\cup\bigcup_i  D'_i\cup \bigcup C^i]$ are of diameter less than $\eta$,
\item[(3b)] for each $i$ components of $S\cap \Int(D'_i)$ are
so small that they stay far away from the fixed points and are moved
off themselves by $f$.
\end{enumerate}
}

Let $Z=T(S''\cup \bigcup D'_j)$. Then $Z$ is a Jordan disk whose
boundary is a simple closed curve $S'$, and all crosscuts $C^i$ are
still contained in $S'$. We modify $S'$, keeping all $C^i$'s but
changing $S'\sm \bigcup C^i$ so that the resulting bumping simple
closed curve $S$ can be partitioned into finitely many links each
of which does not go deep into the interior of $\bigcup D'_j$ and,
hence,  moves off itself under $f$.

Consider the crosscuts $\mathcal{E}^u_\disk$ in $\disk_\infty$. If
the chain $\{F^i_1,\dots,F^i_{k(i)}\}$ intersects a crosscut $Q\in
\mathcal{E}^u_\disk$ let $p_Q$ and $r_Q$ be the first and last point
of intersection of the arc $\cup_i F^i_j$ and $Q$. Then $p_Q\ne
r_Q$. If $\vp([p_Q,r_Q])$ is small then move forward along $S''$.
Otherwise suppose that the endpoints of $\vp(Q)$ are $a_Q$ and $b_Q$
and assume that $a_Q<b_Q$ in the positive order around $X$. Suppose
that $p_Q\in  F^i_j$ which has endpoints  $\gamma_j$ and
$\gamma_{j+1}$ and $r_Q\in F^i_k$ which has endpoints $\gamma_k$ and
$\gamma_{k+1}$ and $\gamma_{j +1}\le \gamma_k$. Replace the subarc
from $\gamma_j$ up to $\gamma_{k+1}$ in $S''$ by an arc joining the
same endpoints whose $\vp$-image is very close to $\vp(Q)$. Moving
forward along $S''$ in the positive direction and making finitely
often similar modifications, we obtain the desired simple closed
curve $S$. This completes the proof of Claim B. \hfill$\qed$

We want to compute the variation of $S$. Each link $Q(j,s(j))$ contributes at
least $1$ towards $\var(f, S)$, and we want to show that all other links have
non-negative variation. To do so we want to apply Lemma~\ref{posvar+}. Hence we
need to verify that all links of $S$ are bumping arcs whose endpoints map back
into $X$ such that their images are disjoint from themselves. By Claim B, all
links of $S$ move off themselves. However some links of $S$ may have endpoints
mapped off $X$. To ensure that for our bumping simple closed curve endpoints $e$ of
its links map back into $X$ we have to replace some of the finite chains of
links of $S$ by one link which is their concatenation (this is similar to what
was done in Theorem~\ref{fixpt}). Then we will have to check if the new
``bigger'' links still have images disjoint from themselves.

Denote by $A$ the initial partition of $S$ into links which are
called $A$-links.
\smallskip

\noindent \textbf{Claim C.} \emph{There exists a partition $A'$ of $S$ whose
links are bumping arcs with endpoints mapped back into $X$ and whose images are
disjoint from themselves. Moreover, $A'$-links are concatenations of $A$-links
of $S$ such that all $A$-links of $S$ of type $Q(i, t)$ remain $A'$-links of
$S$.}

\smallskip

\noindent \emph{Proof of Claim C.} Suppose that $X\cap
S=A=\{a_0<a_1<\dots<a_n\}$ and $a_0\in A$ is such that $f(a_0)\in X$
(by arguments similar to those in Theorem~\ref{fixpt} one can show
that we can make this assumption without loss of generality). We
only need to enlarge and concatenate links of $S$ with at least one
endpoint of the link mapped outside $X$. Therefore all links of $S$
of the form $Q(j, s)$ do not come under this category of links
because by Definition~\ref{repout} their endpoints do map into $X$.
Other links, however, may have endpoints mapped outside $X$.
Observe, that by the construction all those other links are less
than $\eta$ in diameter and hence have the property (15) of the
choice of the constants.

Let $t'$ be minimal such that $f(a_{t'})\not\in X$ and $t''>t'$ be  minimal
such that $f(a_{t''})\in X$. Then $f(a_{t'})\in Z_i$ for some $i$. Since every
component of $[a_{t'},a_{t''}]\sm X$ has  image of diameter less than $\da$
(which is less than the distance between any two sets $Z_r, Z_l$), $f(a_t)\in
Z_i\sm X$ for all $t'\le t<t''$. On the other hand, for $t'\le t<t''$, $a_t\not
\in U_i$. To see this, note that if $f(K_i)\cap Z_i=\0$, then by the above made
choices $f(U_i)\cap Z_i=\0$, and if $f(K_i)\cap Z_i\ne\0$, then $f(U_i\cap
X)\subset X$ by the assumption. Thus, all points $a_t, t'\le t<t''$ are in
$X\sm U_i$ while all their images $f(a_t)$ are in $Z_i\sm X$, which by the
property (\ref{eU}) of the constant $\da$ implies that these two finite sets of
points are at least $3\da$ distant.

Now, by the proven above, all the $A$-links of $S$ in the arc
$[a_{t'-1}, a_{t''}]$ are of diameters less than $\eta$. Hence it
follows from the properties (15) and (\ref{eU}) of the constant
$\da$ that $f([a_{t'-1},a_{t''}])\cap [a_{t'-1},a_{t''}]=\0$ and we
can remove the points $a_t$, for $t'\le t<t''$ from the partition
$A$ of $S$. By continuing in the same fashion we obtain a subset
$A'\subset A$ such that for the closure of each component $C$ of
$S\sm A'$, $f(C)\cap C=\0$ and for both endpoints $a$ and $a'$ of
$C$, $\{f(a),f(a')\}\subset X$. Moreover, as was observed above, the
enlarging of links of $S$ does not concern any links of the original
bumping simple closed curve of the form $Q(j, s)$, in particular for
each $j$, $Q(j,j(s))$ is an $A'$-link of $S$. \hfill \qed

Now we apply a version of the standard argument from the proof of
Theorem~\ref{fixpoint}; here, instead of Theorem~\ref{fpthm} we use the fact
that $f$ satisfies the argument principle. Indeed, by Theorem~\ref{I=V+1} and
Lemma~\ref{posvar+}, $\ind(f, S)\ge \sum_j \var(f, Q(j,j(s)))+1\ge m+1$
contradicting Theorem~\ref{argupr}.
\end{proof}

It is possible to use  a different approach  in Definition~\ref{repout} and
Theorem~\ref{locrot}. Namely, a version of Definition~\ref{repout} could define
repelling outside $X$ at a fixed point $p$ as the existence of a  family of
closed disks $D^j$ with similar properties \emph{except}  now we would require
the existence of at least $\ind(f, p)$ pairwise non-homotopic  rays outside $X$
from $\iy$ to $p$ (landing on $p$) and the existence of the same number of
components of $S^j\sm X$ non-disjoint from corresponding rays and with non-zero
variation on each such component.

Then a version of Theorem~\ref{locrot} would state that if a positively
oriented map with isolated fixed points $f$ repels outside $X$ at each of its
fixed points, and the condition (2h) of Theorem~\ref{locrot} is satisfied, then
$X$ must be a point. The proof of this version of Theorem~\ref{locrot} is
almost  the same, except for a bit heavier notation needed (now that we have
not one, but $\ind(f, p)$ crosscuts with positive variation around each fixed
point in $X$. Since for our applications Theorem~\ref{locrot} suffices we
restricted ourselves to this case.

Theorem~\ref{locrot} implies the following.

\begin{cor}\label{degenerate}
Suppose that $f$ is a  positively oriented map with isolated fixed points, and
$X\subset\C$ is a non-separating and non-degenerate continuum satisfying
condition (2h) 
stated in Theorem~\ref{locrot} and such that all fixed points in
$X$ belong to $\bd X$. Then either $f$ does not repel outside $X$
at one of its fixed points, or the local index at one of its fixed
points is not equal to $1$.
\end{cor}

Lemma~\ref{repel} gives a verifiable sufficient condition for a
fixed point $a$ belonging to a locally invariant continuum $X$ to be such that
the map $f$ repels outside $X$. We will  apply the lemma in the next section.

\begin{lem}\label{repel}
Suppose that $f$ is a positively oriented map, $X\subset\C$ is a non-separating continuum or a point
and $p\in \bd X$ is a fixed point of $f$ such that:

\begin{enumerate}

\item[(i)]
there exists a neighborhood $U$ of $p$ such that $f|_{\ol{U}}$ is
one-to-one and $f(\ol{U}\cap X)\subset  X$,

\item[(ii)]
there exists a ray $R\subset\sphere\sm X$ from infinity such that
$\ol{R}=R\cup \{p\}$, $f|_R:R\to R$ is a homeomorphism and for each $x\in R$,
$f(x)$ separates $x$ from $\infty$ in $R$,

\item[(iii)]
there exists a nested sequence of closed disks $D_j\subset U$ with
boundaries $S_j$ containing $p$ in their interiors such that $\bigcap
D_j=\{p\}$ and $f(\ol{S_j\sm X})\cap D_j=\0$.

\end{enumerate}

Then for a sufficiently large $j$ there exists a component $C$ of $S_j\sm X$
with $C\cap R\ne\0$ and $\var(f, C)>0$, so that $f$ repels outside $X$ at $p$.
\end{lem}

Observe that here we show that $\var(f, C)>0$ without any ``scrambling''
assumptions on $f$.

\begin{proof}
Choose a Jordan disk $U$ as in (i) so that $(\bd U)\cap R=\{q\}$ is
a point and $X\sm \ol{U}\ne \0$. Choose $j$ so that $D_j\cup
f(D_j)\subset U$. By \cite{bo06} there is a component $C$ of $(\bd
D_j)\sm X$ such that $R$ crosses $C$ essentially (see
Definition~\ref{essential}). Slightly adjusting $D_j$, we may
assume that $R\cap \bd D_j$ is finite and each intersection is
transversal. Since $R$ crosses $C$ essentially, $|R\cap C|$ is odd;
since $f$ is one-to-one on $\ol{U}$, $|f(C)\cap R|$ is odd as well.

Let $u, v$ be the endpoints of $C$. Observe, that $C$ can be
included in a simple closed curve $S$ around $X$ so that $X\subset
T(S)$. Since by (i) $f(\ol{U}\cap X)\subset  X$, we see that
$f(u)\in T(S), f(v)\in T(S)$ and variation $\var(f, C)>0$ is
well-defined (see Definition~\ref{vararc}).

\begin{figure}
\begin{picture}(307,231)
\put(75,0){\scalebox{0.3}{\includegraphics{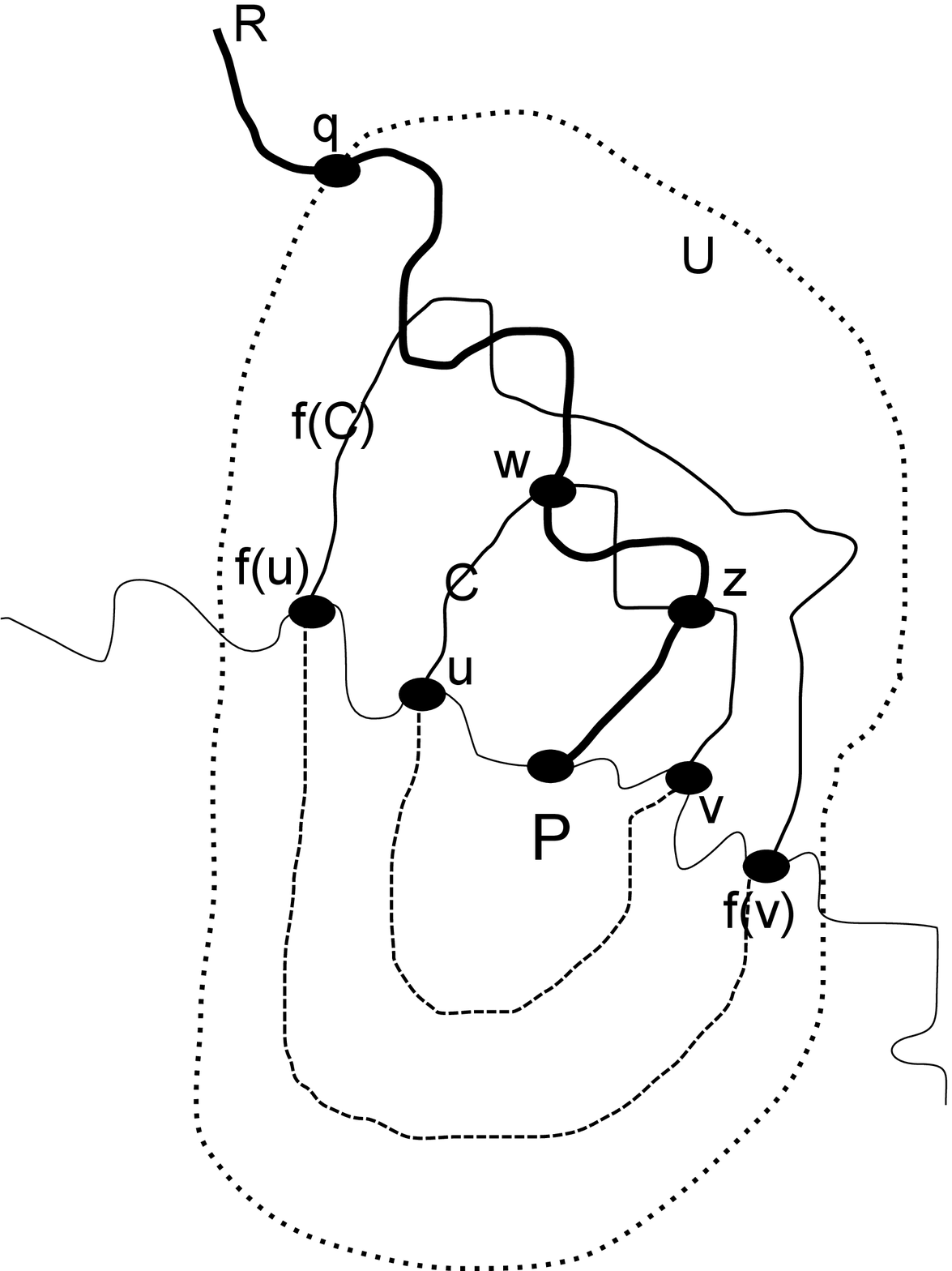}}}
\end{picture}
\caption{Illustration to the proof of Lemma~\ref{repel}.}
\end{figure}

Let us move along $R$ from infinity to $p$ and denote by $w$ be the
first point of $C\cap R$ which we meet and by $z$ the last point.
Then by (ii), $f(C)\cap [p, z]_R=\0$. Also, by (i) $|f(C)\cap
R|=|f(C\cap R)|$ is odd. Since $X\sm D_j\ne \0$ and $X\cap ([z,
w]_C\cup [z, w]_R)=\0$, $X$ is contained in the unbounded component
$V^\iy$ of the complement to $[z, w]_C\cup [z, w]_R$. Since $u$ and
$v$ belong to $U\cap X$, their images $f(u)$ and $f(v)$ belong to
$X$ as well. Hence $f(u), f(v)\in V^\iy$.

\smallskip

\noindent \textbf{Claim A.} \emph{$|f(C)\cap [z, w]_R|$ is even.}

\smallskip

\noindent \emph{Proof of Claim A}.\, A complementary domain $O$ of
$[z, w]_C\cup [z, w]_R$ is called \emph{even/odd} if there is an
arc $J$ from infinity to a point in $O$ so that $J\cap C\cap R=\0$,
$J\cap ([z, w]_C\cup [z, w]_R)$ is finite, every intersection is
transversal and $|J\cap ([z, w]_C\cup [z, w]_R)|$ is even/odd,
respectively. By \cite{ot82} the notion of an even/odd domain is
independent of $J$, well-defined and each complementary domain of
$[z, w]_C\cup [z, w]_R$ is either even or odd. Since by (iii)
$f(\ol{C})\cap \ol{C}=\0, f(\ol{C})$ can only intersect $C\cup R$
at points of $R$. Also, whenever $f(\ol{C})$ meets $[z, w]_R$, it
crosses from an even to an odd domain or vice versa. Since both
$f(u)$ and $f(v)$ are in the unbounded (and hence even) domain of
$\C\sm [f(C)\cup [z,w]_R$], $|f(C)\cap [z, w]_R|$ is even as
desired. \qed

Observe that $f(\ol{C})$ is outside $D_j$ and hence is disjoint from
$[p, z]_R\subset D_j$. Since $|f(C)\cap R|$ is odd and $|f(C)\cap
[p, w]_R|=|f(C)\cap [z, w]_R|$ is even, $|f(C)\cap [w, \iy]_R|$ is
odd. Since every intersection is transversal, we can replace $[w,
\iy]_R$ by a junction $J_w$ such that every intersection point of
$f(C)\cap [w, \iy]_R$ contributes exactly $+1$ or $-1$ to the count.
Hence ${\mathrm{var}}(f, C)\ne 0$.

We prove next that ${\mathrm{var}}(f, C)>0$. Note that $[p, q]_R\subset \ol{U}$
is an arc which meets $\bd U$ only at $q$. Let $W$ be the component of $U\sm
(C\cup R)$ whose boundary contains $(\bd U)\sm q$. Note that $u$ and $v$ are
accessible points in $\bd W$. Hence points $u$ and $v$ can be connected with
an arc $K$ inside $W$ disjoint from $R\cup C$. Then, since $f|_U$ is a
homeomorphism, 
$f(K)\cap R=\0$. By Lemma~\ref{posvar} this implies that
${\mathrm{var}}(f, C)\ge 0$ and by the previous paragraph then
${\mathrm{var}}(f, C)>0$ (basically, we simply choose a junction
$J'_w$ close to $[w, \iy]_R$ such that $f(K)\cap J'_w=\0$ and
conclude that ${\mathrm{var}}(f, C)=\win(f, C\cup K, w)>0$ since $f$
is a positively oriented map).
\end{proof}

It is now easy to see that the following corollary holds.

\begin{cor}\label{locrot1}
Suppose that $X\subset\C$ is a non-separating continuum or a point
and $f:\C\to\C$ is a positively oriented map with isolated fixed
points, and the following conditions hold.

\begin{enumerate}

\item[(a)]
Each fixed point $p\in X$ is topologically repelling, belongs to
$\bd X$, and has a neighborhood $U_p$ such that $f(U_p\cap X)\subset
X$ and $f|_{U_p}$ is a homeomorphism.

\item[(b)] For each fixed point $p\in X$, there exists a ray
    $R\subset\sphere\sm X$ from infinity landing on $p$,
    $f|_R:R\to R$ is a homeomorphism and for each $x\in R$,
    $f(x)$ separates $x$ from $\infty$ in $R$.

\item[(c)]
The map $f$ scrambles the boundary of $X$. Moreover for every $i$
either $f(K_i)\cap Z_i=\0$ or there exists a neighborhood $U_i$ of
$K_i$ with $f(U_i\cap X)\subset X$.

\end{enumerate}

Then $X$ is a (fixed) point.
\end{cor}

\begin{proof}
Let us apply Theorem~\ref{locrot}. To do so, we verify its
conditions. The facts that $X\subset\C$ is a non-separating
continuum or a point and $f:\C\to\C$ is a positively oriented map
with isolated fixed points are clearly satisfied. To verify
condition (1h) of Theorem~\ref{locrot}, suppose that $p\in X$ is a
fixed point. Then by (a) above $p\in \bd X$. Moreover, $p$ is
topologically repelling, and so by Lemma~\ref{ind1} the index at
$p$ is $+1$.

It remains to verify that $f$ repels outside $X$ at $p$. To do so we
apply Lemma~\ref{repel}.  Since $p$ is topologically repelling,
there exists a nested sequence of closed disks $D_j\subset U$ with
boundaries $S_j$ containing $p$ in their interiors, with $\bigcap
D_j=\{p\}$ and $f(\ol{S^j\sm X})\cap D^j=\0$. Hence the condition
(iii) of Lemma~\ref{repel} is satisfied. The condition (i) of
Lemma~\ref{repel} immediately follows from (a) above; the condition
(ii) of Lemma~\ref{repel} immediately follows from (b) above. Hence
by Lemma~\ref{repel} the map $f$ repels outside $X$ at $p$.
Therefore the condition (1h) of Theorem~\ref{locrot} is satisfied.
Condition (2h) of Theorem~\ref{locrot} is also satisfied (it simply
coincides with condition (c) of our corollary), hence by
Theorem~\ref{locrot} $X$ is a point.
\end{proof}

\section{Applications to complex dynamics}\label{complappl}

We begin by introducing a few facts concerning local dynamics at parabolic and
repelling periodic points of a polynomial which were not necessary for stating
the results of this section in Chapter~\ref{descr2} but are needed for the
proofs. A nice description of this can be found in \cite{miln00} (\cite{carlgame93}
can also serve as a good source here).

Let $P:\C\to\C$ be a complex polynomial, $J_P$ its Julia set ($J_P$
is the closure of the set of repelling periodic points of $P$) and
$K_P=T(J_P)$ the ``filled-in'' Julia set. Recall That $\si_d: \ucirc
\to \ucirc$ is defined by $\si_d(\al)=d\al\mod 1$, where
$\ucirc=\reals/\zed$ is parameterized by $[0,1)$. (This map is
conjugate to the map $z\to z^d$ restricted to the unit circle in the
complex plane.) If $p$ is a periodic point of $P$ of period $n$ and
$(P^n)'(p)=re^{2\pi i \al}$ with $r\ge 0$, then $p$ is
\emph{repelling} if $r>1$, \emph{parabolic} if $r=1$ and
$\al\in\mathbb{Q}$, \emph{irrational neutral} if $r=1$ and
$\al\in\R\sm\mathbb{Q}$ and \emph{attracting} if $r<1$. If $p$ is a
repelling or parabolic fixed point in a non-degenerate component
$Y$ of $K_P$, then by \cite{douahubb85, lepr} there exist $1\le k<\infty$ 
external rays $R_{\al(i)}$ such that $\si_d|_{\{\al(1),\dots,\al(k)\}}:
\{\al(1),\dots,\al(k)\}\to \{\al(1),\dots,\al(k)\}$ is a permutation, 
all $\al(i)$ are of the same minimal period under $\si_d$, for
each $j$ the ray $R_{\al(j)}$ lands on $p$, and no other external rays land on $p$.

Components of $\C\sm J_P$ are called \emph{Fatou domains}. There are
three types of bounded Fatou domains $U$. A Fatou domain $U$ is
called an \emph{attracting domain} if it contains an attracting
periodic point, a \emph{Siegel domain} if it contains an irrational
neutral periodic point and a \emph{parabolic domain} if it is
periodic but contains no periodic points. In the latter case there
always exists a parabolic periodic point on the boundary of the
parabolic Fatou domain. An irrational neutral periodic point inside
a Siegel domain is called a \emph{Siegel (periodic) point}; an
irrational neutral periodic point in $J_P$ is called a \emph{Cremer
(periodic) point}.

Any two distinct parabolic Fatou domains which contain the same
parabolic periodic point in their boundaries are separated by two
external rays which land at this parabolic point. It is also known
that points inside these parabolic domains are attracted by the orbit
of this parabolic periodic point while points on the external rays
landing at points of this orbit are repelled to infinity. Suppose
that $p$ is a parabolic fixed point, $P'(p)=e^{2\pi i \frac{r}{q}}$,
$r,q\in\mathbb{Z}$, $R_0,\dots, R_{m-1}$ are all external rays
landing at $p$ and $U_0,\dots, U_{k-1}$ are all Fatou domains which
contain $p$ in their boundaries. Moreover, suppose that both rays and
domains are numbered according to the positive circular order around
$p$. Then combinatorially one can think of the local action of $P$ on
rays and domains at $p$ as a rotation by $r/q$. This means that
$P(U_j)=U_{(j+r)\mod k}$ and all rays between $U_i$ and $U_{i+1}$ are
mapped by $P$ in an order preserving way onto all rays between
$U_{(i+r)\mod k}$ and $U_{(i+1+r)\mod k}$).

Before we continue we want to recall the notion of a general
puzzle-piece which is first introduced in Definition~\ref{genpuz}.

\setcounter{chapter}{5}\setcounter{section}{5} \setcounter{thm}{0}

\begin{defn}[General puzzle-piece]
Let $P:\C\to\C$ be a polynomial. Let $X\subset K_P$ be  a
non-separating subcontinuum or a point such that the following holds.


\begin{enumerate}

\item 
There exists $m\ge 0$ and $m$
pairwise disjoint non-separating continua/points $E_1\subset X, \dots,
E_m\subset X$.
\item There exist $m$ finite sets of external rays $A_1=\{R_{a^1_1}, \dots, R_{a^1_{i_1}}\},
\dots, A_m=\{R_{a^m_1}, \dots, R_{a^m_{i_m}}\}$ 
with $i_k\ge 2, 1\le k\le m$.

\item 
We have $\Pi(A_j)\subset E_j$ (so the set $E_j\cup
(\cup^{i_j}_{k=1} R_{a^j_k})=E'_j$ is closed and connected).

\item 
$X$ intersects a unique component $C_X$ of $\C\sm \cup E'_j$
\index{CX@$C_X$}.

\item For each Fatou domain $U$ either $U\cap X=\0$ or $U\subset X$.

\end{enumerate}

We call such $X$ with the continua $E_i$ and the external rays
$R_{\al^k_i}$ a \emph{general puzzle-piece} and call the continua
$E_i$  \emph{exit continua} of $X$. For each $k$, the set $E'_k$
divides the plane into $i_k$ open sets which we will call
\emph{wedges (at $E_k$)}; denote by $W_k$ the wedge which contains
$X\sm E_k$ (it is well-defined by (4) above).
\end{defn}

\setcounter{chapter}{7}\setcounter{section}{5} \setcounter{thm}{0}


\begin{figure}

\begin{center}

\includegraphics[width=4truein, height=4truein]{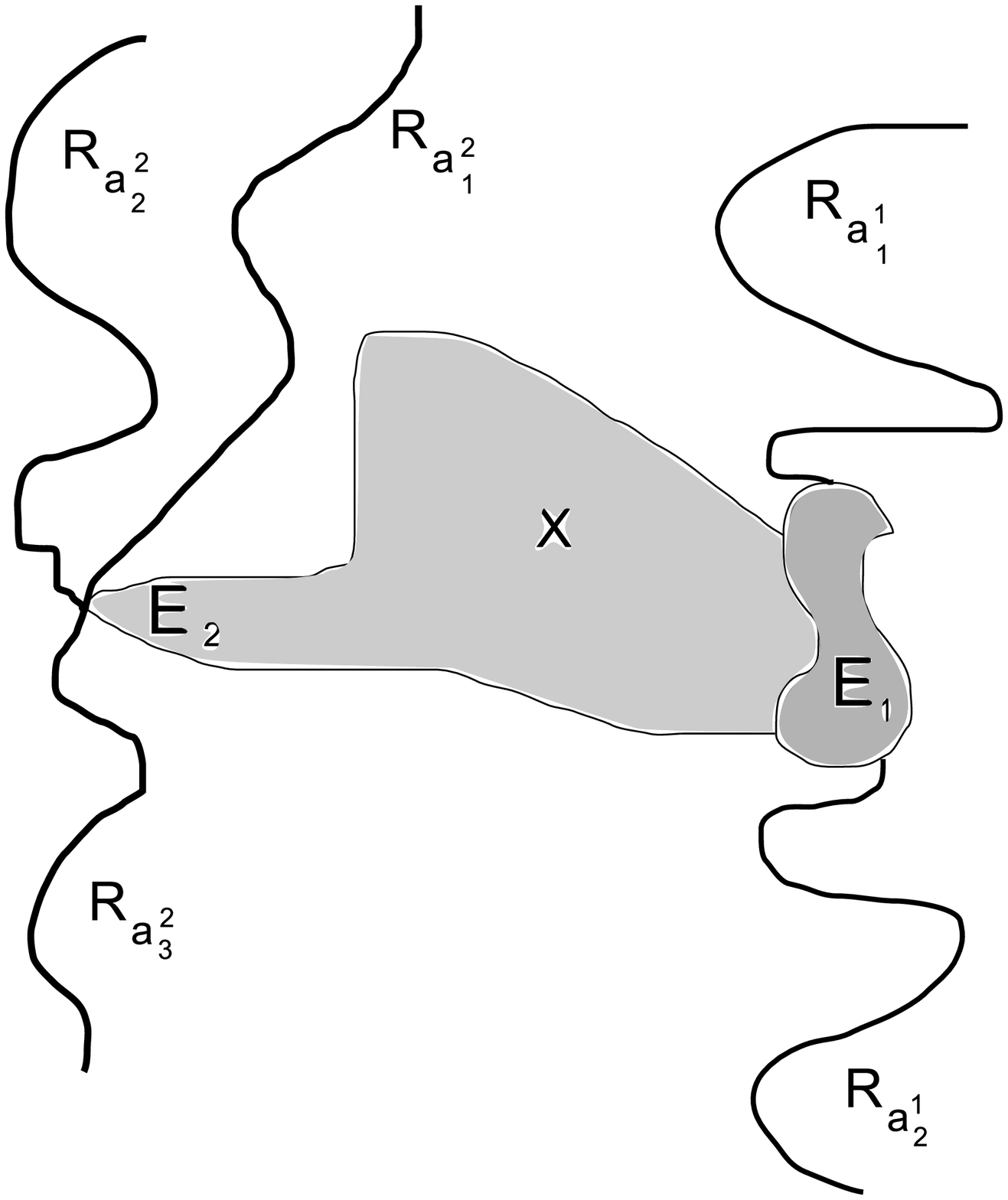}



\caption{A general puzzle-piece} \label{puzzle-piece}

\end{center}

\end{figure}

Let us now see whether condition (1) of Theorem~\ref{locrot} applies
to the polynomial $P$ with fixed parabolic points in a general
puzzle-piece $X$ such that $P(X)\cap C_X\subset X$ (loosely, it
means that $P(X)$ does not ``grow'' at points of $C_X$). We need to
check that at a fixed parabolic point $p\in X$ the map $P$ repels
outside $X$ and the local index is $1$.

For convenience, consider $p\in C_X$. Let $L$ be a cycle of
parabolic domains containing $p$ in their boundaries. To make a more
complete picture, we first observe that either $L\subset X$ or all
domains in $L$ are disjoint from $X$. Indeed, suppose that one of
the domains in $L$ is contained in $X$. Then, as there is no
``growth'' of $P(X)$ at $p$, we see that all domains of $L$ are in
$X$. Otherwise by condition (5) from Definition~\ref{genpuz} all
domains in $L$ are disjoint from $X$.

Suppose first that $P'(p)\ne 1$. There must exist small disks $D$
such that components of $(\bd D)\sm X$ map outside $D$. Let $Q$ be a
crosscut which is a component  of $(\bd D)\sm X$. Choose an external
ray $R$ landing at $p$ and crossing $Q$ essentially. By the analysis
of dynamics around $p$ (the fact that $P$ locally ``rotates'', see,
e.g., \cite{miln00}) and since $P(X)\cap C_X\subset X$, we conclude
that $P(Q)\cap R=\0$. However, then $\var(f, Q)=0$, a contradiction
with the existence of a crosscut of positive variation in $(\bd
D)\sm X$.

On the other hand, if $P'(p)=1$ then, as was explained right after
the proof of Lemma~\ref{ind1}, $\ind(P, p)>1$ (again, see, e.g.,
\cite{miln00}) which contradicts the condition (1) of
Theorem~\ref{locrot} as well. Hence in the parabolic case
Theorem~\ref{locrot} cannot be applied ``as is'' to the polynomial
$P$. Moreover, since at parabolic points the map is not
topologically repelling, neither is Corollary~\ref{locrot1}
applicable in this case.

The idea allowing us to still deal with parabolic points is that we
can change $P$ inside the parabolic domains in question without
compromising the rest of the arguments and modifying these parabolic
points to topologically repelling periodic points. The thus
constructed new map $g$ will no longer be holomorphic but will
satisfy the conditions of Corollary~\ref{locrot1}. We formalize this
idea in the following lemma.

\begin{lem}\label{pararepel}
Suppose that $\{p_j\}, 1\le j\le m,$ are all parabolic fixed points of a
polynomial $P$ with $P'(p_j)=1$. Then there exists a positively oriented map
$g_P=g$ which coincides with $P$ outside the invariant parabolic Fatou domains,
is locally one-to-one at each $p_j$ (hence $p_j$ is not a critical point of $g$)
and is such that all the points $p_j$ are topologically repelling
fixed points of $g$. In particular, $\ind(g, p_j)=+1$ for all $j,\
1\le j\le m$.
\end{lem}

\begin{proof}
Let us consider a fixed parabolic point $p=p_j$. Let $F_i$ be the
invariant Fatou domains containing $p$ in their boundaries $B_i$. By
a nice recent result of Yin and Roesch \cite{roesyin08}, the boundary
$B_i$ of each $F_i$ is a simple closed curve and $P|_{B_i}$ is
conjugate to the map $z\to z^{d(i)}$ for some integer $d(i)\ge 2$.
Let $\psi:F_i\to\disk$ be a conformal isomorphism. Since $B_i$
is a simple closed curve,  $\psi$ extends to a homeomorphism on
$\ol{\disk}$. Since $f|_{B_i}$ is conjugate to the map $z\to
z^{d(i)}$, it now follows that the map $P|_{\ol{F_i}}$ can be
replaced by a map topologically conjugate by $\psi$ to the map
$g_i(z)=z^{d(i)}$ on the closed unit disk $\ol{\disk}$ which agrees
with $f$ on $B_i$. Let $g$ be the map defined by $g(z)=P(z)$ for each
$z\in\C\sm\bigcup F_i$ and $g(z)=g_i(z)$ when $z\in F_i$. Then $g$ is
clearly a positively oriented map.

Since by the analysis of the dynamics around parabolic points \cite{miln00} $P$
repels points away from $p$\, outside parabolic domains $F_i$, we conclude, by
the construction, that $p$ is a topologically repelling fixed point of $g$.
Clearly, $\deg_g(p)=1$, hence by Lemma~\ref{ind1} $\ind(g, p)=\deg_g(p)=1$ as
desired. Continuing in this fashion, we can change $P$ on all invariant
parabolic domains with fixed points $p_j$, $1\le j\le m$, in their boundaries to
a map $g$ which satisfies the requirements of the lemma.
\end{proof}

We use Lemma~\ref{pararepel} in the proof of the
Theorem~\ref{pointdyn}. Recall, that a fixed point $x$ of a
polynomial $P$ is said to be \emph{non-rotational} if there is a fixed external ray
landing at $x$ (it follows that each such point is either repelling
or parabolic).

\begin{thm}\label{pointdyn}
Let $P$ be a polynomial with filled-in Julia set $K_P$ and let $Y$
be a non-degenerate periodic component of $K_P$ such that $P^p(Y)=Y$.
Suppose that $X\subset Y$ is a non-degenerate general puzzle-piece
with $m\ge 0$ exit continua $E_1, \dots, E_m$ such that $P^p(X)\cap
C_X\subset X$ and either $P^p(E_i)\subset W_i$, or $E_i$ is a
$P^p$-fixed point.
Then at least one of the following claims holds:

\begin{enumerate}

\item $X$ contains a  $P^p$-invariant parabolic domain,

\item $X$ contains a $P^p$-fixed point  which is neither repelling nor
    parabolic, or

\item $X$ has an external ray $R$ landing at a repelling or parabolic
    $P^p$-fixed point such that $P^p(R)\cap R=\0$ (i.e., $P^p$ locally
    rotates at some parabolic or repelling $P^p$-fixed point).

\end{enumerate}

Equivalently, suppose that $Y$ is a non-degenerate periodic
component of $K_P$ such that $P^p(Y)=Y$, $X\subset Y$ is a general
puzzle-piece with $m\ge 0$ exit continua $E_1, \dots, E_m$ such that
$P^p(X)\cap C_X\subset X$ and either $P^p(E_i)\subset W_i$, or $E_i$
is a $P^p$-fixed point; if, moreover, $X$ contains only
non-rotational $P^p$-fixed points and does not contain
$P^p$-invariant parabolic domains, then it is degenerate.
\end{thm}

\begin{proof}
We may assume that $p=1$ and $P(Y)=Y$. We show that if none of the
conclusions (1)-(3) hold, then Corollary~\ref{locrot1} applies and
therefore $X$ must be a point, contradicting the assumption that
$X$ is non-degenerate. However, if $X$ contains parabolic points,
we first use Lemma~\ref{pararepel} and replace $P$ by the
positively oriented map $g$ constructed in that lemma. If
conclusion (1) does not hold, then $X$ contains no invariant
parabolic domains, and so $P|_{X}=g|_{X}$. Let us now check
conditions of Corollary~\ref{locrot1}.

First we check condition (a) of Corollary~\ref{locrot1}. Indeed,
consider a fixed point $y\in X$. Then $y$ is a repelling or
parabolic fixed point of $P$ (this is because we assume that claim
(2) of Theorem~\ref{pointdyn} does not hold and hence all
$P^p$-fixed points in $X$ are repelling or parabolic). By
Lemma~\ref{pararepel} this implies that $y$ is a topologically
repelling fixed point of $g$. Let us show that there exists a small
neighborhood $U$ of $y$ such that $g(U\cap X)\subset X$. This is
clear if $y\in C_X$ because $g(X)\cap C_X\subset X$ by our
assumptions. Assume now that $y\nin C_X$ which means that
$\{y\}=E_k$ is one of the exit-continua of $X$. Observe that there
a few \emph{fixed} external rays of $P$ landing at $y$ (the rays
are fixed because we assume that the conclusion (3) of
Theorem~\ref{pointdyn} does not hold and hence all rays which land
at $y$ must be fixed). Choose the two rays $R_1, R_2$ which land at
$y$ and form the boundary of the wedge $W_k$ at $y$ which contains
$X$. Since $g(X)\cap C_X\subset X$ by our assumptions, this implies
that in a small neighborhood $U_k$ of $E_k$ the intersection
$U_k\cap X$ maps (by $g$ or $P$) into $X$ as desired. This
completes the verification of condition (a) of Corollary~\ref{locrot1}.

Condition (b) of Corollary~\ref{locrot1} (i.e., the existence of a
fixed external ray landing at each fixed point in $X$) follows
immediately from the assumption that claim (3) of
Theorem~\ref{pointdyn} above does not hold.

Let us now check condition (c) of Corollary~\ref{locrot1}. Set
$g(X)\sm X=P(X)\sm X=H$. We may assume that $H\ne \0$ and we can
think of $g(X)=P(X)$ as a continuum which ``grows'' out of $X$. In
particular, $m\ge 1$. Fix $k$, $1\le k\le m$. Since $g(X)\cap
C_X\subset X$, any component of $H$ whose closure intersects $E_k$
must be contained in one of the wedges at $E_k$ (such wedges are
defined in Definition~\ref{locrot1}), but not in $W_k$. Let $Z_k$
be the topological hull of the union of all components of $H$ which
meet $E_k$ together with $E_k$. Then $Z_k$ is a non-separating
continuum. Since either $g(E_i)\subset W_i$ or $E_i$ is a fixed
point, the map $g$ scrambles the boundary of $X$ (see
Definition~\ref{scracon}). Moreover, if $E_k$ is mapped into $W_k$
then clearly $g(E_k)\cap Z_k=\0$ (because $Z_k\sm E_k$ is contained
in the other wedges at $E_k$ and is disjoint from $W_k$). On the
other hand, consider a fixed point $y\in X$ such that $E_k=\{y\}$.
Then condition (c) of Corollary~\ref{locrot1} follows from (a)
which has already been verified. Hence, by Corollary~\ref{locrot1}
we conclude, that $X$ is a point, a contradiction.

\end{proof}

Notice that if $X$ is a general puzzle-piece with $m=0$ in
Theorem~\ref{pointdyn}, then $C_X=\C$. Hence in this case $P(X)\cap
C_X\subset X$ implies $P(X)\subset X$ and $X$ is invariant. Thus, a
non-separating invariant continuum $X\subset K_P$ is a general
puzzle-piece if and only if for every Fatou domain $U$ of $P$ either
$U\cap X=\0$, or $U\subset X$.

The proof of the next corollary is left to the reader.

\begin{cor}\label{deg-gen}
Suppose that $Y$ is a non-degenerate periodic component of $K_P$
with $P^p(Y)=Y$, $X\subset Y$ is a general puzzle-piece which is
either invariant or has $m>0$ exit continua $E_1, \dots, E_m$ such
that $X\cap C_X=K_P\cap C_X$, and either $P^{pn}(E_i)\subset W_i$
for all $n>0$, or $E_i$ is a periodic point. If, moreover, $X$
contains no periodic parabolic domains, no attracting, Cremer, or
Siegel periodic points, and at most finitely many periodic points
with more than one external ray landing at them, then $X$ is a
point.
\end{cor}

Theorem~\ref{pointdyn} is useful in proving the degeneracy of
certain impressions and establishing local connectedness of the
Julia set at some points (see \cite{bco08}). Recall that the
\emph{impression} $\imp(\al)$ of an angle in the connected case can
be defined as the intersection of the closures of all shadows
$\ol{\sh(C)}$ of all crosscuts $C$ such that $R_\al$ crosses $C$
essentially. Corollary~\ref{degimpr} is proved for a somewhat larger
class of subcontinua of $J_P$ which includes impressions as an
important particular case.

Consider a repelling or parabolic periodic or preperiodic point $x$
and all external rays landing at $x$. Then the union of two such
rays and $x$ is said to be a \emph{legal cut} of the plane. Also,
suppose that $x, y\in \bd U$ are two periodic or preperiodic points
in the boundary of an attracting or parabolic Fatou domain $U$. By
\cite{roesyin08} there exists an arc $A\subset U$ connecting $x$ and
$y$. The union of $A$ and two external rays landing at $x$ and $y$
is also a \emph{legal cut} of the plane. Finally, call a continuum
$Q$ \emph{periodic} if for some $n>0$ we have $P^n(Q)\subset Q$.

\begin{cor}\label{degimpr}

Let $P:\C\to\C$ be a complex polynomial and $Q\subset J_P$ be a
periodic continuum such that for every legal cut $C$ the set $Q\sm
C$ is contained in one component of $\C\sm C$. Suppose that $T(Q)$
contains no Siegel or Cremer points. Then $Q$ is degenerate. In
particular, if $J_P$ is connected and $Q$ is a periodic impression
such that $T(Q)$ contains no Cremer or Siegel points, then $Q$ is a
point.
\end{cor}

\begin{proof}
By considering an appropriate power of $P$ we may assume that $Q$ is
invariant and non-degenerate. Clearly this implies that $T(Q)$ is
invariant too. Suppose that $p'\in Q$ is a fixed point of $P$ and
$R_\ba$ is an external ray landing at $p'$. Then $P(R_\ba)$ also
lands on $p'$. If $R_\be$ is not fixed, then $C=R_\be\cup P(R_\be)$
is a legal cut. The local dynamics at $p'$ and the fact that $Q$ is
invariant imply now that $Q$ has points on either side of $C$, a
contradiction with the assumptions on $Q$. Hence each fixed
repelling or parabolic point in $Q$ is non-rotational.

Let us show that $Q$ can only intersect the closure of a parabolic
or attracting Fatou domain $U$ at one point. Indeed, suppose
otherwise and let $x, y\in \bd U\cap Q, x\ne y$. Then there exists
an arc $I\subset \bd U$ with endpoints $x, y$, contained in $Q$
because otherwise $Q$ will ``shield'' some points of $\bd U$ from
infinity contradicting the fact that all points of $\bd U$ belong to
the closure of the basin of attraction of infinity. By
\cite{roesyin08} we can find, say, periodic points $u, v\in I$ and
include them in a legal cut $T$ which will separate some points of
$I$ (and hence of $Q$) from other points of $I$, contradiction with
our assumptions. Hence $T(Q)$ cannot contain an attracting or
parabolic domain $U$ since otherwise, by the above, $Q$ must shield
part of $\partial U$ from the basin of attraction of infinity, a
contradiction. This implies that $T(Q)$ cannot contain parabolic
domains or attracting points.   By the assumption $T(Q)$ does not
contain Cremer or Siegel points either. Hence by
Corollary~\ref{deg-gen} $Q$ is a point as desired.
\end{proof}

In the particular case in the end of the statement we assume that
$J_P$ is connected; the same result in fact holds for \emph{all}
Julia sets but will require the introduction of the notion of the
impression for disconnected Julia sets which we   avoid here for the
sake of simplicity. The verification of the fact that impressions
satisfy the conditions of the corollary is straightforward and
therefore is left to the reader. In particular, suppose $R_\al$ is a
periodic external ray and the topological hull $T(\imp(\al))$ of the
impression of $\al$ contains only repelling or parabolic periodic
points. Then, by Corollary~\ref{degimpr}, $\imp(\al)$ is degenerate.

Note that the assumptions of Corollary~\ref{degimpr} are equivalent
to the following. Suppose that $Q\subset J_P$ is a periodic
continuum such that for every legal cut $C$ the set $Q\sm C$ is
contained in one component of $\C\sm C$. As in the proof of
Corollary~\ref{degimpr}, this implies that $Q$ can only intersect
the boundaries of attracting or parabolic domains at no more than
one point. To make the conclusion of the corollary, we need to check
that $T(Q)$ contains no Siegel or Cremer points. We claim that this
is equivalent to the following:

\begin{enumerate}

\item $Q$ contains no Cremer point;

\item if the boundary of
a Siegel disk is decomposable, then $Q$ is disjoint from it;

\item if the boundary of a Siegel disk is indecomposable
(it is not known if such Siegel disks exist), then $Q$ intersects it
in at most one point.

\end{enumerate}

Indeed if (1) - (3) above are satisfied then, by an argument similar
to the proof of Corollary~\ref{degimpr},   $T(Q)$ contains no Cremer
or Siegel points. Now, suppose that $T(Q)$ contains no Cremer or
Siegel points. Then by Corollary~\ref{degimpr} $Q=\{q\}$ is a point.
Hence (1) and (3) hold trivially. If $B$ is the decomposable
boundary of a Siegel disk and $q\in B$, then we may assume that $B$
and $Q$ are invariant. It is known \cite{rog92a,rog92b} that there
exists a monotone map $p:B\to\ucirc$ and an induced map
$g:\ucirc\to\ucirc$ which is an irrational rotation. Since $Q$ is
invariant, $Q=B$, a contradiction. Hence (2) holds as well.


\backmatter

-----------------------------------------------------------------------------
%

\bibliographystyle{amsalpha}




\begin{theindex}

  \item accessible point, 29
  \item acyclic, 30
  \item allowable partition, 28

  \indexspace

  \item $\B$, 35
  \item $\B^\infty$, 35
  \item boundary scrambling
    \subitem for dendrites, 50
    \subitem for planar continua, 51
  \item branched covering map, 68
  \item branchpoint of $f$, 75
  \item bumping
    \subitem arc, 25
    \subitem simple closed curve, 25

  \indexspace

  \item $C(a,b)$, 38
  \item Carath\'eodory Loop, 27
  \item chain of crosscuts, 28
    \subitem equivalent, 28
  \item channel, 29
    \subitem dense, 29, 57
  \item completing a bumping arc, 25
  \item $\complex$, 1
  \item $\rsphere$, 1
  \item conformal
    \subitem external ray, 53
  \item continuum
    \subitem decomposable, 49
    \subitem indecomposable, 49
  \item $\ECH(K)$, 35
  \item $\HCH(B\cap K)$, 35
  \item convex hull
    \subitem Euclidean, 35
    \subitem hyperbolic, 35
  \item counterclockwise order
    \subitem on an arc in a simple closed curve, 13
  \item critical point, 74
  \item crosscut, 25
    \subitem shadow, 26
  \item cutpoint, 51
  \item $C_X$, 53, 85

  \indexspace

  \item $\Disk$, 38
  \item defines variation near $X$, 50
  \item degree, 13
  \item $\deg_f(a)$, 75
  \item $\dg(g)$, 13
  \item $\dg(f_p)$, 16
  \item $\partial$ boundary operator, 2
  \item $\disk^\infty$, 27
  \item dendrite, 6
  \item domain
    \subitem attracting, 52
    \subitem Fatou, 52
    \subitem parabolic, 52
    \subitem Siegel, 52

  \indexspace

  \item embedding
    \subitem orientation preserving, 19
  \item essential crossing, 29
  \item $\mc{E}_t$, 29
  \item exit continuum, 51
    \subitem for a general puzzle-piece, 53
  \item external ray, 29
    \subitem end of, 29
    \subitem essential crossing, 29
    \subitem landing point, 29
    \subitem non smooth, 53
    \subitem smooth, 53

  \indexspace

  \item fixed point
    \subitem for positively oriented maps, 63
    \subitem non-rotational, 53
  \item $(f,X,\eta)$, 50

  \indexspace

  \item $\rG$, 38
  \item $\fg$, 38
  \item $\rg$, 35
  \item $\Gamma$, 38
  \item gap, 37, 68
  \item general puzzle-piece, 53
  \item geodesic
    \subitem hyperbolic, 38
  \item geometric outchannel
    \subitem negative, 50
    \subitem positive, 50

  \indexspace

  \item hull
    \subitem hyperbolic, 37
    \subitem topological, 1
  \item hyperbolic
    \subitem geodesic, 35
    \subitem halfplane, 35
  \item hyperbolic chord, 38

  \indexspace

  \item $id$ identity map, 13
  \item $\im(\mc{E}_t)$, 29
  \item impression, 29
  \item index, 13
    \subitem fractional, 13
    \subitem I=V+1 for Carath\'eodory Loops, 28
    \subitem Index=Variation+1 Theorem, 20
    \subitem local, 52
  \item $\ind(f,x)$, 52
  \item $\ind(f,A)$, 19
  \item $\ind(f,g)$, 13
  \item $\ind(f,g|_{[a,b]})$, 13
  \item $\ind(f,S)$, 19

  \indexspace

  \item $J_P$, 52
  \item Julia set, 52
  \item junction, 14
  \item J{\o}rgensen Lemma, 39

  \indexspace

  \item $K_P$, 52
  \item $\KP$, 37
  \item $\KP_\delta$, 41
  \item $\kp$-chord, 37
  \item $\KP^{\pm}_\delta$, 43
  \item $\KPP$, 37
  \item $\KPP_\delta$, 41
  \item Kulkarni-Pinkall
    \subitem Lemma, 36
    \subitem Partition, 37

  \indexspace

  \item lamination, 68
    \subitem degenerate, 68
    \subitem invariant, 68
  \item landing point, 29
  \item leaf, 68
  \item link, 25
  \item local degree, 75
  \item local index, 52
  \item Lollipop Lemma, 23

  \indexspace

  \item map, 13
    \subitem confluent, 16
    \subitem light, 16
    \subitem monotone, 16
    \subitem negatively oriented, 16
    \subitem on circle of prime ends, 32
    \subitem oriented, 16
    \subitem perfect, 16
    \subitem positively oriented, 16
  \item maximal ball, 35
  \item Maximum Modulus Theorem, 31
  \item monotone-light decomposition of a map, 16

  \indexspace

  \item narrow strip, 59
  \item natural retraction of dendrites, 64
  \item non-separating, xi

  \indexspace

  \item order
    \subitem on subarc of simple closed curve, 13
  \item orientation preserving
    \subitem embedding, 19
  \item outchannel, 50
    \subitem geometric, 50
    \subitem uniqueness, 59

  \indexspace

  \item periodic point
    \subitem Cremer, 53
    \subitem parabolic, 52
    \subitem repelling, 52
    \subitem Siegel, 52
    \subitem weakly repelling, 51
  \item point of period two
    \subitem for oriented maps, 64
  \item positively oriented arc, 77
  \item $\pr(\mc{E}_t)$, 29
  \item prime end, 28
    \subitem channel, 29
    \subitem impression, 29
    \subitem principal continuum, 29
  \item principal continuum, 29

  \indexspace

  \item $\real$, 1
  \item repels outside $X$ at $p$, 76
  \item $R_t$, 29

  \indexspace

  \item shadow, 26
  \item $\Sh(A)$, 26
  \item smallest ball, 35
  \item standing hypothesis, 50

  \indexspace

  \item topological
    \subitem Julia set, 68
    \subitem polynomial, 68
  \item topological hull, 1
  \item topologically
    \subitem attracting, 75
    \subitem repelling, 75
  \item $T(X)$, 1
  \item $T(X)_\delta$, 43

  \indexspace

  \item $U^\infty$, 1
  \item unlinked, 68

  \indexspace

  \item $\val_Y(x)$, 51
  \item valence, 51
  \item $\var(f,A)$, 27
  \item variation
    \subitem for crosscuts, 25
    \subitem of a simple closed curve, 15
    \subitem on an arc, 14
    \subitem on finite union of arcs, 15
  \item $\var(f,A,S)$, 14
  \item $\var(f,S)$, 15

  \indexspace

  \item weakly repelling, 51, 65
  \item wedge (at an exit continuum), 53
  \item $\win(g,\uc,w)$, 13

\end{theindex}

\end{document}